\documentclass[a4paper,11pt]{article}
\usepackage[utf8]{inputenc}
\usepackage{latexsym}
\usepackage{amsfonts}
\usepackage{amssymb}

\usepackage{enumitem} %http://ctan.org/pkg/enumitem
\usepackage{braket}
\usepackage{hyperref}

\usepackage{amsmath}
\usepackage{amssymb}
\usepackage{amsfonts}
\usepackage{amsthm}
\usepackage{indentfirst}
\usepackage{psfrag}
\usepackage{amscd,stmaryrd,latexsym}
\usepackage[margin=1.3in]{geometry}
\usepackage[pdftex]{graphicx}
\usepackage{xcolor}

\newtheorem{thm}{Theorem}[section]
\newtheorem{lem}{Lemma}[section]

\newtheorem{cor}{Corollary}[section]

\newtheorem{Prop}{Proposition}[section]
\newtheorem*{St*}{Statement}

\newtheorem*{theorem*}{Theorem}
\theoremstyle{remark}

\theoremstyle{definition}
\newtheorem{df}{Definition}[section]

\theoremstyle{remark}
\newtheorem{oss}{Remark}[section]

\newcommand{\be}{\begin{equation}}
\newcommand{\ee}{\end{equation}}
\newcommand{\R}{\mathbb{R}}
\newcommand{\N}{\mathbb{N}}

\newcommand{\G}{\Gamma}

\newcommand{\Greg}{\text{gen-reg}\,}
\newcommand{\Sing}{\text{sing}\,}
\newcommand{\SingT}{\text{sing}_T\,}
\newcommand{\Reg}{\text{reg}\,}
\newcommand{\spt}[1]{\text{spt}\,\|#1\|}

\newcommand\res{\mathop{\hbox{\vrule height 7pt width .5pt depth 0pt
\vrule height .5pt width 6pt depth 0pt}}\nolimits}

\def\eps{\mathop{\varepsilon}}

\def\Oc{\mathop{\mathcal{O}}}

\def\Hc{\mathop{\mathcal{H}}}

\def\m{\mu}

\def\s{\sigma}
\def\r{\rho}
\def\Om{\Omega}
\def\om{\omega}

\def\p{\partial}

\def\eps{\mathop{\varepsilon}}

\def\Om{\Omega}
\def\om{\omega}
\def\p{\partial}

\DeclareMathAlphabet{\mathscr}{OT1}{pzc}{m}{it} 

\newenvironment{thmbis}[1]
  {\renewcommand{\thethm}{\ref{#1}$'$}%
   \addtocounter{thm}{-1}%
   \begin{thm}}
  {\end{thm}}

\begin{document} 

\title{\textbf{Stable prescribed-mean-curvature integral varifolds of codimension $1$: regularity and compactness}}
\author{Costante Bellettini \\ University College London \and Neshan Wickramasekera\\ University of Cambridge}
\date{}

\maketitle

\begin{abstract}
In \cite{BW} we developed a regularity and compactness theory in Euclidean ambient spaces for codimension 1 weakly stable CMC integral varifolds satisfying two (necessary) structural conditions. Here we generalize this theory to the setting where  the mean curvature (of the regular part of the varifold) is prescribed by a given $C^{1, \alpha}$ ambient function $g$ and the ambient space is a general $(n+1)$-dimensional Riemannian manifold; we give  general conditions that imply that the support of the varifold, away from a possible singular set of dimension $\leq n-7,$ is the image of a $C^{2}$ immersion with continuous unit normal $\nu$ and mean curvature $g\nu$. These conditions also identify a compact class of varifolds subject to an additional uniform mass bound.  

If $g$ does not vanish anywhere, or more generally if the set $\{g = 0\}$ is sufficiently small (e.g.\ if ${\mathcal H}^{n} \, (\{g =0\}) = 0$), then the results (Theorems~\ref{thm:mainreg_gpos}, \ref{thm:compactness_gpos}, \ref{thm:mainreg_emptyinterior} and \ref{thm:compactness-smallnodalset} below) and their proofs are quite analogous to the CMC  case. When $g$ is arbitrary however, a number of additional considerations must be taken into account: first, the correct second variation assumption is that of finite Morse index (rather than weak stability); secondly, to have a geometrically useful theory with compactness conclusions, one of the two structural conditions on the varifold must be weakened; thirdly, in view of easy examples, this weakening of a structural hypothesis  necessitates additional hypotheses in order to reach the same conclusion. We identify two sets of additional assumptions under which this conclusion holds for general $g$: one in Theorems~\ref{thm:mainregstationarityimmersedparts}, \ref{thm:generalcompactness} (giving regularity and compactness) and the other in Theorems~\ref{thm:mainreg_g0_version1}, \ref{thm:non-orientable} (giving regularity only). We also provide corollaries for multiplicity $1$ varifolds associated to reduced boundaries of Caccioppoli sets. In these cases, some or all of the structural assumptions become redundant.

Finally,  as a direct corollary of the methods used in the proofs of the above theorems, we obtain an abstract varifold regularity theorem (Theorem~\ref{thm:abstract_sheeting_thm}) which does not require the varifold to be a critical point of a functional. This theorem plays a key role in the analysis of varifolds arising in certain phase transition problems, which in turn forms the basis of a PDE theoretic proof of the existence of prescribed mean curvature hypersurfaces in compact Riemannian manifolds (\cite{BW2}). 
\end{abstract}

\tableofcontents

\section{Introduction} 

In \cite{BW} the authors developed a sharp regularity and compactness theory in open subsets of the $(n+1)$-dimensional Euclidean space $(n \geq 2$) for a general class of integral $n$-varifolds having (generalised) mean curvature locally summable to an exponent $p>n$ and satisfying certain structural and variational hypotheses.  The structural conditions---of which there are two---are local conditions that only concern parts of the varifold that a priori have a ``$C^{1, \alpha}$ structure''; specifically, they require that:
\begin{itemize}
\item [(i)] the varifolds have no \emph{classical singularities} (see Definition~\ref{df:classicalsingularity} below) and
\item[(ii)] the \emph{coincidence set near every touching singularity} (Definition~\ref{df:touchingsingularity} below) is ${\mathcal H}^{n}$-null. 
\end{itemize}
The variational hypotheses require, roughly speaking,  that the orientable regular parts of the varifold (non-empty by Allard theory, but may a priori be small in measure) are stationary and stable with respect to the area functional for volume-preserving deformations. This stationarity condition is well-known to lead to a CMC condition (constant scalar mean curvature with respect to a continuous choice of unit normal) on the regular orientable parts. In fact it is equivalent to stationarity, for some constant $\lambda$,  with respect to the functional $$J = A + \lambda \, {\rm Vol}$$ for unconstrained compactly supported deformations, and moreover, an oriented immersion is stationary with respect to $J$ if and only if, with respect to some choice of unit normal, it has scalar mean curvature equal to $\lambda.$  Here $A,$ ${\rm Vol}$ denote the area functional and the enclosed volume functional respectively. The main regularity result of \cite{BW} states that a varifold satisfying the above structural and variational hypotheses  is  supported on the image of a smooth, proper CMC immersion away from a possible ``genuine'' singular set $\Sigma$ of codimension at least $7,$ and that its support away from $\Sigma$ is in fact quasi-embedded,  meaning that it may fail to be embedded only at points where (locally) the support consists of precisely two smooth embedded CMC disks intersecting tangentially and each lying on one side of the other. The associated compactness theorem of \cite{BW} says that any family of such varifolds that additionally satisfies locally uniform mass bounds and a uniform mean curvature bound is compact in the topology of varifold convergence. 

In the present paper we generalise the theory of \cite{BW} to the setting where the mean curvature is prescribed by a $C^{1, \alpha}$ function $g$ on the ambient space. This condition on the mean curvature has a variational formulation:  An orientable immersion has scalar mean curvature, with respect to a choice of orientation, equal to $g$ everywhere if and only if it is stationary with respect to the functional $$J_{g} = A + {\rm Vol}_{g},$$ where ${\rm Vol}_{g}$ is 
the \emph{relative enclosed $g$-volume} (see Definition~\ref{Dfi:rel_encl_gvolpos} below). Thus we here study codimension 1 integral $n$-varifolds with generalised mean curvature locally in $L^{p}$ for some $p >n$  and satisfying two structural conditions ((i) above and (ii) appropriately weakened in case $\{g = 0\}) \neq \emptyset$) whose orientable regular parts are stationary and stable in an appropriate sense with respect to $J_{g}.$ 
We shall, at the same time, treat here the case of a general $(n+1)$-dimensional Riemannian ambient space $N$. We achieve the latter generalisation by considering (cf.\ \cite{SS}, \cite{WicAnnals}), locally near any given point $X_{0} \in N$,  the pull back, via the exponential map, of the functional $J_{g}$  to the tangent space $T_{X_{0}} \, N$ (identified with ${\mathbb R}^{n+1}$). Since the pull back of ${\rm Vol}_{g}$  (on hypersurfaces in $N$) is equal to a functional of the same type, namely, to
${\rm Vol}_{\sqrt{|\mathscr{h}|} \, g \circ \exp_{X_{0}}}$ (on hypersurfaces in a ball in ${\mathbb R}^{n+1}$ taken with the Euclidean metric) where $\mathscr{h}$ is the Riemannian matric on $N$, it is convenient and efficient to handle, as we do here, both generalisations simultaneously.  

If $g$ is non-zero everywhere, these generalisations are straightforward, and the regularity and compactness theorems (Theorems~\ref{thm:mainreg_gpos} and \ref{thm:compactness_gpos} below) and their proofs are indeed completely analogous to those in the CMC case treated in \cite{BW};   slightly more generally, these results carry over to the case 
${\mathcal H}^{n}( \{g = 0\}) = 0$ just as easily (Theorems~\ref{thm:mainreg_emptyinterior} and \ref{thm:compactness-smallnodalset}) with one additional mild assumption (hypothesis (b$^{T}$) of Theorems~\ref{thm:mainreg_emptyinterior} and \ref{thm:compactness-smallnodalset}). When $g$ is an arbitrary $C^{1,\alpha}$ function on $N$ however, there are a number of subtleties that enter both the statements of the theorems and their proofs. For that reason, we shall discuss the results for arbitrary $g$  (Theorems~\ref{thm:mainregstationarityimmersedparts}, \ref{thm:generalcompactness} and ~\ref{thm:mainreg_g0_version1}) separately.   

The methods employed in the proofs of the above results lead very directly to a certain abstract regularity theorem (Theorem~\ref{thm:abstract_sheeting_thm}) for codimension 1 integral varifolds $V$ with no classical singularities in a Euclidean space. This theorem can be regarded as a higher-multiplicity version of the codimension 1 Allard regularity theorem. Like the Allard regularity theorem, it does not require $V$ to be a stationary point of a functional, but instead assumes a certain condition on the first variation of $V$ with respect to the area functional. When $V$ is the pull-back under the exponential map of (a portion of) a varifold $\widetilde{V}$ on an $(n+1)$-dimemsional Riemannian manifold $N$, this first variation condition is implied by summability to a power $p> n$ of the generalised mean curvature of $\widetilde{V}$ in $N$. In place of the second variation hypotheses imposed in the results proved elsewhere in the present work, this theorem makes a hypothesis of a topological nature. Because of the ``non-variational'' nature of the hypotheses, this theorem is of fundamental importance in our work \cite{BW2} that establishes regularity of limit varifolds  associated with sequences of Morse index bounded solutions to inhomogeneous Allen--Cahn equations. The study of these limit varifolds (without any second variation hypotheses) was initiated in \cite{RogTon}, \cite{Ton} and \cite{HutchTon}. These works left open the possibility that higher (even) multiplicity,  zero-mean-curvature portions may appear in the limit and that these pieces may merge with the phase boundary on a large set in an irregular fashion. As shown in \cite{BW2}, these extraneous minimal portions can be smoothly removed subject to a Morse index bound on the Allen--Cahn solutions, and Theorem~\ref{thm:abstract_sheeting_thm} is a key ingredient in that analysis.

The present article should be taken in conjunction with \cite{BW} as well as with \cite{WicAnnals}.  The latter treated the case $g=0$ of our results here for prescribed mean curvature varifolds. It established regularity and compactness theorems, similar to those here, for  area-stationary codimension 1 integral varifolds with no classical singularities and stable regular parts; in the case $g=0$, the maximum principle implies that there are no touching singularities, so the structural condition (ii) is automatically satisfied.  As regards work that concerns embeddedness criteria for mean curvature controlled hypersurfaces, the present work may be viewed a natural completion of a theory, following \cite{WicAnnals} and \cite{BW},  that unifies and extends  three earlier well-known theories  dating back to the period from early 1960's to early 1980's: the regularity theory for locally area minimizing hypersurfaces (due to the combined work of De~Giorgi (\cite{DG2}), Federer (\cite{Fed70}) and Simons (\cite{Simons})), the compactness theory for locally uniformly area bounded stable hypersurfaces with small singular sets (due to Schoen--Simon--Yau (\cite{SSY}) in low dimensions and 
Schoen--Simon (\cite{SS}) in general dimensions) and the regularity theory for boundaries of Caccioppoli sets minimizing area subject to fixed enclosed volume (due to Gonzalez--Massari--Tamanini (\cite{GonzMassTaman80}, \cite{GonzMassTaman})). The present work relies on the techniques developed in the first two of these earlier works as well as on those developed in \cite{WicAnnals} and \cite{BW}.

\medskip

\noindent \textbf{Acknowledgments}: C. B.\ is partially supported by the EPSRC grant EP/S005641/1. This work was completed while both authors were members of the Institute for Advanced Study, Princeton. The authors gratefully acknowledge the excellent research environment and the support provided by the Institute, and the support by the National Science Foundation under Grant No. DMS-1638352.

\medskip

\subsection{Relative enclosed $g$-volume and other definitions} 

Let $N$ be a smooth Riemannian manifold of dimension $n+1 \geq 3$ and let $g \, : \, N \to {\mathbb R}$ be a given function of class $C^{1, \alpha}$. In order to characterize variationally the condition that the mean curvature of an immersed hypersurface in $N$ is prescribed by $g$, we will need the following notion, which extends the definition of enclosed volume (\cite[Section 2]{BarbDoCEsch}): 

\begin{df} [\textit{relative enclosed $g$-volume}] 
\label{Dfi:rel_encl_gvolpos}
 Let $\iota:S \to N$ be a $C^1$ immersion of an $n$-dimensional manifold $S$ into $N$. Let $\Oc \subset \subset N$ be an orientable open set such that $S_{\Oc} = \iota^{-1}(\Oc)$ is orientable and has compact closure in $S$. When these conditions are met we always assume that the orientations of $S_{\Oc}$, $\Oc$ and the unit normal $\nu$ to $S_{\Oc}$ are chosen so that the frame $d \iota(\vec{S}_{\Oc}) \wedge \nu$ is positive with respect to the orientation of $\Oc$. For $\epsilon>0$, let $\psi \, : \, (-\epsilon, \epsilon) \times S \to N$ be a $C^{1}$ map such that $\psi_t( \cdot )=\psi(t, \cdot) \, : \, S \to N$ is an 
 immersion for each $t \in (-\epsilon,\epsilon)$, $\psi_0=\iota$, $\psi_t(x)=\iota(x)$ for all $(t, x) \in(-\epsilon, \epsilon) \times S \setminus S_{\Oc}$ and 
 $\psi_t(x)\in \Oc$ for all $(t, x)  \in(-\epsilon, \epsilon) \times  S_{\Oc})$. 
We define  \emph{relative enclosed $g$-volume} of $\psi_t(S)$ in the oriented manifold $\Oc$, denoted $\text{Vol}_g(t),$ by 
 
$$\text{Vol}_g(t) := \int_{[0,t] \times S_{\Oc}} \psi^*(g\, \chi_{\Oc} d\text{vol}_{N}),$$
where $\chi_{\Oc} d\text{vol}_{N}$ is the volume form on $\Oc \subset N$ induced by the Riemannian metric.
\end{df}

An important class of one-parameter family of immersions is given by pushforwards by ambient deformations induced by a tangential vector field on $N$, compactly supported in $\Oc$. When the hypersurface is embedded and agrees with the $C^1$ boundary of a Caccioppoli set $E$, and when $\Oc$ is bounded and $\psi_t$ is an ambient deformation, it can be checked, writing $E_t=\psi_t(E)$, that $\text{Vol}_g(t)=\int_{\Oc \cap E_t} g\,d\mathcal{H}^{n+1} - \int_{\Oc \cap E} g\,d\mathcal{H}^{n+1}$ (this justifies the terminology).

\medskip

With notation as in Definition \ref{Dfi:rel_encl_gvolpos}, denote by $A(t)$ the area of the immersed hypersurface $\psi_t(S)$ in $\Oc$, so 
$$A(t) = \int_{S \cap \iota^{-1}(\Oc)} dS_t,$$ 
where $dS_t$ is the $n$-volume induced on $S$ by $\psi_t$ and by the Riemannian metric on $N$; equivalently, denoting by $\psi_{t \, \#} |S|$ the integral varifold 
$V = (\psi_{t}(S), \theta_{\psi_{t}})$ where $\theta_{\psi_{t}}(\psi_{t}(x)) = \#\{y \in S \, : \, \psi_{t}(y) = \psi_{t}(x)\}$, we have that 
$$\|\psi_{t \, \#} |S|\|(\Oc) = \int \theta_{\psi_{t}}\, d{\Hc}^n \res (\psi_t(S) \cap \Oc) =A(t).$$ 
Let $$J_g(t)=A(t) + \text{Vol}_g(t).$$

\begin{df}[\textit{Stationarity}] \label{stationary-def}
Let $\iota \, : \, S \to N$ be as in Definition~\ref{Dfi:rel_encl_gvolpos} above. We say that $\iota$ is stationary with respect to $J_{g}$ if for any orientable open subset $\Oc \subset\subset N$  such that $S_{\Oc} = \iota^{-1}(\Oc)$ is orientable and has compact closure in $S$, there exist orientations on $\Oc$ and $S_{\Oc}$ such that for any deformation $\psi$ as in Definition~\ref{Dfi:rel_encl_gvolpos}, we have that 
$$J_{g}^{\prime}(0) = 0.$$
\end{df}

Note that  the condition that a two-sided immersion $\iota \, : \, S \to N$, with unit normal $\nu,$ has mean curvature $g \nu$ is equivalent  to requiring a stationarity  
condition on $\iota$ with respect to $J_{g}$ (see Section~\ref{stationaritydiscussion}).

\begin{oss}\label{vol-preserving-def}
This variational condition implies in particular stationarity with respect to the area functional for deformations that preserve $\text{Vol}_g$, i.e.\ deformations $\psi(t, \cdot)$ for which $\text{Vol}_{g}(t)$ is constant. The converse of this  is ``essentially true'' in the case $g>0$, on which we will focus in Section \ref{g_pos}, but not for arbitrary $g$, as we will discuss in later sections (see Remark \ref{oss:vol_preserv_notexist}).
\end{oss}

Throughout this paper,  we will use terminology (essentially) as in \cite[Definitions 2.1, 2.2, 2.3, 2.4, 2.5]{BW}; the only difference is that when $g$ is of class $C^{1, \alpha}$, the definitions of $\Reg{V}$ and $\Greg{V}$ here will require $C^2$-regularity instead of smoothness as in \cite[Definitions 2.1, 2.5]{BW}. This is because standard elliptic theory implies that the regularity to expect (for our varifolds) is $C^2$ (in fact $C^{2,\alpha}$) rather than $C^\infty$. Higher regularity can then be deduced from elliptic theory whenever higher regularity of $g$ is assumed. 

The precise terminology we shall use is as follows:

Let $V$ be an integral varifold of dimension $n$ on $N$,  and let $\|V\|$ denote the weight measure associated with $V$. 

\begin{df}[\textit{Regular set $\Reg \, {V}$ and singular set $\text{sing} \, V$}]
A point $X \in N$ is a regular point of $V$ if $X \in \spt{V}$ and if there exists $\sigma > 0$ such that $\spt{V} \cap B_{\sigma}(X)$ is an embedded $C^{2}$ hypersurface of $B_{\sigma}(X)$. The regular set of $V$, denoted ${\rm reg} \, V,$ is the set of all regular points of 
$V.$ The (interior) singular set of $V$, denoted ${\rm sing} \, V$, is $(\spt{V}\setminus \Reg{V}) \cap N$. By definition, ${\rm reg} \, V$ is relatively open in ${\rm spt} \, \|V\|$ and ${\rm sing} \, V$ is relatively closed in ${\mathcal N}$.  
\end{df}

\begin{df}[\textit{$C^1$-regular set $\text{reg}_1 V$}]
\label{df:C1embedded}
We define $\text{reg}_1 V$ to be the set of points $X \in \spt{V}$ with the property that there is $\sigma > 0$ such that $\spt{V} \cap B_{\sigma}(X)$ is an embedded hypersurface of $B_{\sigma}(X)$ of class $C^{1}$.
\end{df}
 
\begin{df}[\textit{Set of classical singularities $\text{sing}_C \, V$}]
\label{df:classicalsingularity}
A point $X \in \text{sing} \, V$ is a classical singularity of $V$ if there exists $\sigma >0$ such that, for some $\alpha \in (0, 1]$,  $\spt{V} \cap B_{\sigma}(X)$ is the union of three or more embedded $C^{1,\alpha}$ hypersurfaces-with-boundary meeting  pairwise only along their common $C^{1,\alpha}$ boundary $\gamma$ containing $X$ and such that at least one pair of the hypersurfaces-with-boundary meet transversely everywhere along $\gamma$.  

The set of all classical singularities of $V$ will be denoted by $\text{sing}_C \, V$.
\end{df}

\begin{df}[\textit{Set of touching singularities $\text{sing}_T \, V$ and the coincidence set}]
\label{df:touchingsingularity}
A point $X \in \text{sing}\,  V \setminus \text{reg}_1 V$ is a touching singularity of $V$ if $X \notin \text{sing}_C \, V$ and if there exists $\sigma>0$ such that 
$$\spt{V} \cap {\mathcal N}_\sigma(X) = M_{1} \cup M_{2}$$ 
where $M_{j}$ is an embedded $C^{1,\alpha}$-hypersurfaces  of ${\mathcal N}_{\sigma}(X)$ with $({\overline M_{j}} \setminus M_{j}) \cap {\mathcal N}_{\sigma}(X) = \emptyset$ for $j=1,2.$ 

The set of all touching singularities of $V$ will be denoted by $\text{sing}_T \, V$.

If $X \in  {\rm sing}_{T} \, V$ and $M_{1}, M_{2}, \sigma$ are as above, then for any $\sigma_{1} \in (0, \sigma]$, the \emph{coincidence set of $V$ in $B_{\sigma_{1}}(X)$} is the set 
$M_{1} \cap M_{2} \cap B_{\sigma_{1}}(X)$.  
\end{df}

\begin{oss}[\textit{Graph structure around a point $X\in \text{sing}_T \, V$}]
\label{oss:touchingsinggraphs}
If $X \in \text{sing}_T V$ then each of the two $C^{1,\alpha}$-hypersurfaces $M_{1}, M_{2}$ corresponding to $X$ (as in Definition~\ref{df:touchingsingularity}) contains $X$ and they are tangential to each other at $X;$ the former is implied by the fact that $X \in \text{sing} \, V \setminus \text{reg}_1 \, V$ and the latter by the fact that $X \notin \text{sing}_C \, V$. Let $L$ be the common tangent plane to $M_{1}, M_{2}$ at $X$. Identifying $T_{X} \, N$ with ${\mathbb R}^{n+1}$ so that 
$L$ is identified with ${\mathbb R}^{n} \times \{0\}$, we then have that for $\sigma$ sufficiently small, 
$$\exp_{X} ^{-1} \, \spt{V} \cap B_\sigma^{n+1}(0) = (\text{graph}\, u_1 \cup \text{graph}\, u_2) \cap B_\sigma^{n+1}(0)$$ for two functions $$u_{1}, u_{2} \, : \, B_\sigma^{n}(0)  \to {\mathbb R}$$ of class $C^{1,\alpha}$ such that $u_1(0)=u_2(0)$ and $Du_1(0)=Du_2(0)=0$. Note that $u_{1} \neq u_{2}$ since $X \in \text{sing} V\setminus \text{reg}_1 V$. Note also that we can always choose 
$u_{1}$, $u_{2}$ such that $u_{1} \leq u_{2}$.
\end{oss}

\begin{df}[\textit{Generalized regular set $\Greg{V}$}]
\label{df:regularpoints}
A point $X \in \text{spt}\,\|V\|$ is a generalized regular point if either (i) $X \in \Reg V$ or (ii) $X \in \text{sing}_T V$ and we may choose $C^{2}$ functions $u_1$ and $u_2$ corresponding to $X$ as in Remark~\ref{oss:touchingsinggraphs} such that $u_{1} \geq u_{2}$. The set of generalised regular points will be denoted by $\Greg{V}$. 
\end{df}

\subsection{Regularity and compactness for $g>0$}
\label{g_pos}

\begin{df}[\textit{Weak stability}]
Let $\iota \, : \, S \to N$ be an oriented proper $C^{2}$ immersion as in Definition~\ref{Dfi:rel_encl_gvolpos}, and suppose that $\iota$ is stationary with respect to $J_{g}$. We say that $\iota$ is \emph{weakly stable}\footnote{This terminology is customary in the literature at least when $g$ is a constant: we use it also for any $g>0$.} 
if $$J_{g}^{\prime\prime}(0) \geq 0$$
for all deformations $\psi$ as in Definition~\ref{Dfi:rel_encl_gvolpos} satisfying additionally that $\text{Vol}_{g}(\cdot)$ is constant.
\end{df} 

\begin{oss}[\textit{Stationarity and stability for $\text{Vol}_g$-preserving deformations}]
\label{oss:vol_preserv_pos}
For $g>0$, stationarity of an immersion with respect to the area functional $A(t)$ for $\text{Vol}_g$-preserving deformations is equivalent to the fact that the mean curvature of the immersion is given by $\lambda g$ for some constant $\lambda \in \R$ (see Section \ref{stationaritydiscussion}) which in turn is equivalent to stationarity with respect to $J_{\lambda g}$ for arbitrary deformations. As regards  second variation, however, stability with respect to $J_{\lambda g}$ for arbitrary deformations is in general a strictly more restrictive requirement than weak stability (the former notion is in fact referred to as strong stability).
\end{oss}

We begin by stating some natural generalizations of the results in \cite{BW} for the case in which we have strictly positive ambient functions $g$; these are special instances of the general theorems (Theorems~\ref{thm:mainregstationarityimmersedparts} and \ref{thm:generalcompactness} below) proved in this work.

\begin{thm}[\textbf{regularity: special case $g>0$}]
\label{thm:mainreg_gpos}
For $n \geq 2$ let $N$ be a Riemannian manifold of dimension $n+1$ and $g:N \to \R$, $g>0$ be a $C^{1,\alpha}$ function. Let $V$ be an integral $n$-varifold in $N$ such that
\begin{enumerate}
\item[(a1)] the first variation of $V$ is in $L^{p}_{\rm loc} \, (\|V\|)$ for some $p >n$;
\item[(a2)] ${\rm sing}_{C} \, V  = \emptyset$ (definition~\ref{df:classicalsingularity}); 
\item[(a3)] every touching singularity has a neighborhood in which the coincidence set of $V$ has zero ${\mathcal H}^{n}$ measure (definition~\ref{df:touchingsingularity});
\end{enumerate}
Suppose moreover that the following variational assumptions are satisfied: 
\begin{enumerate}
\item[(b)] the (embedded) $C^1$ hypersurface $S={\rm reg}_{1} \, V$ is stationary with respect to $J_{g}$ in the sense of Definition~\ref{stationary-def} 
taken with $\iota \, : \, S \to N$ equal to the inclusion map; 
\item[(c)] for each orientable open set ${\Oc} \subset N \setminus ({\rm spt} \, \|V\| \setminus {\Greg} \, V)$ such that $\Greg{V} \cap \Oc$ is the image of an orientable proper $C^2$ immersion 
$\iota_{\Oc} \, : \, S_{\Oc} \to \Oc$ that is stationary with respect to $J_{g}$, the immersion $\iota_{\Oc}$  is weakly stable. 
\end{enumerate}
Then there is a closed set 
 $\Sigma \subset {\rm spt} \, \|V\|$ with $\Sigma = \emptyset$ if $n \leq 6$, $\Sigma$ discrete if $n=7$ and ${\rm dim}_{\mathcal H} \, (\Sigma) \leq n-7$ if $n \geq 8$ such that: 
 \begin{enumerate}
\item [(i)] locally near each point  $p \in {\rm spt} \, \|V\| \setminus \Sigma$, either ${\rm spt} \, \|V\|$ is a single $C^2$ embedded disk or ${\rm spt} \, \|V\|$ is precisely two $C^2$ embedded disks with only tangential intersection; moreover the set where two such disks intersect is locally contained in a submanifold of dimension $n-1$;
\item[(ii)] ${\rm spt} \, \|V\| \setminus \Sigma$ ($=\Greg{V}$) is the image of a proper $C^2$ immersion $\iota \, : \, S_{V} \to  N$ and there is a continuous choice of unit normal $\nu$ to $\iota$ such that the mean curvature $H_{V}(x)$ of $S_{V}$ at any $x \in S_{V}$ is given by $H_{V}(x) =  g(\iota(x)) \nu(x)$.
 \end{enumerate}
\end{thm}

\begin{oss}
 \label{oss:immersionGreg_gpos}
Note that it follows from hypothesis (a1) and Allard's regularity theory that ${\rm reg}_{1} \, V$ is non-empty, and thus the stationarity hypothesis (b) is never vacuously true; moreover, 
since $g \in C^{1, \alpha}$, it follows from standard elliptic theory that ${\rm reg}_{1} \, V$ is of class $C^{2,\alpha}$. Also, we shall show  
(see Remark~\ref{oss:orient_gpos}) that whenever $\Oc \subset N$ is orientable,  $\Greg{V} \cap \Oc$ \emph{is} the image of an orientable proper $C^2$ immersion 
$\iota_{\Oc} \, : \, S_{\Oc} \to \Oc$ that is stationary with respect to $J_{g}$, so that hypothesis (c) indeed contains information beyond weak stability of ${\rm reg}_{1} \, V$.  
\end{oss}

\begin{oss}[\textit{A more geometric (smaller) class of varifolds}]
\label{oss:add(m)}
 As pointed out in the Euclidean CMC framework (see \cite[Remarks 2.16, 2.17]{BW}), we may consider the restricted class of varifolds satisfying the assumptions of Theorem \ref{thm:mainreg_gpos} and the following additional constraint:
\begin{enumerate}
\item[(\textbf{m})] for $p \in \SingT{V} \cap \Greg{V}$, let $G_{\ell} = \exp_{p} \left(\text{graph}(u_\ell)\right)$, for $\ell=1,2$, denote the two embedded $C^2$ hypersurfaces that are tangential at $p$ and whose union is $\spt{V}$ in a neighbourhood of $p$ (notation as in Remark~\ref{oss:touchingsinggraphs}). Then the density function $y \mapsto\Theta(\|V\|,y)$ is constant 
on $G_{\ell} \setminus \SingT{V}$ for each $\ell$. (In other words, the two hypersurfaces have separately constant integer multiplicity).
\end{enumerate}
Condition (\textbf{m}) rules out examples such as that in \cite[Figure 3]{BW}.  For the (smaller) family of varifolds that satisfy (\textbf{m}) in addition to the hypotheses of Theorem~\ref{thm:mainreg_gpos}, conclusion \textit{(ii)} of Theorem~\ref{thm:mainreg_gpos} can be strengthened to the following: \emph{there exists an oriented $n$-manifold $S_V$ and a $C^2$ immersion $\iota:S_V \to N$ that is stationary and weakly stable with respect to $J_g$ such that $V = \iota_\#(|S_V|)$, $\Greg{V} = \iota(S_{V})$ and the orientation of $S_V$ agrees with $\frac{\vec{H}}{g}$}. It is this formulation of the regularity result that we will generalize to arbitrary $g$ in Theorem \ref{thm:mainregstationarityimmersedparts}.
\end{oss}

Subject to locally uniform mass and mean curvature bounds, both the class of integral $n$-varifolds 
satisfying the hypotheses of Theorem~\ref{thm:mainreg_gpos} and the smaller class of 
integral $n$-varifolds identified in Remark \ref{oss:add(m)} are compact in the varifold topology.

\begin{thm}[\textbf{compactness: special case $g>0$}]
\label{thm:compactness_gpos}
Let $n \geq 2$ and $N$ be a Riemannian manifold of dimension $n+1$, let $g_j \, : \, N\to \R$ and $g_\infty \, : \, N\to \R$ be positive functions in $C^{1,\alpha}(N)$ such that 
$g_j \to g_\infty$ in $C^{1,\alpha}(K)$ for each compact $K \subset N$. Let $(V_j)$ be a sequence of integral $n$-varifolds that satisfy the assumptions of Theorem~\ref{thm:mainreg_gpos} with $g_{j}$ in place of $g$ and have locally uniformly bounded mass, i.e.\ $\sup_{j} \, \|V_{j}\| (K) < \infty$ for each compact $K \subset N$. Then: 
\begin{itemize}
\item[(a)] there is a subsequence of $(V_{j})$ that converges in the varifold topology to an integral $n$-varifold $V_\infty$ that satisfies the assumptions (and conclusions) of Theorem \ref{thm:mainreg_gpos} with $g_{\infty}$ in place of $g;$ 
\item[(b)] if additionally condition \emph{(\textbf{m})} holds with $V_{j}$ in place of $V$,  then  condition \emph{(\textbf{m})} also holds with $V_{\infty}$ in place of $V$. 
\end{itemize}
In either case (a) or (b), the possibility $g_\infty \equiv 0$ is allowed, in which case $\SingT{V_\infty}=\emptyset$.
\end{thm}
 
\begin{oss}[\textit{varifold limit of embedded $J_g$-stationary weakly stable hypersurfaces}]
\label{oss:GHLM}
Theorem \ref{thm:compactness_gpos} answers in particular the natural question of characterising, for an orientable ambient space $N$, the varifold limit of a sequence of orientable properly embedded hypersurfaces $M_{j}$ with mean curvature $g_j\nu_{j}$ for some continuous unit normal $\nu_{j}$ (or equivalently, $J_{g_{j}}$-stationary properly embedded hypersurfaces $M_{j}$) that are weakly $J_{g_{j}}$-stable  (i.e.~stable for $\text{Vol}_{g_j}$-preserving deformations). 
Adding assumption (\textbf{m}) as in (b) identifies a \emph{smaller} compact class that contains embedded weakly stable hypersurfaces.
\end{oss}

\begin{oss}[\textit{Stability for ambient test functions}]
\label{oss:weakeningstabilityassumption_gpos}
We point out that for the proof of Theorem \ref{thm:mainreg_gpos} a weaker version of the stability hypothesis (c) suffices: namely, it is enough to assume stability for $\text{Vol}_g$-preserving deformations of $\Greg{V}$ (as an immersion) of a specific type, i.e.~those that are induced by any ambient function $\phi \in C^1_c(\Oc)$ with $\int_{\Greg{V}\cap \Oc} \phi g d{\Hc}^n =0$ (here $\Oc$ is as in the theorem). 
As we will discuss in Section \ref{stabilitydiscussion}, this integral constraint guarantees that there exists a $\text{Vol}_g$-preserving deformation, as an immersion, whose initial velocity is $\phi \nu$, where $\nu$ is the orientation of $\Greg{V}$ in $\Oc$ given by $\frac{\vec{H}}{g}$.

If (\textbf{m}) is additionally assumed in Theorem \ref{thm:mainreg_gpos}, then the weakening of (c) just discussed can be expressed more intrinsically for the varifold by requiring the constraint on $\phi$ to take the form $\int_{\Oc} \phi g d\|V\| =0$. 
\end{oss}

\begin{oss}
\label{oss:stationarityfolvolpres}
In view of Remark \ref{oss:vol_preserv_pos}, an alternative statement can be obtained by replacing (b) of Theorem \ref{thm:mainreg_gpos} with the requirement of stationarity with respect to the area functional for $\text{Vol}_g$-preserving deformations; the conclusion is the same except for the fact that the mean curvature of $\Greg{V}$ is $\lambda g$, for some $\lambda \in \R$ (c.f.\ \cite{BW}).
\end{oss}

\begin{oss}[\textit{finite Morse index}]
\label{oss:finiteindexg_pos}
Weak $g$-stability implies that the Morse index with respect to $J_g$ is at most $1$. In Section \ref{stabilitydiscussion} we point out the well-known fact that weak $g$-stability, or more generally finiteness of the Morse index with respect to $J_g$, implies very directly that about any point there exists a ball in which strong stability holds, i.e.\ the validity of the stability inequality for all test functions compactly supported in the ball. This local consequence is the only implication of weak $g$-stability we will need for the proof of the regularity conclusion in Theorem \ref{thm:mainreg_gpos}. (The same remark applies to the other regularity results that we will present).
\end{oss}

\begin{oss}[a more general version of Theorem \ref{thm:mainreg_gpos}]
\label{oss:emptyinterior}
It is immediate that the results of this section also hold if $g$ is strictly negative, rather than strictly positive. The non-vanishing of $g$ is not the optimal assumption in Theorem \ref{thm:mainreg_gpos}, which in fact holds true if $g\geq 0$ and $\Reg_1\,V \cap \{g=0\}$ has empty interior in $\Reg_1\,V$, see Theorem \ref{thm:mainreg_emptyinterior}.

\end{oss}

\subsection{Generalization to arbitrary $g \in C^{1,\alpha}$, part I}
\label{generalizationI}

\begin{oss}[\textit{$\text{Vol}_g$-preserving deformations, arbitrary $g$}]
 \label{oss:vol_preserv_notexist}
 We begin with the remark that, when  $g \in C^{1, \alpha}$ is arbitrary, $\text{Vol}_g$-preserving deformations might not exist at all. For example, in $\R^3$ with cordinates $(x,y,z)$, let $g<0$ for $x<0$, $g=0$ for $x=0$, $g>0$ for $x>0$ and consider the hypersurface given by the plane $\{x=0\}$ oriented by choosing the normal $(1,0,0)$: then every deformations strictly increases $\text{Vol}_g$, so stationarity for $\text{Vol}_g$-preserving deformations has no content. 
 
 Moreover, even if $\text{Vol}_g$-preserving deformations exist, stationarity with respect to area for $\text{Vol}_{g}$ preserving deformations does not imply, for any choice of constant $\lambda$, stationarity with respect to $J_{\lambda g}$. Consider the following example (where $g\geq 0$): let $g>0$ in $\{(x,y,z):x^2+y^2+z^2<1\}$ and $g=0$ on $\{(x,y):x^2+y^2+z^2\geq 1\}$ and consider the unit sphere $S$. Then any $\text{Vol}_g$-preserving deformation of $S$ must be one-sided, i.e.~parametrized on $t\in [0,\eps),$ with initial velocity pointing outward. (If we push $S$ to the interior of the unit ball in a neighborhood of any point, then the enclosed volume strictly decreases, and cannot be balanced by pushing $S$ into the exterior somewhere else.) In fact $S$ is a minimizer of area for $\text{Vol}_{g}$-preserving deformations. However, $S$ does not  satisfy $\left.\frac{d}{dt}\right|_{t=0} J_{\lambda g}(t) =0$ (for any $\lambda$) for if it did, then the mean curvature must be equal to $\lambda g$ which vanishes on $S$.
Thus weak $g$-stability (i.e.\ stability for ${\rm Vol}_{g}$ preserving deformations) is not the correct notion of stability to consider for general $g$. Recalling Remark~\ref{oss:finiteindexg_pos} (which says that in case $g >0$ weak $g$-stability implies Morse index 1 with respect to $J_{g}$), we shall consider, for general $g$, the class of 
stationary points of $J_{g}$ with Morse index 1, or more generally, bounded Morse index.
\end{oss}

\begin{df}[\textit{Finite Morse index relative to ambient functions}]
\label{df:index}
Let $s$ be a non-negative integer, $\Oc \subset N$ be an orientable open set and let $\iota \, : \, S \to \Oc$ be an oriented, proper $C^{2}$ immersion, with unit normal $\nu,$ that is stationary with respect to $J_{g}.$ 
We say that $\iota$ has \emph{Morse index $s$ in $\Oc$ relative to ambient functions} if $s$ is the dimension of the largest subspace $F$ of $C^1_c(\Oc)$ such that for each $\phi \in F,$ 
$$J_g^{\prime\prime}(0) <0$$ for (normal) deformations $\psi$ given by $\psi(t, x) = \exp_{\iota(x)}\left(\iota(x) +  t (\phi \circ \iota(x))(\nu_{\iota(x)})\right)$ where $x \in S$ and $t\in (-\eps, \eps)$ for some $\eps>0$.  

\end{df}

Concerning the structural hypotheses, we shall continue to assume $(a2)$ of Theorem~\ref{thm:mainreg_gpos}, i.e.\ that ${\rm sing}_{C} \, V = \emptyset$, and this is of course necessary. As regards the structural hypothesis $(a3)$ of Theorem~\ref{thm:mainreg_gpos}, note that no class of hypersurfaces satisfying this condition literally can be compact: consider 
$g$ a smooth non-negative function that vanishes in the closure of a non-empty open set $U,$ and a sequence of embedded hypersurfaces with mean curvature prescribed by $g$, that are stable with respect to $J_g$ and have uniformly bounded area. 
As illustrated in Figure \ref{fig:Touchingdoubleminimal}, it is easy to produce an example in which they converge to a limiting varifold that has a multiplicity 2 minimal hypersurface in $U$, with touching singularities on $\p U$ (picture on the right of Figure \ref{fig:Touchingdoubleminimal}). Here $(a3)$ of Theorem~\ref{thm:mainreg_gpos} clearly fails on the limit varifold.

\begin{figure}[h]
\centering
 \includegraphics[width=4.5cm]{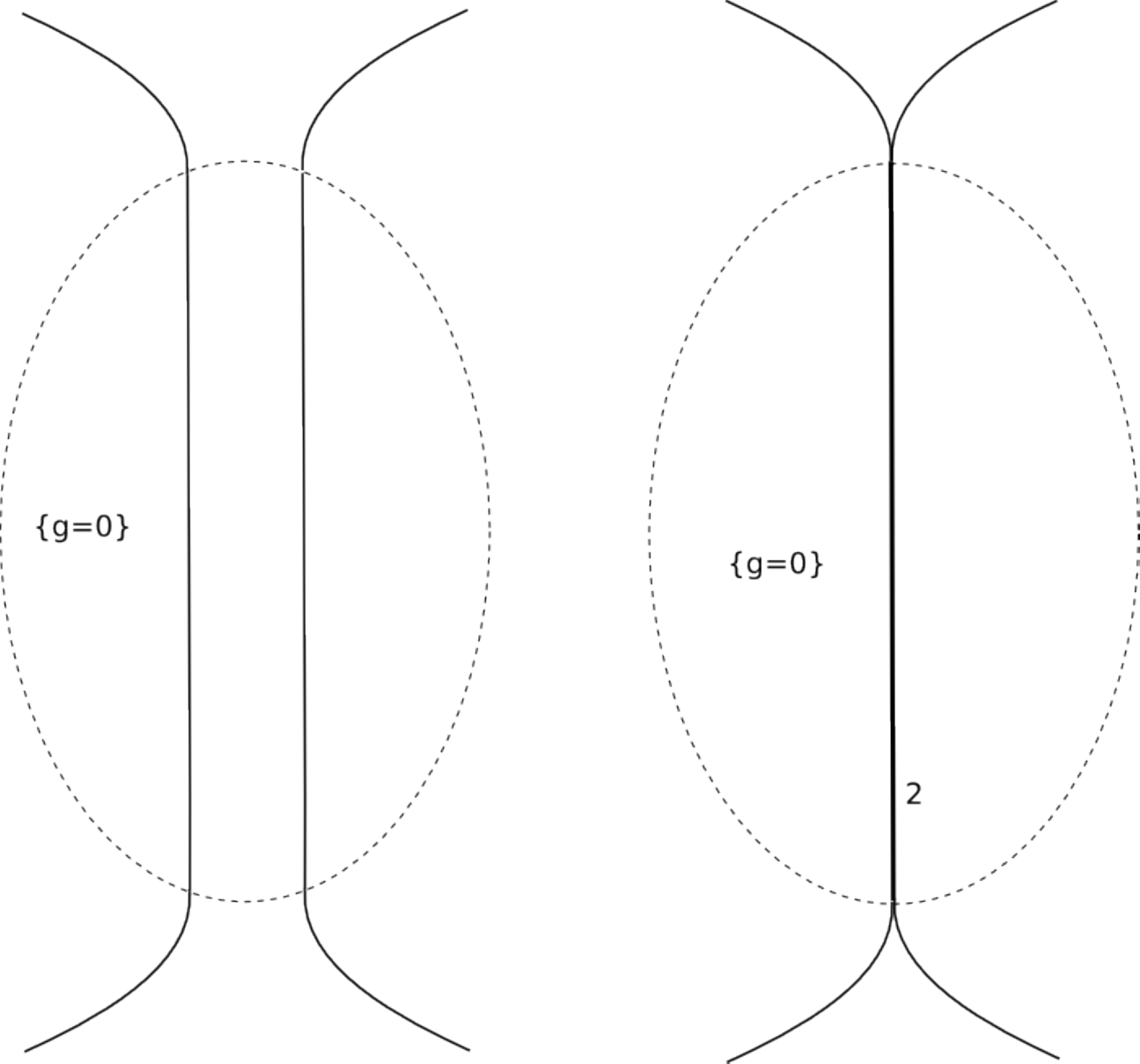} 
\caption{\small The embedded hypersurface depicted on the left belongs to a sequence $V_j$ converging in varifold sense to the integral varifold $V$ on the right, that has a minimal portion with multiplicity $2$. Condition $(a3)$ of Theorem~\ref{thm:mainreg_gpos}  fails for $V$ on $\p \{g=0\}$.}
 \label{fig:Touchingdoubleminimal}
\end{figure}

A similar example can be constructed in the case that $\{g=0\} = C \times \R$, where $C$ is a Cantor set with positive measure in $\R$. For the limit varifold we may have that $C \times \{0\}$ is in the support with multiplicity $2$, and that every point in $C \times \{0\}$ is a touching singularity for which $(a3)$ of Theorem~\ref{thm:mainreg_gpos} fails. Note that in this example the set $\{g=0\} \cap \Reg_1\,V$ has no interior in $\Reg_1\,V$.

As it turn out, the remedy for this issue is the following weakening of $(a3)$ of Theorem~\ref{thm:mainreg_gpos} which suffices for regularity and compactness theorems in the case of general $g$: 
\begin{enumerate}
\item[(\textbf{T})] 
Every touching singularity $p \in {\rm sing}_{T} \, V$ has a neighborhood $U_{p}$ such that the  coincidence set $C_{p}$ of $V$ in $U_{p}$ satisfies ${\mathcal H}^{n} \, (C_{p} \cap \{g \neq 0\}) = 0$
(see Definition~\ref{df:touchingsingularity});
\end{enumerate}

In particular, if $g>0$ as considered in Theorem \ref{thm:mainreg_gpos}, then the two conditions are the same \footnote{To avoid confusion, we remark that in the interior of $\{g=0\}$ there are no touching singularities by the maximum principle for minimal hypersurfaces (and therefore (\textbf{T}) is redundant there), so it is really on the boundary of $\{g=0\}$ that $(a3)$ of Theorem~\ref{thm:mainreg_gpos} has to be replaced by (\textbf{T}).}.

The following regularity and compactness results for arbitrary $g$ are proved in the present work and provide a general framework to answer the compactness question raised earlier. 

\begin{thm}[\textbf{regularity of finite index $J_{g}$-stationary hypersurfaces, I}]
\label{thm:mainregstationarityimmersedparts}
For $n \geq 2$ let $N$ be a Riemannian manifold of dimension $n+1$ and $g:N \to \R$ be a $C^{1,\alpha}$ function. Let $V$ be an integral $n$-varifold in $N$ such that

\begin{enumerate}
\item[(a1)] the first variation of $V$ in $L^{p}_{\rm loc} \, (\|V\|)$ for some $p >n$;
\item[(a2)] no point of ${\rm spt} \, \|V\|$ is a  classical singularity of ${\rm spt} \, \|V\|$; 
\item[(a3)] $V$ satisfies (\textbf{T}).
\end{enumerate}
Suppose moreover that: 
\begin{enumerate}
\item[(b)] the (embedded) $C^1$ hypersurface $S={\rm reg}_{1} \, V$ is stationary with respect to $J_{g}$ in the sense of Definition~\ref{stationary-def} 
taken with $\iota \, : \, S \to N$ equal to the inclusion map; 
\item[(b$^T$)] \emph{(redundant if $g\geq 0$)} each touching singularity $y \in {\rm sing}_{T} \, V$ has a neighborhood $\Oc$ such that writing 
$\exp_{y}^{-1} (\spt{V}\cap \Oc) = {\rm graph} \, u_{1} \cup {\rm graph} \, u_{2}$ for $C^{1,\alpha}$ functions $u_1 \leq u_2,$\footnote{Such $u_{1}, u_{2}$ always exist by the definition of touching singularity; see remark~\ref{oss:touchingsinggraphs}.} we have that for $j = 1,2$, the stationarity of 
$S_{j} = \Reg_1 \,V \cap \exp_{y} \, \text{graph} (u_j)$ (which follows from (b)) holds for the orientation that agrees with one of the two possible orientations of $\exp_{y} \, \text{graph} (u_j)$;
\item[(b$^{\ast}$)] \emph{(implied by (\textbf{m}) if $g> 0$)} for each orientable open set ${\Oc} \subset N \setminus ({\rm spt} \, \|V\| \setminus {\Greg} \, V)$ there exist an oriented $n$-manifold $S_{\Oc}$ and a proper $C^2$ immersion $\iota_{\Oc}:S_{\Oc} \to \Oc$ with $V\res \Oc = (\iota_{\Oc})_\#(|S_{\Oc}|)$ such that $\iota_{\Oc}$ is stationary with respect to $J_{g}$; 
\item[(c)] for each orientable open set $\widetilde{\Oc} \subset\subset  N,$ letting ${\Oc} = \widetilde{\Oc} \setminus ({\rm spt} \, \|V\| \setminus {\Greg} \, V)$ and $\iota_{\Oc}:S_{\Oc} \to N$ be as in (b*), $\iota_{\Oc}$ has finite Morse index in $\Oc$ relative to ambient functions (Definition~\ref{df:index}). 
\end{enumerate}
Then there is a closed set 
 $\Sigma \subset {\rm spt} \, \|V\|$ with $\Sigma = \emptyset$ if $n \leq 6$, $\Sigma$ discrete if $n=7$ and ${\rm dim}_{\mathcal H} \, (\Sigma) \leq n-7$ if $n \geq 8$ such that: 
 \begin{enumerate}
\item [(i)] locally near each point  $p \in {\rm spt} \, \|V\| \setminus \Sigma$, either ${\rm spt} \, \|V\|$ is a single smoothly embedded disk or ${\rm spt} \, \|V\|$ is precisely two smoothly embedded disks with only tangential intersection; if we are in the second alternative and if $g(p)\neq 0$, then locally around $p$ the intersection of the two disks is contained in an $(n-1)$-dimensional submanifold;
\item[(ii)] if $N$ is orientable then there exist an oriented $n$-manifold $S_V$, a proper $C^2$ immersion $\iota:S_V \to N$ and a global choice of unit normal $\nu$ on $S_V$ such that $V=\iota_\#(|S_V|),$ ${\rm spt} \, \|V\| \setminus \Sigma = \iota(S_V)$ and the mean curvature $H_{V}(x)$ of the immersion at $x\in S_V$ is given by $H_{V}(x) =  g(\iota(x))  \nu(x)$;
for arbitrary $N$, this conclusion applies to $V \res \left(N\setminus \{g=0\}\right)$.
 \end{enumerate}
\end{thm}

\begin{thm}[\textbf{compactness}]
\label{thm:generalcompactness}
Let $n \geq 2$ and $N$ be a Riemannian manifold of dimension $n+1$, let $g_j:N\to \R$ and $g_\infty:N\to \R$ be $C^{1,\alpha}$ functions such that $g_j \to g_\infty$ in $C_{\rm loc}^{1,\alpha}$. Let $V_j$ be a sequence of integral $n$-varifolds such that the hypotheses of Theorem~\ref{thm:mainregstationarityimmersedparts} are satisfied with $V_{j}$ in pace of $V$ and $g_{j}$ in place of $g.$  Suppose also that we have 
$\limsup_{j \to \infty} \, \|V_{j}\| (K) < \infty$ for each compact set $K \subset N$ and that $\limsup_{j \to \infty} \, \text{Morse Index}\, (\iota_{\Oc_{j}}) < \infty$ for each orientable open set $\widetilde{\Oc} \subset\subset N$,  where $\Oc_{j} = \widetilde{\Oc} \setminus ({\rm spt} \, \|V_{j}\| \setminus \Greg{V_{j}})$ and $\iota_{\Oc_{j}}$ is as in hypothesis (c) of Theorem~\ref{thm:mainregstationarityimmersedparts}. 
Then there is a subsequence $(V_{j^{\prime}})$ of $(V_{j})$ that converges in the varifold topology to a varifold $V$ that satisfies the assumptions (and conclusions) of Theorem \ref{thm:mainregstationarityimmersedparts} taken with $g_\infty$ in place of $g$. Furthermore, we have that for each orientable open $\widetilde{\Oc} \subset\subset N$, 
$\text{Morse Index} \, (\iota_{\Oc}) \leq \liminf_{j^{\prime} \to \infty} \, \text{Morse Index} \, (\iota_{{\Oc}_{j^{\prime}}})$ where $\Oc = \widetilde{\Oc} \setminus ({\rm spt} \, \|V\| \setminus \Greg{V})$ 
and $\iota_{\Oc}$ is as in  hypothesis (c) of Theorem~\ref{thm:mainregstationarityimmersedparts}.
\end{thm}

\begin{oss}[\textit{varifold limit of embedded finite-index $J_{g}$-stationary hypersurfaces}]
\label{oss:compactnessforembedded} 
Let $g_j  \to g$ in $C_{\rm loc}^{1,\alpha}(N)$, where $N$ is an oriented $(n+1)$-manifold, and let $(M_j)$ be a sequence of $C^2$ oriented embedded hypersurfaces in $N$ such that $\overline{M_j}=M_j$ and $M_{j}$ is stationary with respect to $J_{g_j}.$ Assume that the Morse index  of $M_{j}$ with respect to $J_{g_{j}}$ is locally bounded uniformly in $j.$ 
Then for any integral $n$-varifold $V$ arising as the varifold limit of a subsequence of the sequence $(|M_{j}|)$ (at least one such $V$ exists if $|M_{j}|$ have locally uniformly bounded mass in $N$), we have by Theorem~\ref{thm:generalcompactness} that $V$ is a $C^2$ immersion, i.e.~$V=\iota_\# |S|$, where $S$ is an oriented $n$-dimensional manifold (possibly with many connected components), $\iota:S \to N$ is a $C^2$ immersion that is stationary and has locally bounded Morse index  with respect to $J_g$. Moreover $\text{dim}_{\Hc}\left(\overline{\iota(S)} \setminus \iota(S)\right) \leq n-7$ and the lack of embeddedness of $\iota$ can only arise in the manner described in (i) of Theorem \ref{thm:mainregstationarityimmersedparts}.
(Within the proof of Theorem \ref{thm:generalcompactness}, see Remark \ref{oss:index_and_necks}, we will also obtain finer information on the convergence, in particular we will see that it is locally graphical and $C^2$, possibly with multiplicity, except for a finite set of points in $\iota(S)$.)

Therefore Theorem \ref{thm:generalcompactness} characterizes (analogously to what was observed in Remark \ref{oss:GHLM}) the varifold limit of embedded hypersurfaces that are stationary and stable with respect to the functional $J_g$ under a uniform area bound (a mean curvature bound is implicilty assumed by requiring $g_j  \to g$ in $C^{1,\alpha}(N)$). 
\end{oss}

\begin{oss}[\textit{Non-orientable ambient spaces}] 
\label{rem:non-orientable}
Consider $N$ that is not necessarily orientable, and let $g_{j}$, $g$ be as in the preceding remark. Let $M_{j}$ be a sequence of two-sided, embedded hypersurfaces of $N$ with $\overline{M_{j}} = M_{j}$ such that the mean curvature of $M_{j}$ is $g_{j}(x) \nu_{j}(x)$  for every  $x \in M_{j}$ where $\nu_{j}$ is a continuous unit normal to $M_{j}.$ Then the relative enclosed volume $\text{Vol}_{g_{j}}$ need not be well-defined for arbitrary compactly supported deformations of $M_{j}$ so we cannot speak of stationarity of $M_{j}$ in $N$ with respect to $J_{g_{j}}$. For deformations supported in any small geodesic ball however $J_{g_{j}}$ is well-defined and with respect to such deformations stationarity of $M_{j}$ holds. If additionally the Morse index of $M_{j}$ in any given small geodesic ball is uniformly bounded independently of $j$, and if $V$ is an integral varifold arising as the varifold limit of a subsequence of $(|M_{j}|)$,  it follows from Theorem~\ref{thm:generalcompactness} that $V$ satisfies all of the properties described in the preceding remark except that in place of orientability of $S$ we have 
that there exists a continuous (on $S$) unit normal to $\iota.$

\end{oss}

\begin{oss}[\textit{On the stationarity assumptions}]
\label{oss:why_b*}
Assumptions (b*) and (b$^T$) in Theorem \ref{thm:mainregstationarityimmersedparts} were not required in Theorem \ref{thm:mainreg_gpos}. The reason for (b$^T$) in the case when $g$ can take both positive and negative values is the fact that the one-sided maximum principle does not hold with the same strength as in the case $g\geq 0$; we will discuss this further in Sections \ref{furtherconsequencesoftheassumptions} and Appendix \ref{gchangessign}. 

The fact that hypothesis (b$^{\ast}$) is implied by hypothesis (\textbf{m}) when $g>0$ will be discussed in Remark \ref{oss:orient_gpos}. 
Let us now describe the reasons for assuming (b*). To simplify the discussion, let $g\geq 0$ and consider the varifold in Figure \ref{fig:immersion_high_mult} (we may take a trivial product with $\R$ to make $n\geq 2$ as in our theorems), with the multiplicities indicated: here (a1), (a2), (a3) are satisfied and there is no ``genuine'' singular set, i.e.~$\spt{V}=\Greg{V}$. Moreover, the varifold satisfies the stationarity assumption (b) with respect to $J_g$ on $\Reg_1\,V$ (for an appropriate choice of $g \geq 0$).

\begin{figure}[h]
\centering
 \includegraphics[width=6cm]{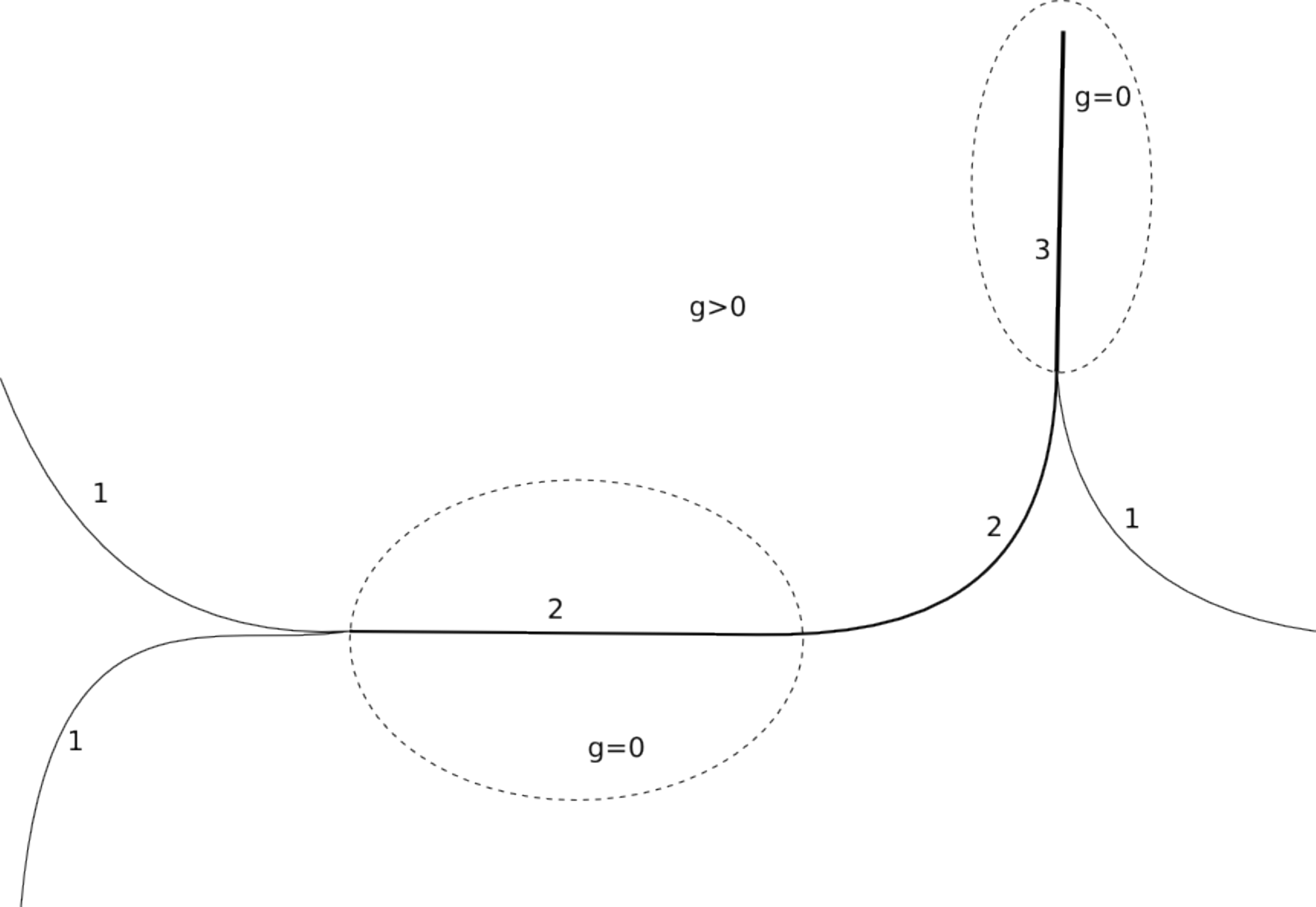} 
\caption{\small For this varifold, with the multiplicities as indicated, (\textbf{T}) is satisfied but (b$^*$) is not. Indeed, writing the varifold as the pushforward via a $C^2$ immersion, or writing its support as the image of a $C^2$ immersion, one of the connected components of the immersion is not stationary for $J_g$; it is orientable but no orientation agrees with $\frac{\vec{H}}{g}$ on $\{g \neq 0\}$.}
 \label{fig:immersion_high_mult}
\end{figure}

Note that, around every point, there is a neighbourhood in which $\spt{V}$ is the image of a $C^2$ immersion that is stationary with respect to $J_g$ (in fact, for $g\geq 0$ and assuming (\textbf{T}), it follows from the one-sided maximum principle in Section \ref{furtherconsequencesoftheassumptions} that stationarity of $\Reg_1\,V$ implies that locally around any point $\Greg{V}$ agrees with the image of a $J_g$-stationary immersion whose domain is the union of two disks). Also, locally in a neighborhood of every point, $V$ is realized as the push-forward of an oriented hypersurface by a $C^2$ immersion that is $J_g$-stationary; in other words, (b*) is valid locally around any point. However, the only way to write $\Greg{V}$ \textit{globally} as the image of a $C^2$ immersion is to use an immersion with $3$ connected components with one of them (namely the $C^2$ curve starting at the bottom-left point and ending at the top-right point) failing to be stationary with respect to $J_g$. (The same immersion also realizes $V$ as a pushforward.) This varifold is therefore undesirable from the point of view of a regularity result whose aim is to conclude that the varifold is globally a classical $J_g$-stationary immersion (as in conclusion (ii) of Theorem \ref{thm:mainregstationarityimmersedparts}). Moreover, this varifold is not a  limit of embedded stable $J_g$-stationary oriented hypersurfaces (as such, undesirable to answer the compactness question in Remark~\ref{oss:compactnessforembedded}). 

The varifold in Figure \ref{fig:immersion_high_mult} cannot be ruled out by imposing local structural assumptions on the support (assumption (\textbf{T}) on touching singularities was introduced precisely to allow the support to have local structures such as those in the right picture in Figure \ref{fig:Touchingdoubleminimal} for $\SingT{V}$): one solution to rule it out is to impose (b*). A different one will be examined in Section \ref{generalizationII}.
\end{oss}

\begin{oss}[\textit{On the stability assumption}]
\label{oss:weakeningstabilityassumption}
Assumption (c) in Theorem \ref{thm:mainregstationarityimmersedparts} requires the finiteness of the Morse index  (of the generalized-regular part of $V$) with respect to $J_g$; rather than requiring this for all deformations as an immersion, it does so only for those deformations that are induced by an ambient compactly supported test function. (Compare with Remarks \ref{oss:weakeningstabilityassumption_gpos} and \ref{oss:finiteindexg_pos} the case $g>0$.) The finiteness of the index implies (see Section \ref{stabilitydiscussion}) that around every point there is a ball in which stability holds for deformations that are induced by all ambient test function (supported in the ball), therefore, for the local conclusion \textit{(i)}, we may reduce to the case $s=0$.
\end{oss}

\begin{oss}[\textit{The open set $\Oc$ in (b*) and (c)}]
 \label{oss:openset_assumptions}
For conclusion \textit{(i)} of Theorem \ref{thm:mainregstationarityimmersedparts} it suffices to assume the validity of (b*) and (c) for open sets $\Oc$ of specific type, namely for $\Oc=C \setminus Z$, where $C$ is (in local exponential coordinates) an open cylinder or an open ball and $Z$ is a closed set with $\text{dim}_{\Hc}(Z)\leq n-7$ (see also Remark \ref{oss:forlaterref}). The fact that this weaker assumption suffices for \textit{(i)} will be clear within the proof in Section \ref{proof_reg} (in fact, we will reduce Theorem \ref{thm:mainregstationarityimmersedparts} to Theorem \ref{thm:regularity_stronglystable}). With this weaker assumption, conclusion \textit{(ii)} of Theorem \ref{thm:mainregstationarityimmersedparts} is a priori valid only in any ambient open ball or open cylinder.
\end{oss}

\subsection{Generalization to arbitrary $g \in C^{1,\alpha}$, part II}
\label{generalizationII}

The varifolds that satisfy the hypotheses of Theorem~\ref{thm:mainregstationarityimmersedparts} are regular and, by Theorem~\ref{thm:generalcompactness}, form a closed set in the varifold topology. For certain applications however, one is interested not so much on conditions guaranteeing both regularity and closure as conditions implying regularity alone that are easily checked. In this respect hypothesis (b*) of Theorem \ref{thm:mainregstationarityimmersedparts} may be unsatisfactory. This is the case, for instance, 
in the important example of the Allen--Cahn approximation scheme studied in \cite{RogTon} (see also \cite{HutchTon}). In that setting, $J_{g}$ stationary integral varifolds arise as  limits of ``weighted'' level sets of $W^{1,2}$ functions that solve certain PDEs. While hypotheses (b) and (b$^T$) of Theorem~\ref{thm:mainregstationarityimmersedparts} can easily be checked in such a situation, hypothesis (b$^{\ast}$) may be difficult to verify or may even be invalid. As pointed out in Remark \ref{oss:why_b*}, the difficulty is that local considerations are not sufficient to check (b$^{\ast}$); that is to say, while (b*) may be valid in a small ball around every point of $\Greg{V}$, it may fail in large open sets that only intersect $\Greg{V}$ (as is the case in the varifold in Figure \ref{fig:immersion_high_mult}). 

The result below (Theorem~\ref{thm:mainreg_g0_version1}) is a regularity theorem which, at the expense of not having an associated compactness result, allows to draw the same conclusions as in Theorem \ref{thm:mainregstationarityimmersedparts} but with hypothesis (b*) replaced by a condition of a different nature (hypothesis (a4)). This condition is compatible with the conclusions of \cite{RogTon} and thus provides a regularity theorem that can play a role in such a context.\footnote{We point out that the integral varifolds obtained by the limiting process analysed in  \cite{RogTon} may be as those described in the discussion that precedes (\textbf{T}), such as the varifold on the right in Figure \ref{fig:Touchingdoubleminimal}: therefore also in this context it is important to weaken the assumption on the coincidence set to (\textbf{T}).}

\begin{thm}[\textbf{regularity, II (multiplicity assumption)}]
\label{thm:mainreg_g0_version1}
For $n \geq 2$ let $N$ be a Riemannian manifold of dimension $n+1$ and $g:N \to \R$ be a $C^{1,\alpha}$ function. Let $V$ be an integral $n$-varifold in $N$ such that

\begin{enumerate}
\item[(a1)] the first variation of $V$ in $L^{p}_{\rm loc} \, (\|V\|)$ for some $p >n$;
\item[(a2)] no point of ${\rm spt} \, \|V\|$ is a  classical singularity of ${\rm spt} \, \|V\|$; 
\item[(a3)] $V$ satisfies (\textbf{T}); 
\item[(a4)] the multiplicity of $V$ is $1$ on $\Reg_1\,V \cap \{g\neq 0\}$.
\end{enumerate}
Suppose moreover that: 
\begin{enumerate}
\item[(b)]  the (embedded) $C^1$ hypersurface $S={\rm reg}_{1} \, V$ is stationary with respect to $J_{g}$ in the sense of Definition~\ref{stationary-def} 
taken with $\iota \, : \, S \to N$ equal to the inclusion map;

\item[(b$^T$)] \emph{(redundant if $g\geq 0$)} each touching singularity $y \in {\rm sing}_{T} \, V$ has a neighborhood $\Oc$ such that writing 
$\exp_{y}^{-1} (\spt{V}\cap \Oc) = {\rm graph} \, u_{1} \cup {\rm graph} \, u_{2}$ for $C^{1,\alpha}$ functions $u_1 \leq u_2,$\footnote{Such $u_{1}, u_{2}$ always exist by the definition of touching singularity; see remark~\ref{oss:touchingsinggraphs}.} we have that for $j = 1,2$, the stationarity of 
$S_{j} = \Reg_1 \,V \cap \exp_{y} \, \text{graph} (u_j)$ (which follows from (b)) holds for the orientation that agrees with one of the two possible orientations of $\exp_{y} \, \text{graph} (u_j)$;
\item[(c)] for each orientable open set ${\Oc} \subset N \setminus ({\rm spt} \, \|V\| \setminus {\Greg} \, V)$ with ${\Oc} \subset \subset N$, such that $V \res \Oc = \iota_{\Oc \, \#} \, (|S_{\Oc}|)$ where $S_{\Oc}$ is an oriented  $n$-manifold and $\iota \, : \, S_{\Oc} \to {\Oc}$ is  a proper oriented $J_{g}$-stationary $C^{2}$ immersion, $\iota_{\Oc}$ has finite Morse index relative to ambient functions (Definition~\ref{df:index}). 
\end{enumerate}
Then there is a closed set 
 $\Sigma \subset {\rm spt} \, \|V\|$ with $\Sigma = \emptyset$ if $n \leq 6$, $\Sigma$ discrete if $n=7$ and ${\rm dim}_{\mathcal H} \, (\Sigma) \leq n-7$ if $n \geq 8$ such that: 
 \begin{enumerate}
\item [(i)] locally near each point  $p \in {\rm spt} \, \|V\| \setminus \Sigma$, either ${\rm spt} \, \|V\|$ is a single smoothly embedded disk or ${\rm spt} \, \|V\|$ is precisely two smoothly embedded disks with only tangential intersection; if we are in the second alternative and if $g(p)\neq 0$, then locally around $p$ the intersection of the two disks is contained in an $(n-1)$-dimensional submanifold;
\item[(ii)] If $N$ is orientable, then ${\rm spt} \, \|V\| \setminus \Sigma$ is the image of a proper $C^2$ immersion $\iota:S_V \to N$ such that $V=\iota_\#(|S_V|)$. In this case, letting $M_j$ denote the minimal embedded hypersurface in $N$ (possibly with many connected components) such that $M_j \subset \Greg{V}$ and $\Theta(\|V\|, y) =j$ for $y \in M_j$, there is a continuous choice of unit normal $\nu$ on $S_V \setminus \iota^{-1}(\cup_{j \text{odd}}M_j)$ such that the mean curvature $H_{V}(x)$ of $S_V \setminus \iota^{-1}(\cup_{j \text{odd}}M_j)$ at $x$ is given by $H_{V}(x) =  g (\iota(x))  \nu(x)$; moreover, if each $M_j$ is orientable for $j$ odd, then $\nu$ extends continuously to all of $S_V$. Each $M_j$ for $j \geq 3$ is such that $\overline{M_j}\setminus M_j$ has Hausdorff dimension $\leq n-7$. 
 \end{enumerate}
\end{thm} 

\begin{oss}
If $N$ is simply connected (thus orientable), then it follows from the claim in Step 4 of the proof of Proposition~\ref{Prop:orientabilityimmersion} and Reamrk~\ref{discard} that $M_{j}$ as above for $j$ odd are orientable, and hence $S_{V}$ is orientable. Thus in this case ${\rm spt} \, \|V\| \setminus \Sigma$ is the image of a proper, oriented $C^2$ immersion $\iota:S_V \to N$ such that $V=\iota_\#(|S_V|)$.
\end{oss} 

\begin{oss}
The orientability aspects of conclusion (ii) will be discussed in Remark \ref{oss:O_not_simply_conn}.
\end{oss} 

In case $N$ is non-orientable, for an arbitrary embedded two-sided hypersurfaces $M$ with a continuous unit normal $\nu$, the condition that the mean curvature is $g \nu$ cannot be variationally characterised as stationarity with respect to $J_{g}$ because relative enclosed $g$-volume $\text{Vol}_{g}$ is defined only when $M$ is oriented. (cf. Remark~\ref{rem:non-orientable}). However, 
our proof of Theorem~\ref{thm:mainreg_g0_version1} gives the following result, in which hypothesis (b) of Theorem~\ref{thm:mainreg_g0_version1}  is replaced with a non-variational mean curvature assumption, and a conclusion analogous to conclusion (ii) of Theorem~\ref{thm:mainreg_g0_version1} is obtained.

\begin{thmbis}{thm:mainreg_g0_version1}[\textbf{Non-orientable $N$}]
\label{thm:non-orientable}
Suppose all hypotheses of Theorem~\ref{thm:mainreg_g0_version1} hold except (b). Suppose that  for each open set ${\Oc} \subset\subset N \setminus ({\rm spt} \, \|V\| \setminus {\rm reg}_{1} \, V)$ such that ${\rm reg}_{1} \, V \cap {\Oc}$ is two-sided, there exists a continuous unit normal $\nu$ such that the mean curvature $H_{V}$ on the (embedded) $C^1$ hypersurface 
${\rm reg}_{1} \, V \cap {\Oc}$ is given by $H_{V} = g \nu$. Then conclusion (i) of Theorem~\ref{thm:mainreg_g0_version1} holds. Furthermore, if for all odd $j$, the hypersurfaces $M_j$ as in conclusion (ii) of Theorem~\ref{thm:mainreg_g0_version1} are two-sided, then there is an $n$-manifold $S_{V}$ and a  proper $C^2$ immersion $\iota:S_V \to N$ such that $V=\iota_\#(|S_V|),$  and there is a continuous choice of unit normal $\nu$ on $S_V$ such that the mean curvature $H_{V}(x)$ of $S_V$ at $x$ is given by $H_{V}(x) =  g (\iota(x))  \nu(x)$. 
\end{thmbis}

\begin{oss}
The two-sidedness of $M_j$ for $j$ odd is true for the varifolds obtained in \cite{RogTon}. 
\end{oss}

\subsection{Remarks on the regularity theorems}

\begin{oss}
 \label{oss:smallportions} 
A key common feature of the variational assumptions in Theorems \ref{thm:mainreg_gpos}, \ref{thm:mainregstationarityimmersedparts}, \ref{thm:mainreg_g0_version1} is the fact that they are required \emph{only on orientable portions of the regular set}; here regular means $C^1$ for stationarity and and $C^2$ for stability, i.e.~respectively the regularity conditions that allow to state stationarity and stability for immersions in a classical way. A priori these portions could be very small in measure, since we assume no size control on the genuine singular set $\spt{V} \setminus \Greg{V}$. The only global assumption on $V$ is the $L^p_{\text{loc}}$-condition on the first variation, which, by Allard's compactness theorem, can be verified in many situations of interest. We also stress that the only control required on the singular set by assumption involves just two specific types of singularities that can be described by means of \emph{regular pieces that come together in a regular fashion}: this mild hypothesis favours checkability in the potential applications of the theory.
\end{oss}

\begin{oss}
 \label{oss:globalconclusion} 
The regularity theorems \ref{thm:mainreg_gpos}, \ref{thm:mainregstationarityimmersedparts} and \ref{thm:mainreg_g0_version1} have a twofold scope (in the case of Theorem \ref{thm:mainreg_gpos} we refer here to the version of the result that was described in Remark \ref{oss:add(m)}, i.e.~assuming (\textbf{m})). The three theorems give:

\noindent (i) the \textit{local} regularity of $\spt{V}$, i.e.~the desired conclusion from the point of view of varifold theory (in our context, a point is considered regular if it is in $\Greg{V}$);

\noindent (ii) a \textit{global} geometric characterization, namely the fact that $V$
is a classical immersion that is $C^2$ oriented and $J_g$-stationary. 
\end{oss}

\begin{oss}
\label{oss:orient_gpos}
We digress on the variational assumptions in Theorem \ref{thm:mainreg_gpos}. It can be checked that condition (c) is automatically satisfied if $\Oc$ is a small enough ambient ball $B$, since in this case $\Greg{V} \cap B$ is, by definition, either a graph or the union of two graphs (thus an oriented immersed hypersurface, stationary thanks to the one-sided maximum principle --- moreover the graph structure gives stability). The key observation that gives power to (c) in Theorem \ref{thm:mainreg_gpos} is the fact that \textit{whenever} ${\Oc} \subset N \setminus ({\rm spt} \, \|V\| \setminus {\Greg} \, V)$, then $\Greg{V} \cap \Oc$ \textit{is} an oriented $C^2$ immersion that is stationary with respect to $J_g$. This will be proved later in a more general context. However, we point out that, when $g>0$, it suffices to observe that the local immersion that describes $\Greg{V}$ in a neighbourhood of any point is naturally oriented by $\frac{\vec{H}}{g}$; considering the unit sphere bundle over $\Oc$ and taking the lift of $\Greg{V}$ according to the orientation $\frac{\vec{H}}{g}$ (so that points in $\Greg{V} \setminus \Reg_1\,V$ are lifted twice, since by the one-sided maximum principle, $\vec{H}$ points in opposite directions for the two graphs) we construct an embedded oriented manifold (the lift) and a $C^2$ immersion (the projection) that describes $\Greg{V} \cap \Oc$, as desired. (Also note that if (\textbf{m}) is additionally assumed, then the local lift can be done with multiplicity and we obtain, in the bundle, an embedded manifold with an integer multiplicity that is constant on each connected component, so the abstract manifold $S_{\Oc}$ can be constructed by creating as many copies of each connected component as the multiplicity requires.) These observations also show  that hypothesis (b$^{\ast}$) in Theorem \ref{thm:mainregstationarityimmersedparts} is implied by hypothesis (\textbf{m}) when $g>0$.
\end{oss}
  
\begin{oss}
\label{oss:forlaterref}
We stress, for later reference, that within the proof of Theorems \ref{thm:mainreg_gpos}, \ref{thm:mainregstationarityimmersedparts} and \ref{thm:mainreg_g0_version1} the stability condition (c) must be used (in its analytic formulation, the so-called stability inequality) on ambient balls of fixed size in which $\spt{V} \setminus \Greg{V}$ is known to be suitably small ($n-7$-dimensional at the most), as pointed out also in Remark \ref{oss:openset_assumptions}. This requires to known that $\Greg{V}$ is, in such a ball  and away from the small set, a $C^2$-immersed oriented manifold. As explained in Remark \ref{oss:orient_gpos}, this is true in a rather straightforward way and for the whole of $\Greg{V}$ for $g>0$, i.e.~in Theorem \ref{thm:mainreg_gpos}. For Theorem \ref{thm:mainregstationarityimmersedparts} this is just (b*). However, it becomes much more delicate in Theorem \ref{thm:mainreg_g0_version1}, when (b*) is not assumed and we drop the assumptions $g>0$ (even just for $g\geq 0$). More precisely, the major difficulty arises when $\{g=0\} \cap \Reg_1\,V$ has non-empty interior in $\Reg_1\,V$ (compare with Theorem~\ref{thm:mainreg_emptyinterior}): the mean curvature vanishes on such a set, preventing an immediate orientation, moreover this set may have to be covered with multiplicity by the immersion with image $\Greg{V} \cap \Oc$ (see the picture on the right in Figure \ref{fig:Touchingdoubleminimal}). The immersion itself has to be defined differently than in Remark \ref{oss:orient_gpos} and its orientability is not straightforward; this is discussed in Section \ref{orientability}.
\end{oss}

\begin{oss}
The properness of the immersion with image $\Greg{V}$ in the stability hypothesis (c) of our theorems is needed so that we can speak of deformations as an immersion that are ``compactly supported in $\Oc$'' (it guarantees that they are compactly supported in $S_{\Oc}$ too), and we will indeed need to use deformations as an immersion that are induced by test functions compactly supported in $\Oc$. 
\end{oss}

\begin{oss}
A global orientation of $\Greg{V}$ could be assumed if $V$ is the multiplicity-$1$ varifold associated to the boundary of a Caccioppoli set. Whilst too restrictive for the general theory, as explained in Sections \ref{generalizationI} (see e.g.~Fig. \ref{fig:Touchingdoubleminimal}) and Section \ref{generalizationII} (see e.g.~\cite{BW2}, \cite{RogTon}), the assumption that $V$ is the multiplicity-$1$ varifold associated to the boundary of a Caccioppoli set is fulfilled in certain interesting instances. Moreover, in this case, assumption $(\mathbf{T})$ is redundant and the functional $\text{Vol}_g$ has a very natural meaning, as pointed out after Definition \ref{Dfi:rel_encl_gvolpos}. We thus provide, in the next section, corollaries of Theorem \ref{thm:mainreg_g0_version1} for this special class of integral varifolds.
\end{oss}

\subsection{Corollaries for Caccioppoli sets}
\label{Caccioppoli}

\begin{oss}[\textit{redundancy of (a3), (a4)}]
\label{oss:Caccioppoliredundant1}
 Theorems~\ref{thm:mainreg_g0_version1} and \ref{thm:non-orientable} can of course be applied to the case in which $V$ is the multiplicity-$1$ varifold associated to the reduced boundary $\p^* E$ of a Caccioppoli set $E\subset N$. In this case De Giorgi's rectifiability theorem implies that (a3) and (a4) of Theorems~\ref{thm:mainreg_g0_version1} and \ref{thm:non-orientable} are automatically satisfied (see \cite[Remark 2.20]{BW} for (a3), while the validity of (a4) follows from the stronger fact that $\Reg_1\,V$ has multiplicity $1$). 
\end{oss}

Throughout this section, we let $N$ be a Riemannian manifold of dimension $n+1$ with $n\geq 2$ and $g:N\to \R$ be a $C^{1,\alpha}$ function. Whenever $C \subset N$ is a Caccioppoli set and $U \subset \subset N$ is an open subset, denote by $J_{g,U}$  the functional
$$J_{g,U}(C) = \|\p^*C\|(U) + \int_{C\cap U} g \, d{\Hc}^{n+1}.$$

As pointed out after Definition \ref{Dfi:rel_encl_gvolpos}, in the case in which $\p^*C \cap U$ is $C^1$ embedded, for any one-parameter family of deformations $\psi_t$ (as in Definition \ref{Dfi:rel_encl_gvolpos}) that fix $N\setminus U$ we have $\text{Vol}_g(t) = \int_{\psi_t(C) \cap U} g d{\Hc}^{n+1} - \int_{\psi_t(C) \cap U} g d{\Hc}^{n+1}$. The functional $J_{g,U}$ provides an ``absolute'' notion of $g$-enclosed volume, as opposed to the relative one in Definition \ref{Dfi:rel_encl_gvolpos}. Caccioppoli sets provide a natural setting to pose a stationarity condition with respect to $J_{g,U}$ that is stronger than the one required in Theorems~\ref{thm:mainreg_g0_version1} and \ref{thm:non-orientable}.

\begin{oss}
\label{oss:Caccioppoliredundant2}[\textit{redundancy of (b$^T$)}]
We note that the requirement of stationarity with respect to $J_{g,U}$ made on $\Reg_1\,|\p^*E|$ gives already more information than the one with respect to $J_g$: in fact (we will discuss this below in more generality) the stationarity condition with respect to $J_{g,U}$ for $U$ such that $\p^*E \cap U$ is $C^1$-embedded implies that the generalized mean curvature of $|\p^* E|$ in $U$ is given by $g \nu$, where $\nu$ is the \textit{exterior} unit normal to $E$. This additional information, that was not be present in the general framework (where we had neither a notion of interior nor a global well-defined normal to $V$), together with the fact that multiplicity $\geq 2$ points form a set of $n$-dimensional measure zero (by the De Giorgi rectifiability theorem), implies also that (b$^T$) of Theorems~\ref{thm:mainreg_g0_version1} and \ref{thm:non-orientable} is automatically satisfied when we require stationarity of $\Reg_1\,|\p^* E|$ with respect to $J_{g,U}$.
\end{oss}

In this section we provide two different formulations of the regularity theorem for Caccioppoli sets (Corollaries \ref{cor:Caccioppoli1} and \ref{cor:Caccioppoli2} below), in which we remove further structural assumptions from Theorem~\ref{thm:non-orientable} (in addition to (a3), (a4), (b$^T$) that have been removed in Remarks \ref{oss:Caccioppoliredundant1} and \ref{oss:Caccioppoliredundant2} above), by requiring a stronger stationarity condition in two different ways. At the end of the section, we also provide a result that characterizes varifold limits of Caccioppoli sets that satisfy the assumptions of Theorem \ref{thm:non-orientable}. 

 \medskip

Corollary \ref{cor:Caccioppoli1} below requires stationarity of the reduced boundary of a Caccioppoli set $E$ with respect to $J_{g,U}$ under all ambient deformations, i.e.~under any one-parameter family of diffeomorphisms of the ambient manifold that are the identity outside a compact set. The reason why this is stronger than (b) of Theorem \ref{thm:non-orientable} is that we impose the stationarity condition on the whole of $\p^*E$, rather than only on its $C^1$-embedded regular part. By strengthening the stationarity requirement in this fashion, condition (a1) of Theorem \ref{thm:non-orientable} is automatically satisfied (see \cite[Remark 2.19]{BW} or \cite[Ch. 17]{Maggi}) and moreover we obtain (following the argument in \cite[Remark 2.19]{BW}) that the generalized mean curvature of $\p^*E$ is given by $g \nu$, where $\nu$ is the outward unit normal. We thus have:

\begin{cor}
\label{cor:Caccioppoli1}
Let $E \subset N$ be a Caccioppoli set and $U\subset \subset N$ be open. Assume that, for every one-parameter family $\psi(t,x)$ of diffeomorphisms of $N$ that fix $N$ outside a compact set $K\subset U$ (i.e.~, for $\eps>0$, $\psi:(-\eps,\eps) \times N \to N$ is of class $C^1$, for every $t\in(-\eps,\eps)$ the map $\psi_t(\cdot)=\psi(t,\cdot):N\to N$ is a diffeomorphisms with $\psi_t(\cdot)=Id$ on $N\setminus K$, $\psi_0(\cdot)=Id$ on $N$) the following stationarity condition holds:
$$\left.\frac{d}{dt}\right|_{t=0} J_{g,U}(E_t) =0,$$
where $E_t$ is the Caccioppoli set $(\psi_t)(E)$. Assume moreover that  

\noindent (a2) of Theorem \ref{thm:mainreg_g0_version1} holds with $|\p^* E|$ in place of $V$ (i.e.~$|\p^* E|$ has no classical singularities in $U$);
 
\noindent (c) of Theorem \ref{thm:mainreg_g0_version1} holds with $|\p^* E|$ in place of $V$, $U$ in place of $N$. 

Then there is a closed set $\Sigma \subset \overline{\p^{*} E} \cap U$ with $\Sigma = \emptyset$ if $n \leq 6$, $\Sigma$ discrete if $n=7$ and ${\rm dim}_{\mathcal H} \, (\Sigma) \leq n-7$ if $n \geq 8$ such that locally near each point  $p \in  (\overline{\p^{*} E} \cap U)\setminus \Sigma$, either $\overline{\p^{*} E}$ is a single smoothly embedded disk or $\overline{\p^{*} E}$ is precisely two smoothly embedded disks with only tangential intersection.

If the assumptions hold true for all $U\subset \subset N$ then there exist an $n$-dimensional manifold $S$ and a $C^2$ immersion $\iota:S \to N$ such that $\overline{\p^{*} E} \setminus \Sigma = \iota(S)$ and $|\p^* E| = \iota_{\#}(|S|)$; moreover, there exists a continuous choice of unit normal $\nu$ on $S$ such that $\nu$ coincides with the exterior (measure-theoretical) unit normal on $\p^* E$ and $g \nu$ is the mean curvature of $\iota$. 
\end{cor}

The second way in which we can strengthen the stationarity assumption (b) is by considering the $L^1_{\text{loc}}$-topology on Caccioppoli sets and working with continuous one-sided deformations of a given Caccioppoli set $E$. By doing so, we are able to remove (a2) as well as (a1), (a3), (a4) and (b$^{T}$) . We refer to \cite[Definition 2.6]{BW} for the definition on one-sided one-parameter family of deformations with respect to the $L^1_{\text{loc}}$-topology and to \cite[Section 9]{BW} for the argument that rules out classical singularities under this stronger stationarity notion.

\begin{cor}
\label{cor:Caccioppoli2}
Let $E \subset N$ be a Caccioppoli set and $U\subset \subset N$ be open. Assume that, for every one-sided one-parameter family $\{E_t\}_{t\in[0,\eps)}$ of deformations of $E$ in $U$ (continuous with respect to the $L^1_{\text{loc}}$-topology, see \cite[Definition 2.6, 2.7]{BW}) the following stationarity condition holds:
$$\left.\frac{d}{dt}\right|_{t=0}J_{g,U}(E_t) \geq 0.$$
Assume moreover the following stability condition: (c) of Theorem \ref{thm:mainreg_g0_version1} holds with $|\p^* E|$ in place of $V$, $U$ in place of $N$. Then the conclusions of Corollary \ref{cor:Caccioppoli1} hold.
\end{cor}

\begin{oss}[\textit{analytic data}]
\label{oss:analytic}
If in Corollary \ref{cor:Caccioppoli2} we further assume that the Riemannian manifold $N$ and the function $g$ are analytic, then we can weaken the stability assumption to the following requirement:

\medskip

\noindent (c$^{\prime}$) for each orientable open set ${\Oc} \subset U \setminus ({\rm spt} \, \|D\chi_E\| \setminus {\Reg} \, |\p^*E|)$, $E$ has finite Morse index with respect to the functional $J_{g,\Oc}$ for ambient deformations that fix the complement of $\Oc$. 

\medskip

Note that this is a condition on $\Reg\, |\p^*E|$ \textit{only} (rather than on $\Greg{|\p^*E|}$). In Section \ref{Caccioppoli} we show why this weaker assumption implies the stability condition in Corollary \ref{cor:Caccioppoli2}. In this analytic framework, Corollary \ref{cor:Caccioppoli2} generalizes the regularity results that are known for minimizers of $J_{g,U}$ (see \cite{GonzMassTaman}, \cite{Morgan}). We conjecture that Corollary \ref{cor:Caccioppoli2} with the stability assumption (c$^{\prime}$) holds true for arbitrary smooth manifolds and for $g\in C^{1,\alpha}$.
\end{oss}

We conclude the discussion on Caccioppoli sets by addressing a  \textit{compactness} question. Figure \ref{fig:Touchingdoubleminimal} shows that in the class of Caccioppoli sets (i.e.\ with respect to $BV_{\rm loc}$ topology) we cannot expect a compactness result, analogous to Theorem \ref{thm:generalcompactness}, for boundaries of Caccioppoli sets that are stationary with respect to $J_{g,U}$ and satisfy locally uniform mass and Morse index bounds. Indeed, assuming that in Figure \ref{fig:Touchingdoubleminimal} the exterior of $E_j$, for $j\in \N$, is given by the portion between the two curves, the multiplicity 2 minimal interface that appears in the varifold limit is not present in the $BV_{\text{loc}}$-limit, causing the reduced boundary of the limiting Caccioppoli set $E$ to develop ``varifold boundary'' (at the cusps), violating the condition (a1) that the first variation be in $L^p(\|\p^*E\|)$. It is interesting to investigate and characterize, for the class of Caccioppoli sets whose reduced boundaries satisfy the assumptions of Theorem \ref{thm:non-orientable} (these assumption are implied by the hypotheses in either of the Corollaries in this section), the possible varifold limit, under the natural mass and mean curvature bounds. We have the following optimal characterization for such varifold limits. The proof of the result is (see Section \ref{proof_Caccioppoli}) an immediate consequence of Theorems \ref{thm:generalcompactness},\ref{thm:mainreg_g0_version1}, \ref{thm:non-orientable}.

\begin{cor}[varifold limits of finite index Caccioppoli sets]
\label{cor:Caccioppolicompactness} 
For $j\in \N$, let $g, g_j:N\to\R$ be $C^{1,\alpha}$ functions such that $g_j \to g$ in $C^{1,\alpha}_{\text{loc}}(N)$. Let $\{E_j\}_{j\in \N}$ be a sequence of Caccioppoli sets in $N$ such that, for every $j$, assumptions (a1), (a2) of Theorem~\ref{thm:mainreg_g0_version1} are valid with $|\p^*E_j|$ in place of $V$.
Assume that for every open set $\Oc \subset N \setminus (\spt{D\chi_{E_j}} \setminus \Reg_1\,|\p^*E_j|)$ the generalized mean curvaure of $\p^*E_j \cap \Oc$ is given by $g_j \nu_j$, where $\nu_j$ denotes the exterior unit normal to $E_j$. 

For every open set $\tilde{\Oc} \subset \subset N$ such that $|\p^*E_j|\res \Oc_j = \iota_{\Oc_j\,\#}|S_{\Oc_j}|$, where $\Oc_j = \tilde{\Oc}\setminus (\spt{D\chi_{E_j}} \setminus \Greg{|\p^*E_j|})$ and $\iota_{\Oc_j}:S_{\Oc_j} \to \Oc_j$ is a $C^2$ immersion of an $n$-manifold $S_{\Oc_j}$ whose mean curvature is given by $g_j n_j$ with $n_j$ a global normal on $S_{\Oc_j}$ that agrees with $\nu_j$ on $\iota_{\Oc_j}^{-1}(\p^* E \cap \Oc_j)$, assume that the Morse index relative to ambient functions (see Definition~\ref{df:index}) of $\iota_{\Oc_j}$ with respect to $J_{g_j}$ is finite.

Assume moreover that $\limsup_{j\to \infty} \|D\chi_{E_j}\|(K) <\infty$ for every compact $K \subset \subset N$ and that $\limsup_{j\to \infty} 
\text{Morse index } (\iota_{\Oc_j}) <\infty$ for every $\tilde{\Oc} \subset \subset N$, where $\Oc_j = \tilde{\Oc}\setminus (\spt{D\chi_{E_j}} \setminus \Greg{|\p^*E_j|})$ and $\iota_{\Oc_j}$ is as above.

Then there is a subsequence $\{j^{\prime}\}$ and an integral $n$-varifold $V$ of $N$ such that $|\p^* E_{j^{\prime}}| \to V$ as varifolds, and $V$ satisfies the hypotheses of Theorem \ref{thm:non-orientable}, and the local regularity conclusion (i) of Theorem \ref{thm:mainreg_g0_version1} holds. Moreover, $V=\iota_{\#}|S|$, where $\iota:S \to N$ is a $C^2$ immersion with image $\Greg{V}$ and there is a global normal $\hat{n}$ on $S$ such that the mean curvature of $\iota$ is $(g\circ \iota) \hat{n}$.
\end{cor}

\section{Preliminaries}
\label{preliminaries}

Let $g$ be a $C^{1,\alpha}$ function on $N$ (without any sign assumption). In Sections \ref{stationaritydiscussion} and \ref{stabilitydiscussion} we recall some well-known facts on hypersurfaces regarding stationarity and stability with respect to $J_g$ (and with respect to the area functional for deformations that preserve $\text{Vol}_g$). 

 \subsection{Stationarity: classical results} 
\label{stationaritydiscussion}
 
With notation and orientation as in Definition \ref{Dfi:rel_encl_gvolpos}, we consider a proper $C^1$ immersion $\iota: S \to \Oc$, for $\Oc \subset N$ open and oriented and $S$ an orientable $n$-dimensional manifold; let $S$ be oriented by a choice $\nu$ of unit normal so that the frame $d \iota(\vec{S}) \wedge \nu$ is positive with respect to the orientation of $\Oc$. Let $dS$ denote the $n$-volume induced on $S$ by $\iota$ and by the Riemannian metric of $N$. 

Let $\psi$ be a deformation of $\iota$ as in Definition~\ref{Dfi:rel_encl_gvolpos}, with $\psi_t(x)=\psi(t,x)$ for each $t \in(-\eps,\eps)$ and $x \in S$. Thus $\psi$ is differentiable on $(-\eps,\eps) \times S$, $\psi_0=\iota$ and $\psi_t(x)=\iota(x)$ for all $t\in(-\eps, \eps)$ and $x \in S \setminus \iota^{-1}(\Oc)$. Recall that stationarity of $\iota$ with respect to the functional $J_g(t)= A(t) + \text{Vol}_{g}(t) (=\|(\psi_t)_\# (|S|)\res \Oc\| + \text{Vol}_g(t))$ is the requirement
\begin{equation}
 \label{eq:stationarityS}
 \left.\frac{d J_g(t)}{dt}\right|_{t=0} = 0 \text{ for any $\Oc$ and } \psi \text{ as above.}
\end{equation}

As we are about to show, the condition just stated is equivalent to the requirement that $\iota$ is a $C^2$ immersion whose mean curvature vector $\vec{H}$ is pointwise dictated by the value attained by $g$, i.e.~$\vec{H}=g \nu$. Moreover, if $\iota$ is a $C^1$ embedding, then the stationarity with respect to $J_g(t)$ for ambient deformations suffices to draw the same conclusion.
Indeed, for $\iota$, $\Oc$ and $\psi$ as above, one can compute that

$$\left.\frac{d \text{Vol}_g(t)}{dt}\right|_{t=0} = \int_S (g \circ \iota) \left.\frac{d \psi}{dt}\right|_{t=0} \cdot \nu \,dS,$$
where $\nu$ is the unit normal to the immersed hypersurface in $\Oc$. (The proof of this fact follows \cite[Lemma 2.1 (ii)]{BarbDoCEsch} with the due changes.) In particular, we point out that the normal velocity $\zeta=\left.\frac{d \psi}{dt}\right|_{t=0} \cdot \nu$ of a $\text{Vol}_g$-preserving deformation satisfies 

\begin{equation}
\label{eq:zeroaverageg}
\int \zeta (g\circ \iota) \, dS=0.
\end{equation}
Recall the first variation formula (e.g. \cite[Lemma 2.1 (i)]{BarbDoCEsch} for $C^2$ immersions and \cite{SimonNotes} in general) 
$$\left.\frac{d A(t)}{dt}\right|_{t=0} = -\int_S \left.\frac{d \psi}{dt}\right|_{t=0} \cdot \vec{H} \, dS,$$
where $\vec{H}=H\nu$ is the generalized mean curvature vector, with the convention that $H$ stands for the sum of principal curvatures of $\iota$ (rather than the average as in \cite{BarbDoCEsch}). The claim on the generalized mean curvature $\vec{H} = (g\circ \iota) \nu$ follows and implies that $S$ is $C^2$ by standard elliptic theory (see Sections \ref{exponentialchart} and \ref{furtherconsequencesoftheassumptions} for details). To avoid confusion, note that any connected orientable $S$ that is stationary for the functional $\|\psi_{t \, \#} |S|\res \Oc\| + \text{Vol}_g(t)$ can only be so for one of the two possible orientations (for the other orientation, it is stationary for the functional $\|(\psi_{t\, \#} |S|\res \Oc\| - \text{Vol}_g(t)$ instead). 

\medskip
 
\begin{oss}
\label{oss:impliesVol_pres}
Condition (\ref{eq:stationarityS}) implies that $S$ is a stationary hypersurface with respect to the standard area functional for $\text{Vol}_g$-preserving (ambient) deformations (since for such deformations $J_g$ is just the area up to a constant).
\end{oss}

In the case $g>0$ we have, conversely to (\ref{eq:zeroaverageg}): whenever $\Oc$ is orientable, $\iota$ is smooth and $\zeta \in C^1_c(\iota^{-1}(\Oc))$ is such that $\int \zeta \,(g \circ \iota) \,dS =0$, then there exists a deformation $\psi:(-\eps,\eps) \times S \to N$ such that $\text{Vol}_g(t) =0$ for all $t\in (-\eps, \eps)$ and $\left.\frac{d \psi}{dt}\right|_{t=0} = \zeta \nu$, i.e.~$\psi$ is a $\text{Vol}_g(\cdot)$-preserving normal deformation (as an immersion) of $\iota$, with initial velocity $\zeta \nu$. For the proof it suffices to suitably modify \cite[Lemma 2.4]{BarbDoCarmo} or \cite[Lemma 2.2]{BarbDoCEsch}: we note that the proof would fail if $g$ vanished identically on an open set contained in the hypersurface, more precisely, following  \cite[Lemma 2.4]{BarbDoCarmo}, we would not be able to choose the auxiliary function (that is denoted by $g$ in  \cite[Lemma 2.4]{BarbDoCarmo}), or, following  \cite[Lemma 2.2]{BarbDoCEsch}, we would not be able to solve the initial value problem (because the volume form denoted there by $d\overline{M}$ has to be multiplied by our function $g$, and so the left hand side could vanish identically). The examples in Section \ref{oss:vol_preserv_notexist} show that this failure is not a technical problem of the proof, but the deformation might fail to exist. A more general condition than $g>0$ for the validity of the property just expressed, is the fact that $\{g=0\}$ has non-empty interior in the hypersurface.

\begin{oss}
\label{oss:lambda_g_nu}
If $g>0$, then the condition of stationarity with respect to the standard area functional for $\text{Vol}_g$-preserving (ambient) deformations (see Remark \ref{oss:impliesVol_pres}) turns out to be ``essentially'' equivalent to (\ref{eq:stationarityS}), in the following sense. Let $\iota$ be an oriented immersion of $S$ (an analogous statement holds for an embedding) into an oriented manifold, assume that $\iota$ is stationary with respect to the hypersurface area for $\text{Vol}_g$-preserving deformations as an immersion (ambient deformations if $\iota$ is an embedding). Then there exists $\lambda \in \R$ such that the mean curvature $\vec{H}$ of $\iota$ is given by $\lambda g \nu$ and $\iota$ satisfies the stationarity condition 
$\left.\frac{d}{dt}\right|_{t=0}\left(\|\psi_{t \, \#}(|S|)\res \Oc\| + \lambda \text{Vol}_g(t)\right)=0$ for arbitrary deformations as an immersion (ambient deformations if $\iota$ is an embedding). The proof follows \cite[Prop. 2.3]{BarbDoCEsch} and \cite[Prop. 2.7]{BarbDoCarmo}. The real number $\lambda$ becomes a Lagrange multiplier and the $\text{Vol}_g$-preserving constraint gets encoded in the functional, thus allowing unconstrained deformations. We then conclude that, upon replacing $g$ with $\lambda g$, $\iota$ is stationary in the sense of (\ref{eq:stationarityS}).
\end{oss}
 
\subsection{Stability: classical results}
\label{stabilitydiscussion}

Keeping notations as in Section \ref{stationaritydiscussion}, assuming that the immersion $\iota$ is stationary for $J_g$, the second variation 
$$\left.\frac{d^2}{dt^2 }\right|_{t=0} J_g(t)=\left.\frac{d^2}{dt^2 }\right|_{t=0}\left(\|\psi_{t \, \#} |S|\res \Oc\| + \text{Vol}_g(t)\right)$$
along a normal deformation $\psi_t$ with initial (normal) velocity $\zeta=\left.\frac{d \psi}{dt}\right|_{t=0}$ can be expressed by

$$J_g^{\prime \prime} (0) =\int_{S} |\nabla \zeta|^2-(\text{Ric}(\nu, \nu)+|A|^2 - D_{\nu}g)\zeta^2,$$ 
where $A$ is the second fundamental form of the immersed hypersurface and $\nabla$ denotes the gradient induced by the metric $dS$. This can be proved following the arguments in \cite[(2.4)-(2.5)]{BarbDoCEsch} and \cite[Appendix 4]{BarbDoCarmo}. Stability of $\iota$ with respect to $J_g$ (as an immersion) implies therefore that $\iota$ satisfies the strong stability inequality
\begin{equation}
\label{eq:strongstabilityinequalityimmersion}
\int_{S} (\text{Ric}(\nu, \nu)+|A|^2- D_{\nu}g)\zeta^2 \, dS \leq \int_{S} |\nabla \zeta|^2 \, dS
\end{equation}
for any $\zeta \in C^1_c(S)$. 

\medskip

In our regularity results Theorems \ref{thm:mainregstationarityimmersedparts}, \ref{thm:mainreg_g0_version1} we impose a condition on $J_g^{\prime \prime}$ for $V\res \Oc$ (whenever $V\res \Oc$ \emph{is} an oriented, $J_{g}$-stationary $C^2$ immersion) by employing normal deformations (as an immersion) induced by $\phi \in C^1_c(\Oc)$. If the Morse index in (c) of Theorems \ref{thm:mainregstationarityimmersedparts}, \ref{thm:mainreg_g0_version1} is zero, then the assumption amounts to requiring the validity of (\ref{eq:strongstabilityinequalityimmersion}) for all $\zeta$ of the type $\phi \circ \iota$. More intrinsically, we can write the inequality as 
$$\int_{\Oc} (\text{Ric}(\nu, \nu)+|A|^2- D_{\nu}g)\phi^2 \, d\|V\| \leq \int_{\Oc} |\nabla \phi|^2 \, d\|V\|$$
for all $\phi \in C^1_c(\Oc)$, where $\nabla$ is the ambient gradient projected onto the tangent plane to $V$, which is well-defined everywhere on $\Greg{V}$ (there are no classical singularities).

\medskip

\textit{Finite index implies strong stability locally}. If the Morse index $s\neq 0$ then the same inequality holds for all $\phi \circ \iota$ such that $\phi$ is $L^2$-orthogonal to an $s$-dimensional linear subspace of $C^1_c(\Oc)$. We will now show that this condition implies that for every point we can find a ball around it such that (c) is valid in the ball with Morse index $0$.

Assume condition (c) is valid in $\Oc$ with Morse index $s$. The key observation is that, given any two disjoint open sets contained in $\Oc$, then in at least one of them the Morse index   is $\leq s-1$. If that were not the case, then we would find that the total index in $\Oc$ is $\geq 2s$. Consider now an arbitrary point $x$ in $\Oc$ and a fixed ball $B_R(x) \subset \Oc$. Choose $0<r<R$ and consider the disjoint open sets $B_r(x)$ and $B_R(x) \setminus \overline{B}_r(x)$. By the previous observation, in at least one of the two we have index $\leq s-1$ and this can be done for every $r$ in $(0,R)$. Therefore, either we find $r$ such that the index in $B_r(x)$ is $\leq s-1$, or the index is $\leq s-1$ in $B_R(x) \setminus \{x\}$. The latter means that for every $\phi \in C^1_c(B_R(x) \setminus \{x\})$ and orthogonal to a $(s-1)$-dimensional subspace, the inequality holds. Using $n\geq 2$ and the fact that the density $\Theta(\|V\|,x)$ is finite (by (a1), via the almost monotonicity formula), a standard capacity argument shows that for every $\phi \in C^1_c(B_R(x))$ orthogonal to a $(s-1)$-dimensional subspace the inequality holds. In other words, both alternative lead to the existence of a ball around $x$ in which the index is $\leq s-1$. The argument can now be iterated finitely many times to reach the conclusion.

\medskip

\textit{Weak stability, $g>0$}. For Theorem \ref{thm:mainreg_gpos}, with (c) replaced by the weaker assumption in Remark \ref{oss:weakeningstabilityassumption_gpos}, we note that the argument just described applies as in the case $s=1$. In fact, one can show that weak stability condition assumed there implies that the Morse index  is $\leq 1$; alternatively, one can directly prove that whenever two disjoint open sets $A_1$ and $A_2$ in $\Oc$ are considered, then in at least one of them we have strong stability. (To see this, argue by contradiction and assume that stability fails both for a function $\phi_1$ supported in $A_1$ and for a function $\phi_2$ supported in $A_2$. Take a linear combination $\phi$ of them such that $\phi$ verifies the integral constraint that is needed for the validity of the weak stability. Then adding up the two reversed inequalities for $\phi_1$ and $\phi_2$, we contadict the weak stability inequality for $\phi$. Also see \cite{BW} for a detailed proof.)

\begin{oss}[\textit{$\text{Vol}_g$-preserving deformations, $g>0$}]
Let $g>0$. The stability of $\iota$ with respect to the area functional for $\text{Vol}_g$-preserving normal deformations (as an immersion) implies that $\iota$ satisfies the weak stability inequality
\begin{equation}
\label{eq:weakstabilityinequalityimmersion}
\int_{S} (\text{Ric}(\nu, \nu)+|A|^2- D_{\nu}g)\zeta^2 \, dS \leq \int_{S} |\nabla \zeta|^2 \, dS
\end{equation}
for any $\zeta \in C^1_c(S)$ such that $\int_{S}\zeta (g \circ \iota) \,dS =0$. (To see this, given such a $\zeta$ we can produce, as said earlier, a $\text{Vol}_g(\cdot)$-preserving normal deformation (as an immersion) of $\iota$, with initial velocity $\zeta$. Along this deformation the functional $J_g$ agrees with the area functional, so the immersion is stable for $J_g$ under this deformation, i.e.~the inequality is valid.)

Note that the validity of the weak stability inequality (\ref{eq:weakstabilityinequalityimmersion}) for $\iota$ is in fact equivalent to the variational characterization just discussed (i.e.~stability with respect to the area functional for $\text{Vol}_g$-preserving normal deformations). More precisely, assume that for any $\zeta \in C^1_c(S)$ such that $\int_{S}\zeta (g \circ \iota) \,dS =0$  we have $\int_{S}(\text{Ric}(\nu, \nu)+|A|^2- D_{\nu}g)\zeta^2  \,dS\leq \int_{S} |\nabla \zeta|^2 \,dS$. Let $\psi$ be any $\text{Vol}_g$-preserving normal deformation (as an immersion) of $S$ in $\Oc$. Then the initial normal velocity $\zeta$ of $\psi$ satisfies the constraint for the validity of the weak stability inequality. By the expression of the second variation $J_g^{\prime \prime} (0)$ we conclude that $S$ in $\Oc$ is stable with respect to $J_g$ for $\text{Vol}_g$-preserving deformations (as an immersion), and therefore $S$ in $\Oc$ is stable with respect to the area functional for $\text{Vol}_g$-preserving normal deformations (as an immersion).
\end{oss}

\section{Local consequences of stationarity}

\subsection{Localization to a chart}
\label{exponentialchart}
  
All our regularity conclusions in this work are local results. For the local study of a varifold as in Theorems \ref{thm:mainreg_gpos}, \ref{thm:mainregstationarityimmersedparts}, \ref{thm:mainreg_g0_version1} around an arbitrary point $X\in N$ it is convenient to use a local chart given by the exponential map $\text{exp}_X$ at $X$. Let $R_X$ denote the injectivity radius at $X$ and $\mathcal{N}_{\rho}(X)$ the normal coordinate ball of radius $\rho < R_X$ around $X$. 

If $\widetilde{V}$ is an integral $n$-varifold in $N$, then for $X_0 \in \spt{\widetilde{V}}$ and $\rho_0 < R_{X_0}$ we consider the integral $n$-varifold $V$ in $B^{n+1}_{\rho_0}(0)$ defined by
$$V={\text{exp}_{X_0}^{-1}}_\# \left(\widetilde{V} \res \mathcal{N}_{\rho_0}(X_0)\right).$$ In the following discussion, assuming that $\tilde{V}$ satisfies the assumptions of Theorem \ref{thm:mainreg_gpos} we derive the hypotheses that we need to impose on $V$ in order to rephrase Theorem \ref{thm:mainreg_gpos}  in $B^{n+1}_{\rho_0}(0)\subset \R^{n+1}$.

$\bullet$ The structural conditions (a2) and (a3) are not affected by the exponential map, so they hold in the same formulation for $V$.

$\bullet$ Assume stationarity of $\Reg\,\widetilde{V}$ with respect to the functional in Theorem \ref{thm:mainreg_gpos} (b): this becomes a stationarity condition for $\Reg\, V$ with respect to a new functional, as we now describe. 

Define, for $M$ a $C^1$ hypersurface in $B^{n+1}_{\rho_0}(0)$,

$$\mathcal{F}_{X_0}(M)=\int_{M} |\Lambda_n (D\text{exp}_{X_0}(y)) \circ T_{y} M| d{\Hc}^n(y) = \int_M F(y,\nu(y)) d{\Hc}^n(y)$$
where  $\nu$ is a continuous unit normal to $M$ and for $(y, p) \in B_{\r_{0}}^{n+1}(0)\times {\mathbb R}^{n+1} \setminus \{0\}$, $F(y,p) = |\Lambda_n (D\text{exp}_{X_0}(y)) \circ p^{\perp}|$. The functional $\mathcal{F}_{X_{0}}$ in $B^{n+1}_{\rho_0}(0)$  encodes, via the area formula, the $n$-dimensional area functional in $\mathcal{N}_{\rho_0}(X_0)$ (see \cite[Remark 1]{SS}). The integrand $F:B^{n+1}_{\rho_0}(0) \times \R^{n+1} \setminus \{0\}  \to [0,\infty)$ is smooth, depends only on the geometric data and satisfies the homogeneity condition $F(y,\lambda p)=\lambda F(y,p)$ for $\lambda>0$ (\cite[(1.2)]{SS}). Moreover, the bounds \cite[(1.3)-(1.4)]{SS} hold for some constants $\mu$, $\mu_{1}$ depending only on the Riemannian manifold $N$. Since $D\text{exp}_{X_0}$ is the identity at $0$ we have $F(0,p)=|p|$ for every $p\in \R^{n+1} \setminus \{0\}$ as required by \cite[(1.5)]{SS} (i.e.~$\mathcal{F}$ agrees with the Euclidean area at the origin). 

The relative enclosed $g$-volume in $N$ changes to $\text{Vol}_{(g\circ \text{exp}_{X_0})\sqrt{|\mathscr{h}|}}$ in $B^{n+1}_{\rho_0}(0)$, where $\mathscr{h}$ is the Riemannian metric on $B^{n+1}_{\rho_0}(0)$ obtained by pulling back the metric on $N$ by the exponential map. We thus conclude that $\Reg\,V$ is stationary in $B^{n+1}_{\rho_0}(0)$ with respect to the functional $\mathcal{F}_{X_0} + \text{Vol}_{(g\circ \text{exp}_{X_0})\sqrt{|\mathscr{h}|}}$. Since $\sqrt{|\mathscr{h}|}>0$ and smooth, it is readily seen that $\text{Vol}_{(g\circ \text{exp}_{X_0})\sqrt{|\mathscr{h}|}}$ is a functional of the same type as the relative enclosed $g$ volume on $N$ (which also preserves any sign assumption on $g$), so we will write the stationarity assumption on $V$ with respect to a functional $\mathcal{F}_{X_{0}}  + \text{Vol}_{g}$ on $B^{n+1}_{\rho_{0}}(0)$, with $F$ satisfying the conditions above.

\medskip

We postpone the discussion of the stability condition for $\mathcal{F}_{X_{0}}  + \text{Vol}_{g}$ to Section \ref{Schoenslemma}.

\medskip

$\bullet$ We finally consider condition (a1) of Theorem \ref{thm:mainregstationarityimmersedparts}, i.e.~the first variation of $\widetilde{V}$ in $N$ with respect to the (Riemannian) area functional is representable as $\delta \, \widetilde{V} = -H_{\widetilde V} \|\widetilde{V}\|,$ with $H_{\widetilde V} \in L^p(\|\widetilde{V}\|)$, $p>n$. Recall that this means
$$\delta \, \widetilde{V}(\widetilde{\psi}) = -\int_{N}\langle H_{\widetilde V} , \widetilde\psi \rangle d\|\widetilde{V}\|$$ 
for each compactly supported vector field $\widetilde\psi \in \chi_{c}(N)$. (The vector-valued function $H_{\widetilde V} \in L^{p}_{\rm loc}(\|\widetilde{V}\|)$ is referred to as the generalised mean curvature of $\widetilde{V}$.)
Pulling back via $\text{exp}_{X_0}$ as above, we obtain
$$\delta_{{\mathcal F}_{X_{0}}}V(\psi) = -\int_{B_{\r_{0}}^{n+1}(0)} \langle H_{\widetilde{V}}\circ\text{exp}_{X_{0}},\, \psi \rangle \, d\|V\|$$
for all $\psi \in C_{c}(B_{\r_{0}}^{n+1}(0); {\mathbb R}^{n+1}).$
By a direct calculation, using the fact that the first variation of $V$ with respect to the (Euclidean) area functional is 
given by $$\delta V(\psi) = \int_{B_{\r_{0}}^{n+1}(0) \times G_{n}} {\rm div}_{S} \, \psi(X) \, dV(X, S),$$ 
it then follows (as in \cite[Section 5]{SS}) that 

\begin{eqnarray}\label{approx-firstvar}
\left|\int_{B_{\r_{0}}^{n+1}(0) \times G_{n}} {\rm div}_{S} \, \psi(X) \, dV(X, S) \right| &\leq & \m_{1}\int_{B_{\r_{0}}^{n+1}(0)} \left(|\psi(X)| + |X||\nabla\psi(X)|\right) \, d\|V\|(X)\nonumber\\ 
&& + \mu \int_{B_{\r_{0}}^{n+1}(0)} {\hat H}_{V}(X)|\psi(X)| \, d\|V\|(X)
\end{eqnarray}
\noindent
for all $\psi \in C_{c}^{1}(B_{\r_{0}}^{n+1}(0); {\mathbb R}^{n+1}),$ where ${\hat H}_{V} = |H_{\widetilde{V}} \circ {\rm exp}_{X_{0}}|$ and $\mu$, $\m_{1}$ are fixed positive  constants depending only on the Riemannian manifold $N$ and $| \cdot |$ denotes 
the Euclidean length.   

For any $Y \in B_{\r_{0}/2}^{n+1}(0)$, let $\varphi_{Y} = {\rm exp}_{{\rm exp}_{X_{0}}(Y)}^{-1} \circ {\rm exp}_{X_{0}}$ in a neighbourhood of $Y$ and extend $\varphi_{Y}$ to be a diffeomorphism of $B_{\r_{0}}^{n+1}(0)$ onto itself. 
Then $\varphi_{Y}(Y) = 0$ and inequality (\ref{approx-firstvar}) holds with constants $\mu, \mu_{1}$ depending only on (fixed) bounds on the metric and its derivatives in $\text{exp}_{X_0}^{-1}\left({\mathcal N}_{\r_{0}}(X_{0})\right)$ and with 
the varifold $\varphi_{Y \, \#} V$ in place of $V$ and with ${\hat H}_{\varphi_{Y \, \#} V} = {\hat H}_{V} \circ \varphi_{Y}^{-1}.$

In summary, if (a1) of Theorem \ref{thm:mainregstationarityimmersedparts} holds for $\widetilde{V}$ in $\mathcal{N}_{\rho_0}^{n+1}(X_0)$, then the varifold $V={\text{exp}_{X_0}^{-1}}_\# \left(\widetilde{V} \res \mathcal{N}_{\rho_0}(X_0)\right)$ satisfies the following condition:

(a1$^{\prime}$) There exists a non-negative function ${\hat H}_{V} \in L^{p}(\|V\|)$ and positive constants $\mu_{1}$, $\mu$ with the property that for any $Y \in B_{\r_{0}/2}^{n+1}(0)$ there exists a diffeomorphism $\varphi_Y$ of $B_{\rho_0}^{n+1}(0)$ onto itself with $\varphi_Y(Y)=0,$ $\varphi_{0} = {\rm identity}$ and with $\||D\varphi_{Y}\|_{C^{0}(B_{\rho_{0}}^{n+1}(0))} +
 \|D\varphi_{Y}^{-1}\|_{C^{0}(B_{\rho_{0}}^{n+1}(0))}  \leq \mu_{1}$ such that inequality (\ref{approx-firstvar}) holds with $\varphi_{Y \, \#} V$ in place of $V$ and with 
${\hat H}_{\varphi_{Y \, \#} V} = {\hat H}_{V} \circ \varphi_{Y}^{-1}$ in place of ${\hat H}_V.$

\begin{oss}
Suppose (\ref{approx-firstvar}) holds for a varifold $V$ in $B_{\r_{0}}^{n+1}(0).$ Then for any $\sigma \in (0, 1)$, (\ref{approx-firstvar})
holds with $\eta_{0, \sigma \, \#} \, V$ in place of $V$, $B_{{\r_{0}}/\sigma}^{n+1}(0)$ in place of $B_{\r_{0}}^{n+1}(0)$ and with $\mu_{1}\sigma$ in place of $\mu_{1}$ and 
${\hat H}_{\eta_{0, \sigma \, \#} \, V} = \sigma {\hat H}_{V}$. Also, for any orthogonal rotation $\G \, : \, {\mathbb R}^{n+1} \to {\mathbb R}^{n+1}$, (\ref{approx-firstvar})
holds with $\G_{\#} \, V$ in place of $V$ and with ${\hat H}_{\G_{\#} \, V} = {\hat H}_{V} \circ \G^{-1}$.
\end{oss}

\subsection{PDEs and the one-sided maximum principle}
\label{furtherconsequencesoftheassumptions}

Allard's theorem \cite{Allard} and \cite[Lemma A1]{BW} guarantee that $\Reg_1\,V$ is an open dense set under assumption (a1), and moreover (by standard PDE theory) that $\Reg_1\,V$ is of class $C^{1,\alpha} \cap W^{2, p}$. 
The variational assumption (b) in Theorems \ref{thm:mainreg_gpos}, \ref{thm:mainregstationarityimmersedparts}, \ref{thm:mainreg_g0_version1} implies higher regularity of $\Reg_1\,V$:

\begin{Prop}[\textbf{$\Reg_1\,V$ is $C^2$}]
\label{Prop:reg1}
For the varifolds considered in Theorems \ref{thm:mainreg_gpos}, \ref{thm:mainregstationarityimmersedparts}, \ref{thm:mainreg_g0_version1} we have $\Reg_1\,V \subset \Reg{V}$, in particular $\Reg_1\,V \subset \Greg{V}$ .
\end{Prop}

\begin{proof}
Whenever $B_\sigma^{n+1}(p) \cap \spt{V} \subset \Reg_1\,V$ then, upon possibly reducing $\sigma>0$ we can locally write $\Reg_1\,V$ around $p$ as the graph (with respect to the tangent plane at $p$) of a $C^1$ function $u$. If $N$ is an open set in $\R^{n+1}$, the stationarity condition with respect to $J_g$ for one of the two possible orientations of $\text{graph}(u)$ implies the validity of one of the two following PDEs in their weak form: 
$$\text{div}\left( \frac{Du}{\sqrt{1+|Du|^2}}\right) = g(\cdot, u) \, \text{ or } \,\text{div}\left( \frac{Du}{\sqrt{1+|Du|^2}}\right) = -g(\cdot, u).$$
By elliptic theory we get that $u \in C^{2,\alpha}$, therefore $\Reg_1\,V \subset \Reg{V}$. For arbitrary $N$, we write the stationarity condition with respect to $\mathcal{F}  + \text{Vol}_{g}$ as described in Section \ref{exponentialchart} and we get the weak form of one of the following (two) PDEs (note that the following Euler-Lagrange equation is written by turning the functional $\mathcal{F}$ into a non-parametric one with $y=(x,u(x))$ and using the area formula, as in \cite[16.8]{GT}):
\begin{equation}
 \label{eq:E-L_Riem}
\sum_{j=1}^n D_j\left( D_{p_j}F((x,u),Du,-1)\right) + D_{n+1}F((x,u),Du,-1) = \pm g(x,u),
\end{equation}
where the second order terms 
$$D_{p_i p_j}F((x,u),Du) D^2_{ij} u$$
guarantee the ellipticity of the PDE thanks to the fact that $F$ satisfies conditions \cite[(1.2)-(1.5)]{SS} (see \cite[16.8]{GT}, \cite[Remark 1]{SS}). The regularity conclusion therefore holds.
\end{proof}

\begin{oss}
Proposition \ref{Prop:reg1} applied in a small enough neighbourhood of $p \in \SingT{V}$ under assumptions (a1)-(a2)-(b) implies that $\spt{V}$ is, in the chosen neighbourhood, the union of two $C^{1,\alpha}$ graphs that are $C^2$ away from the closed set on which they coincide.  
\end{oss}

We pass to a further important consequence of the stationarity condition.

\begin{Prop}[\textbf{one-sided maximum principle}]
\label{Prop:one-sided-max}
Let $V$ satisfy assumptions (a1)-(a2)-(a3)-(b)-(b$^T$) of Theorem \ref{thm:mainregstationarityimmersedparts}. Let $p\in \SingT{V}$ and let $u_1 \leq u_2$ be the $C^{1,\alpha}$ functions whose (ordered) graphs describe $\spt{V}$ in a cylindrical neighbourhood $\mathcal{C}=B^n_\sigma \times (-\sigma, \sigma)$ of $p$, in a reference frame chosen with respect to the tangent plane at $p$. Denote by $A$ the open complement of the set $\{x\in B_\sigma^n:u_1=u_2\}$. Then $u_1$ and $u_2$ satisfy (classically) on $A$ one of the two PDEs (\ref{eq:PDEsfortwographs}) below, and they do not satisfy the same one (i.e.~they solve with distinct signs on the right-hand-side). If $g \geq 0$ then $u_1$ solves (classically) on $A$ the PDE with the minus sign and $u_2$ solves (classically) on $A$ the PDE with the plus sign.
\end{Prop}

On the connected components of $A$ that are in the interior of $\{u_1=u_2\}$ the function $u_1=u_2$ solves the PDE with $g=0$ thanks to (\textbf{T}), so there is nothing to prove in this case. 

\begin{proof}
Let $p=0$. The Euler-Lagrange equation in Proposition \ref{Prop:reg1} tell us that on any connected component of $A$ each function $u_1$ and $u_2$ (with $u_1 \leq u_2$, $u_1(0)=u_2(0)$ and $Du_1(0)=Du_2(0)=0$ --- on each connected component of $A$ each graph is embedded and admits only two orientations) is $C^2$ and solves one of the following PDEs (to simplify the discussion we consider the case that $N$ is Euclidean, the same proof adapts to the general case by considering the PDE (\ref{eq:E-L_Riem}))
\begin{equation}
 \label{eq:PDEsfortwographs}
 \text{div}\left( \frac{Du}{\sqrt{1+|Du|^2}}\right) =+  g(\cdot, u)  \,\,\, \text{ or }\,\,\,\text{div}\left( \frac{Du}{\sqrt{1+|Du|^2}}\right) =-  g(\cdot, u).
\end{equation} 
We analyse, case by case, according to the four possibilities for the choices of the signs on the right-hand side of the equations, the PDE for the difference $v=u_1-u_2$ on any other connected component of $A$: we conclude that, if the right-hand-sides have the same sign, then by the mean value theorem $v$ solves a PDE of the type

\begin{equation}
 \label{eq:differencemaxprinc}
 D_i\left(\left(\delta_{ij}+b_{ij}(Du_1, Du_2)\right) D_j v\right)=c v,
\end{equation}
where $b_{ij}$ is smooth and $0$ at $0$ and $c$ is smooth (more precisely, either $c=D_{n+1} g (\xi)$ or $c=-D_{n+1} g (\xi)$, with $\xi$ smooth). For $y$ in the chosen connected component of $A$, consider a ball $B_r(y)$ contained in the chosen connected component and such that there is a point $z$ on the boundary of the ball that belongs to $\{u_1=u_2\}$ (this can be done by choosing any ball and enlarging the radius until the boundary touches $\{u_1=u_2\}$). Observe that $v< 0$ on $B_r(y)$ and it extends to $z$ with $v(z)=0$ and $Dv(z)=0$ (because $u_1(z)=u_2(z)$ and $Du_1(z) = Du_2(z)$ by the structural assumptions): this contradicts Hopf boundary point lemma (in its version that holds regardless of the sign of the $0$-th order term, see the comments that follow \cite[Theorem 3.5]{GT}). Therefore the right-hand-sides in (\ref{eq:PDEsfortwographs}) have opposite signs for $u_1$ and $u_2$: however, for the moment, which function carries which sign might depend on the connected component of $A$. 

If $g$ takes both negative and positive values then assumption (b$^T$) forces the conclusion (we comment further on this in Appendix \ref{gchangessign}). We therefore assume for the remainder of the proof that $g\geq 0$.

We continue the previous analysis and we will now rule out the possibility that the right-hand-side is $g(\cdot, u_1)$ for $u_1$ and $-g(\cdot, u_2)$ for $u_2$. If that were the case we would get, again writing the PDE for $v$, 
\begin{equation}
 \label{eq:differencemaxprinc2}
 D_i\left(\left(\delta_{ij}+b_{ij}(Du_1, Du_2)\right) D_j v\right)\geq 0,
\end{equation}
where we used $g\geq 0$; working on $B_r(y)$ as above, the conditions $v<0$, $v(z)=0$ and $Dv(z)=0$ contradict Hopf boundary point lemma for subsolutions of elliptic PDEs without $0$-th order term, see \cite[Theorem 3.5]{GT}.

The only possibility that is allowed is therefore that the right-hand-sides in (\ref{eq:PDEsfortwographs}) are $-g(\cdot, u_1)$ for $u_1$ and $g(\cdot, u_2)$ for $u_2$, which geometrically means that the mean curvature vector points downwards for the bottom graph and upward for the top graph wherever it is non-zero.
\end{proof}

We continue the analysis of Proposition \ref{Prop:one-sided-max} at a touching singularity $p\in \SingT{V} \cap \Greg{V}$.

\begin{Prop}[\textbf{\textbf{local stationarity of $\Greg{V}$}}]
\label{Prop:loc_station_greg}
Under the assumptions of Proposition \ref{Prop:one-sided-max} and with the further constraint that $\spt{V} \cap \mathcal{C} \subset \Greg{V}$, upon possibly making $\mathcal{C}$ smaller we have that the two $C^2$ graphs of $u_1$ and $u_2$ are separately $J_g$-stationary; in the case $g\geq 0$, we moreover have that $u_1$ has mean curvature $-g \nu$ and $u_2$ has mean curvature $g \nu$, where $\nu$ denotes the upward orientation of each graph.
\end{Prop}

\begin{proof}
Les us consider first the case $g(p)>0$, which we can assume to be true in the neighbourhood $\mathcal{C}$. Then $\{u_1=u_2\}$ has $0$-measure by (a3), in particular empty interior. The $C^2$ regularity of $u_j$ implies that the mean curvature of $\text{graph}(u_j)$ is continuous for $j=1,2$; moreover $g$ is continuous on each graph. By Proposition \ref{Prop:one-sided-max} the mean curvature is $g\nu$ for $\text{graph}(u_2) \cap \left(\{u_1\neq u_2\}\times (-\sigma, \sigma)\right)$ and $-g \nu$ for $\text{graph}(u_1) \cap \left(\{u_1\neq u_2\}\times (-\sigma, \sigma)\right)$, so the mean curvature of each graph extends by continuity to the whole graph. In other words the two graphs are separately stationary.

Let us consider now the case $g\geq 0$ in $\mathcal{C}$ and $p\in \SingT{V} \cap \Greg{V}$ with $g(p)=0$ and with (a3) valid on $\mathcal{C}$. This time $\{u_1=u_2\}$ may have non-empty interior $I$ and by (a3) $\text{graph}(\left.u_j\right|_I)  \subset \{g=0\}$. Away from $I$, the previous argument extends to give that the mean curvature is $g\nu$ for $\text{graph}(u_2) \cap \left(\overline{\{u_1\neq u_2\}}\times (-\sigma, \sigma)\right)$ and $-g \nu$ for $\text{graph}(u_1) \cap \left(\overline{\{u_1\neq u_2\}}\times (-\sigma, \sigma)\right)$. Since on $\text{graph}(\left.u_j\right|_I)$ we have $g=0$ the continuity of $g$ and of the mean curvature of each graph gives the conclusion: the top graph with the upward orientation $\nu$ has mean curvature $\vec{H}=g \nu$ \textit{everywhere} and similarly for the bottom graph.
Equivalently, each disk is stationary with respect to $J_g$ for the choice of orientation just discribed. A local structure like the one in Figure 
\ref{fig:touching_ruledout} is ruled out by assumption (\textbf{T}).

\begin{figure}[h]
\centering
 \includegraphics[width=4cm]{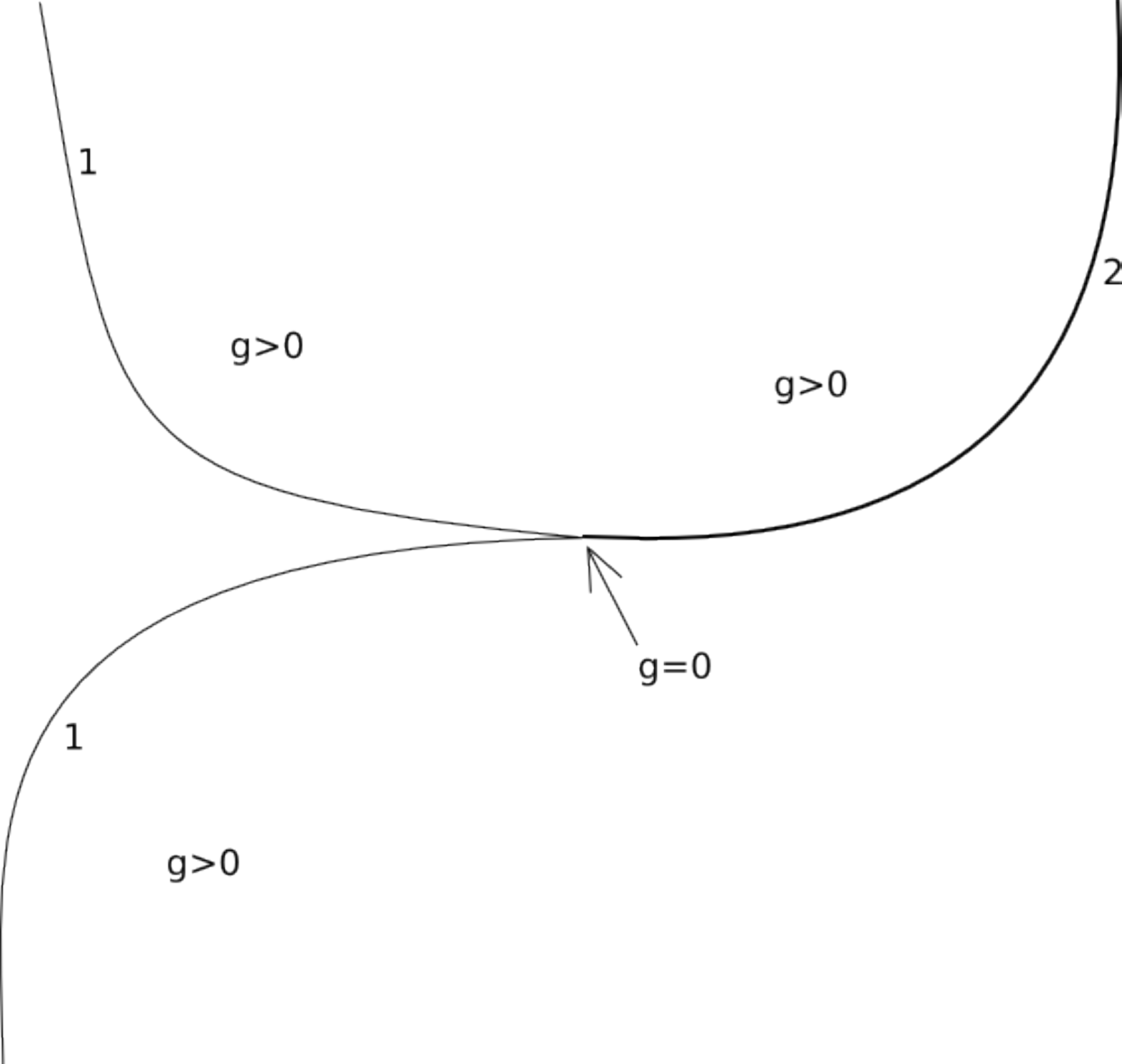} 
\caption{\small Assumption (\textbf{T}) rules out this local structure for $\spt{V}$ (in the picture $g$ has an isolated $0$). The two graphs that describe $\spt{V}$ in a neighbourhood of a touching singularity are allowed to coincide on a set of positive measure only if that portion is minimal ($g=0$).}
 \label{fig:touching_ruledout}
\end{figure}
The case in which $g$ can attain both positive and negative values follows immediately from (b$^T$).

\end{proof}

\begin{oss}
The result in Proposition \ref{Prop:loc_station_greg} is only valid locally. Note, to this end, that we are not assuming the validity of (b*) in Proposition \ref{Prop:loc_station_greg}, so we are allowing the varifold in Figure \ref{fig:immersion_high_mult} for which the conclusion of Proposition \ref{Prop:loc_station_greg} is valid locally but fails in larger open sets (consider for example the varifold obtained by removing the segment with multiplicity $3$ and the curve with multiplicity $1$ at the right hand of the picture: assumption (\textbf{T}) fails on such a large open set and so does the conclusion of Proposition \ref{Prop:loc_station_greg}).
\end{oss}

\begin{oss}
For $p\in\SingT{V} \cap \Greg{V}$ with $g(p)>0$, analysing the Hessian of $v=u_1-u_2$ in the above proof, there must exist an index $\ell \in \{1, ..., n\}$ such that $D^2_{\ell \ell}v(0) \neq 0$ (because of the non-vanishing of the mean curvature, $g(0)\neq 0$). The implicit function theorem then gives that the set $\{D_\ell v =0\}$ is, locally around $0$, an $(n-1)$-dimensional submanifold. The set of points $\{u_1 = u_2\}$ is contained in the set $\{Du_1 = Du_2\}$ (since there are no classical singularities by (a2)) and therefore $\{u_1 =u_2\} \subset \{Dv =0\} \subset \{D_\ell v =0\}$, a submanifold of dimension $(n-1)$. This implies in particular the the set $(\Greg{V} \setminus \Reg{V})$ has locally finite $\mathcal{H}^{n-1}$-dimensional measure.
\end{oss}

\begin{oss} 
Note that the proofs above also imply that $\ell$-fold touching singularities with $\ell \geq 3$  are not possible when the $\ell$ ordered graphs are $C^2$. Indeed,if that were the case, we could apply the previous reasoning to each possible couple of graphs that describe $\spt{V}$ in a neighbourhood of $p$ and we would find a contradiction for at least one couple. In particular, the smoothly immersed part of $\text{spt}\,V$ under assumptions (a), (a'), (b) agrees with $\Greg{V}$. We will draw an even stronger conclusion within the proof of the regularity theorems, ruling out $\ell$-fold touching singularities for $\spt{V}$ altogether.
\end{oss}

\begin{oss}\label{minimal-touching}
If $p \in \Greg{V} \setminus \Reg{V}$, then by definition of touching singularity  there is a neighbourhood of $p$ in which $\spt{V}$ consists of two $C^{2}$ disks $D_{1}$, $D_{2}$ intersecting only tangentially and with $p \in D_{1} \cap D_{2}$. Then for each $j$, the set $\widetilde{D}_{j} = \{g \neq 0\} \cap D_{j}$ (open in $D_{j}$) is non-empty and $p \in \partial \, \widetilde{D}_{j}$. To see this, choose local coordinates centered at $p$ with $T_{p} \, V = {\mathbb R}^{n} \times \{0\},$ and write $D_{j} = {\rm graph} \, u_{j}$ with $u_{j}  \in C^{2}(\Omega)$ for some open set $\Omega \subset {\mathbb R}^{n}$.  By Proposition~\ref{Prop:loc_station_greg}, each $u_{j}$ solves one of the two PDE's in (\ref{eq:PDEsfortwographs}) and not both solve the same one. Suppose contrary to the claim that $g = 0$ in a neighbourhood of $p$ in one of the disks, say $D_{1}.$ Then in a neighbourhood of the origin $u_{1}$, we may take either $g(x, u_{1}(x))$ or $-g(x, u_{1}(x))$ on the right hand side of the PDE that $u_{1}$ solves (trivially, since $g(x, u_{1}(x))=0$ near the origin). The claim then follows from the argument of Proposition~\ref{Prop:one-sided-max}.     
\end{oss}

\section{$\Greg{V}$ is an immersed oriented hypersurface}
\label{orientability}

The second variation assumption (c) implies, as seen in Section \ref{stabilitydiscussion}, that for any point $p \in {\rm spt} \, \|V\|,$ there is a small ambient neighborhood $U_{p}$ containing $p$ such that any orientable $C^2$-immersed portion of $\Greg{V} \cap U_{p}$ satisfies the strong stability inequality. What is needed in the proof of our regularity results---which proceeds by induction on the multiplicity---is the use of the strong stability inequality on $\Greg{V} \cap B$ for any ball $B$ with $\overline{B} \subset U_{p}$ whenever  
$\text{dim}_{\Hc}\left((\Sing{V} \setminus \Greg{V}) \cap B\right) \leq n-7$. Thus one must know that for a fixed small ball $B$, $\Greg{V} \cap B$ is an immersed orientable hypersurface whenever $\text{dim}_{\Hc}\left((\Sing{V} \setminus \Greg{V}) \cap B\right) \leq n-7$. This is quite straightforward, as shown in Remark~\ref{oss:orient_gpos}, when $g>0$, i.e.~in Theorem~\ref{thm:mainreg_gpos}, and it is immediate by (b*) in Theorem~\ref{thm:mainregstationarityimmersedparts}. (In these two cases we may in fact take any open set $\subset\subset U_{p}$ in place of $B$). In Theorem~\ref{thm:mainreg_g0_version1} however it is a delicate question. We address this in the present section. 

\begin{thm}[$\Greg{V}$ is an oriented immersion]
\label{thm:orientabilityGreg}
Let $V$ be an integral $n$-varifold in a simply connected open set $B$ such that assumptions (a1)-(a2)-(a3)-(a4)-(b)-(b$^T$) of Theorem \ref{thm:mainreg_g0_version1} are satisfied; assume moreover that $\text{dim}\left((\spt{V} \setminus \Greg{V}) \cap B \right) \leq n-7$. Then there exists an oriented $n$-manifold $S_B$ and a $C^2$ immersion 
$\iota_B: S_B \to B$ such that $\iota_{B}(S_{B}) \subset B \setminus (\spt{V} \setminus \Greg{V});$ 
moreover, $(\iota_B)_\sharp (|S_B|) = V \res B$.
\end{thm}

We begin by showing that multiplicity higher than $2$ on $\Greg{V}$ can only happen on  ``removable'' minimal hypersurfaces.

\begin{lem} 
\label{lem:mult12}
Let $V$ be an integral varifold, in an open set $B$, that satisfies assumptions (a1)-(a2)-(a3)-(a4)-(b)-(b$^T$) of Theorem \ref{thm:mainreg_g0_version1}; assume moreover that the local regularity statement (i) is true in $B$. Then we have the following conclusions on the multiplicity:

\noindent $\bullet$ if $x\in \Reg_1\,V$ then $x$ has multiplicity $1$ or $2$, unless the connected component of $\Reg_1\,V$ that contains $x$ is a minimal hypersurface whose closure (in $B$) is disjoint from $\Greg{V} \setminus \Reg_1\,V$ (in which case the multiplicity of this minimal hypersurface can be an arbitrary integer);

\noindent $\bullet$ $\Greg{V} \setminus \Reg_1\,V$ has multiplicity $2$.

\end{lem}

\noindent
{\it Notation:} In the proof of this lemma and subsequently, the notation $\stackrel{\circ}{(\{g=0\} \cap \Reg_1\,V)}$ indicates  interior taken relative to the manifold ${\rm reg}_{1} \, V$. 
\begin{proof}
 
Recall that the multiplicity is an integer on $\Greg{V}$, see \cite[Lemma A.2]{BW}, and the multiplicity is a constant integer on any connected component of $\Reg_1\,V$, see \cite[Lemma A.1]{BW}. Let $p \in \Greg{V} \setminus \Reg_1\,V$: we will prove that $\Theta(\|V\|,p)=2$. By definition, there esist two $C^2$ disks $D_1$ and $D_2$ that describe $\spt{V}$ in an open neighbourhood $B_r(p) \subset B$, $D_1$ and $D_2$ can only intersect tangentially and $D_1\neq D_2$. Consider the two open sets $A_j = D_{j} \setminus D_{1} \cap D_{2}$, $j=1,2$; then $A_1$ and $A_2$ are $C^2$ embedded portions of $V$ and the multiplicity is constant on each of them, we denote it by $Q_j \in \N$. Note that it is not possible that on one of the two disks $g$ vanishes identically (otherwise we would contradict that one-sided maximum principle, since if $g$ vanishes on $D_j$ we can write the PDE for the function defining $D_j$ both with $g$ or $-g$ on the r-h-s). Both $D_1$ and $D_2$ contain therefore an open set on which $g\neq 0$ and this open set must intersect the respective $A_j$ non trivially, since the set $\{D_1=D_2\} \cap \{g\neq 0\}$ has empty interior (it has $0$ measure); we denote by $\tilde{A}_j=A_j \cap \{g\neq 0\}$ and note that $\tilde{A}_j$, for $j=1,2$, is an embedded portion of $V$, with the multiplicity equal on one hand equal to $Q_j$, on the other hand equal to $1$ by assumption. We therefore conclude that $A_j$ have multiplicity $1$ as well for $j=1,2$. This forces the multiplicity to be $2$ on the set $\{D_1=D_2\}$ (e.g.~by \cite[Lemma A.2]{BW}), in particular it is $2$ on $p$.

For the case $x\in \Reg_1\,V$, in a neighbourhood of $x$ we have that $V$ is embedded, with constant multiplicity: note that if $x \in \Reg_1\,V  \setminus \stackrel{\circ}{(\{g=0\} \cap \Reg_1\,V)}$ then in any neighbourhood of $x$ there are points in $\Reg_1\,V$ with $g\neq 0$, where the multiplicity is $1$ by assumption, therefore the multiplicity is also $1$ at $x$. The only case to analyse is therefore $x \in \stackrel{\circ}{(\{g=0\} \cap \Reg_1\,V)}$. Consider  the connected component $Z$ of $\Reg_1\,V$ that contains $x$ and denote with $Q_x\in \N$ the constant multiplicity on $Z$. If there exists a point in $Z$ where $g\neq 0$ then $Q_x=1$, so assume this is not the case, and assume further that $Z$ is disconnected from $\Greg{V} \setminus \Reg_1\,V$. Then whenever $p_j \to p \in B$, $p_j \in Z$, $p \notin Z$, we must have that $p \in \Sing{V} \setminus \Greg{V}$, which is a set of codimension $7$ by assumption, i.e.~$Z$ is a minimal hypersurface in $B$ whose closure in $B$ satisfies $\text{dim}_{\Hc}\left(\overline{Z}\setminus Z\right)\leq n-7$, in other words $Z$ is a minimal hypersurface in $B$ with possibly a codimension-$7$ singular set. This is the only case, allowed in the lemma, in which the multiplicity can be an arbitrary integer.

The only case left to analyse is the one in which there exists a sequence $p_j \to p \in B$, $p_j \in Z$, $p \in \Greg{V} \setminus \Reg_1\,V$. We thus have, for the touching singularity $p$, two $C^2$ disks $D_1$ and $D_2$ that describe $\spt{V}$ in a neighbourhood $B_r(p) \subset B$, $D_1$ and $D_2$ can only intersect tangentially and $D_1\neq D_2$. The same argument used above, implies that the multiplicity is $1$ on $\{D_1\neq D_2\}$ and $2$ on $\{D_1=D_2\}$, in particular $2$ at $p$. By upper-semicontinuity, the multiplicity of $p_j$ (and of $x$) is either $1$ or $2$.
\end{proof}

\begin{oss}\label{discard}
Our final goal is to show that $\Greg{V} \cap B$, under the assumptions of Theorem~\ref{thm:orientabilityGreg} and when $B$ is simply connected (a ball being the only case that we will need), is an oriented immersed hypersurface (of class $C^{2}$). For this purpose, any connected component $W$ of $\Reg_1{V} \cap B$ with $\overline{W} \cap (\Greg{V}
\setminus \Reg_1\,V)=\emptyset$ can be discarded. To see this, note first that 
$\overline{W} \setminus W \subset {\rm spt} \, \|V\| \setminus \Greg V$, and since 
${\rm dim}_{\mathcal H} \, ({\rm spt} \, \|V\| \setminus \Greg V) \leq n-7$, it follows that 
the multiplicity 1 varifold associated with $W$ is stationary with respect to $J_{g}.$ Furthermore, $W$ is clearly the image of an immersion and it is orientable by \cite{Samelson} (a codimension-$7$ singular set $\overline{W} \setminus W$ in $B$ does not affect the argument). Thus, until the conclusion of the proof of Proposition~\ref{Prop:orientabilityimmersion}, we assume that such connected 
components of $\Reg_{1} V$ have been removed, and hence, in view of  Lemma~\ref{lem:mult12}, that the following holds: 
\end{oss}

\noindent{\bf Working assumption:} \textit{$\Greg{V}$ has multiplicity $1$ or $2$ everywhere in $B.$} 

\medskip

\noindent
{\it Notation:} In the remainder of this section, we shall use the notation 
$$M_{2} = \stackrel{\circ}{(\{g=0\} \cap \Reg_1\,V)} \cap \{\Theta=2\};$$
  $$M_{1} = \stackrel{\circ}{(\{g=0\} \cap \Reg_1\,V)} \cap \{\Theta=1\};$$
$$U = \{g \neq 0\};$$
$${\Oc}_{1}=(B \setminus \overline{M_1}) \setminus (\spt{V} \setminus \Greg{V}) \;\; {\rm and}$$
$${\Oc}_{2}=(B \setminus \overline{M_2}) \setminus (\spt{V} \setminus \Greg{V}).$$

\begin{lem}
 \label{lem:immersion1} Let $V$ be an integral $n$-varifold in an open set $B$ that satisfies assumptions (a1)-(a2)-(a3)-(a4)-(b)-(b$^T$) of Theorem~\ref{thm:mainreg_g0_version1} and suppose that the working assumption above holds. Then there exists an $n$-dimensional manifold $S_{{\Oc}_{2}}$ and a $C^{2}$ immersion $\iota_{{\Oc}_{2}} \, : \, S_{{\Oc}_{2}} \to {\Oc}_{2}$ with $\iota_{{\Oc}_{2}}(S_{{\Oc}_{2}}) = \Greg{V} \setminus \overline{M_2}.$ 
 \end{lem}

\begin{proof}

Consider the (abstract) manifold $S:=\Reg_1\,{V} \cap \Oc$, and denote with $\iota_S:S\to \Reg_1\,{V} \cap \Oc$ the immersion (embedding at this stage) into $\Oc$ given by the inclusion map. We denote with $\Gamma$ the following class of paths:
\begin{eqnarray*}
\Gamma:=\left\{\gamma:[0,1]\to \Greg{V} \cap ((B \setminus M_{2}) \setminus ({\rm Sing \, V} \setminus {\Greg \, V}))),\; \gamma \in C^0([0,1]),\right.&&\\ 
&&\hspace{-5in}\left. \;\gamma(0), \;\gamma(1) \in \Reg_1\,{V} \cap \Oc \text{ and if}\;\; \gamma(s)\in \Greg{V} \setminus \Reg_1\,V\right. \;\; \mbox{then}\\
&&\hspace{-5in}\mbox{there exists} \;\; \delta>0 \;\; \mbox{s.t.}\;\; \gamma((s-\delta, s+\delta)) \;\; \mbox{is  contained in one of the}\\
&&\hspace{-5in}\mbox{two disks that describe}\;\; \Greg{V}\;\; \mbox{in a neighbourhood of} \;\;\left.\gamma(s) \right\}.
\end{eqnarray*}

We define the following function $d:S \times S \to [0,\infty]$:

$$d(x,y)=\text{inf}\left\{\text{length}(\gamma): \gamma \in \Gamma,  \gamma(0)=\iota_S(x), \gamma(1)=\iota_S(y) \right\}.$$

The function $d$ is a metric on $S$. (If $x=y$ then clearly $d(x,y)=0$; if $d(x,y)=0$ then, since the ambient distance $|\iota_S(x)-\iota_S(x)|$ is always $\leq d(x,y)$  we get $\iota_S(x)=\iota_S(y)$ and since $\iota_S$ is an embedding of $S$ we also get $x=y$. The definition of $d$ is symmetric in $x$ and $y$, so $d(x,y)=d(y,x)$. For $x,y,z \in S$, given any path in $\Gamma$ joining $\iota_S(x)$ to $\iota_S(y)$ and any path in $\Gamma$ joining $\iota_S(y)$ to $\iota_S(z)$, the composition path joins $\iota_S(x)$ to $\iota_S(z)$ and is still in $\Gamma$, so the triangular inequality is satisfied.)

The metric space $(S, d)$ will now be ``completed with respect to $\Oc$'' to yield a metric space $(S_{\Oc}, d)$. This means that we consider all Cauchy sequences $(p_j)$ in $(S, d)$ such that $\lim_{j \to \infty} \, \iota_S(p_j) \in \Oc$ (note that it is automatic that $\lim_{j \to \infty} \, \iota_{S}(p_{j}) \in B$ since $\iota_{S}$ is distance decreasing), and let $S_{\Oc}$ be the collection of equivalence classes of such Cauchy sequences under the relation $(p_j) \sim (q_j)$ if  $d(p_j, q_j)\to 0,$ equipped with the metric $d((p_{j}), (q_{j})) = \lim_{j \to \infty} \, d(p_{j}, q_{j})$ (this is the usual completion for metric spaces but we are not including points whose image lie in the boundary of $\Oc$).

We will show that $S_{\Oc}$ is a manifold. For any point that was already in $S$, then there is a neighbourhood in which $S$ and $S_{\Oc}$ are isometrically identified and said neighbourhood is a disk, so we only need to consider the points that have been added in the completion. Let $p$ be any such point, and let $p_j$ be a Cauchy sequence representing it. Being a Cauchy sequence for $d$ implies also that $\iota_S(p_j)$ is a Cauchy sequence for the ambient distance and therefore $\iota_S(p_j) \to p_a$ for the ambient distance for some point $p_a$. Since $\Greg{V}$ is closed in $\Oc$ by assumption, $p_a \in \Greg{V} \cap \cap \Oc$ and $p_a$ cannot be in $\Reg_1\,V$ (since we are assuming that $p$ is a point that has been added), so $p_a \in \Greg{V} \setminus \Reg_1\,V$. For any two points $a,b \in \Reg_1\,V$ that are close to a touching singularity but on distinct disks, we have $d(a,b)>c>0$ for some $c$, therefore the Cauchy sequence $p_j$ is all contained in one of the two disks (say $D$) that describe $\Greg{V} \cap B_r(p_a)$. A neighbourhood of $p$ in the completion is given by all possible limits of Cauchy sequences $x_j$ such that, for $j$ large enough, $d(x_j,p_j)<c/2$. In particular $x_j$ is in the same disk as $p_j$. A neighbourhood of $p$ in the completion is therefore $D$ (because on $D$ the distance $d$ is identified with the intrinsic induced distance and $D \setminus \Reg_1\,V$ has empty interior for it).

The proof also shows that $\iota_S:S \to \Oc$ extends by continuity (since $D \setminus \Reg_1\,V$ has empty interior) to give an immersion with image $\Greg{V} \cap \Oc$, that is injective except for the set that maps to $\Greg{V} \setminus \Reg_1\,V$, which is covered twice (once for each disk). The map is actually $C^2$, because by definition $\Greg{V}$ is $C^{2}$.

To see that the immersion just constructed is proper, let $K$ be a compact set in $\Oc$ and note that, by the choice of $\Oc$, $\Greg{V} \cap K$ is compact in $B$. By continuity of the norm of the second fundamental form of the immersion, there exists $\sigma_K>0$ such that, whenever $x \in\Greg{V} \cap K$ and $r\leq \sigma_K$, $B_{r}^{n+1}(x) \cap \spt{V}$ is either an embedded disk or the union of two embedded disks intersecting tangentially; in particular, the inverse image of $B_{r}^{n+1}(x)$ via the immersion is made up of either one (topological) disk or two disjoint (topological) disks. The $C^2$ regularity further implies that the diameter of each such disk with respect to the metric $d$ is at most $C_K r$, where $C_K$ only depends on $K$ and on the bound for the $C^2$ norm of the second fundamental form of $\Greg{V} \cap K$. Given $\eps>0$ such that $\eps/{C_K}<\sigma_K$, by choosing a finite collection of balls of radius $\eps/{C_K}$  that covers $\Greg{V} \cap K$, their inverse images provide a finite collection of open sets that covers the inverse image of $K$, each open set having diameter smaller than $\eps$. This proves that the inverse image of $K$ is totally bounded. The fact that it is complete follows easily by construction. This establishes properness.
\end{proof}

\begin{oss}
 \label{oss:Gregimmersedemptyinterior}
 The previous proof shows an alternative way to write $\Greg{V}$ as a $C^2$ immersion under the assumptions of Theorem \ref{thm:mainreg_gpos} (where $M_2$ is empty), without exploiting the stationarity condition (which was done, on the other hand, in Remarks \ref{oss:immersionGreg_gpos}, \ref{oss:orient_gpos}).
\end{oss}

\begin{oss}
\label{oss:metric_clos_bdry}
In Lemma \ref{lem:immersion1} we defined the metric completion of $\Reg_1\,V$ with respect to the open set $(B \setminus \overline{M_2}) \setminus (\Sing{V} \setminus \Greg{V})$: this meant that we only considered Cauchy sequences (with respect to $d$) with image (via the inclusion map) converging to a point in $(B \setminus \overline{M_2}) \setminus (\Sing{V} \setminus \Greg{V})$ with respect to the ambient distance. (Note that the inclusion map is distance descreasing). The same argument can be used to define the metric completion of $\Reg_1\,V$ with respect to the larger set $(B \setminus M_{2})  \setminus (\Sing{V} \setminus \Greg{V})$; this means that we consider Cauchy sequences (with respect to $d$) with image (under $\iota_{S}$) converging to a point in $(B \setminus M_{2})  \setminus (\Sing{V} \setminus \Greg{V})$. Let $Y$ be this metric space. 
$Y$ need not be a manifold, nor a manifold-with-boundary (because we do not have any regularity assumption for the boundary of the open set $\{g\neq 0\}$, nor transversality of this boundary to $V$). The immersion defined in Lemma \ref{lem:immersion1} extends to $Y$ by continuity; we denote by $\iota_Y$ this extension. 
\end{oss}

Recall that $\Greg{V} \cap U$ is easily seen to be an \textit{oriented} $C^2$-immersed hypersurface by considering the unit vector $\frac{\vec{H}}{g}$ (see Remark \ref{oss:orient_gpos}). We have:

\begin{lem}[Orientability of $\Greg{V}$ in the complement of $\stackrel{\circ}{(\Reg_1\,V \cap \{g=0\})}$]
\label{lem:orientability}
Let $V$ be an integral varifold in an open set $B$ that satisfies assumptions (a1)-(a2)-(a3)-(a4)-(b)-(b$^T$) of Theorem \ref{thm:mainreg_g0_version1} and  $Y$ be the metric space defined in Remark \ref{oss:metric_clos_bdry}. Let $S_{U} = \iota^{-1}_{{\Oc}_{2}}(U)$ and $\iota_{U} = 
\iota_{{\Oc}_{2}} \res S_{U}.$ Then the manifold $S_U$ is oriented by $\nu=\frac{\vec{H}}{g}$ 
and $\iota_{U}(S_{U}) = \Greg{V} \cap U$. (Note that $S_U$ is an open set in the $n$-manifold $S_{{\Oc}_{2}}$, which in turn is a subset of $Y$.) Moreover, $\nu$ extends continuously to $\overline{S_{U}}$, where the closure is taken in $Y$. \end{lem}

\begin{proof}
Let $p \in \overline{S_{U}}$. We wish to show that there  exists a unit vector $\nu(p)$ in $T_{\iota_Y(p)} N$ such that $\nu(p_\ell) \to \nu(p)$ whenever $p_\ell \to p$ with $p_\ell \in S_U$.
Note that it is then automatic that $\nu(p)$ is orthogonal to $T_{\iota_{Y}(p)} V.$
We distinguish two cases: 

(i) $\iota_Y(p) \in \Reg_1\,{V}$; 

(ii) $\iota_Y(p) \in (\Greg{V} \setminus \Reg_1\,{V})$.

(i) In this case $p$ is actually a point in the manifold defined in Lemma \ref{lem:immersion1}. Since $\iota_Y(p)$ is an embedded point, we can choose $\sigma$ small such that $\spt{V} \cap B^{n+1}_\sigma(\iota_Y(p))=\Reg_1\, V \cap B^{n+1}_\sigma(\iota_Y(p))$ is embedded and orientable and denote this portion by $R$; choose the orientation $\tilde{\nu}$ of $R$ for which the stationarity condition $\vec{H}=g \tilde{\nu}$ holds (this exists by assumption (b)). Then $\tilde{\nu}$ agrees with the orientation $\nu=\frac{\vec{H}}{g}$ on $S_U$ and moreover, $p$ has a neighborhood in $Y$ that is mapped by $\iota_{Y}$ into $R.$ Therefore $\nu$ is extendable by continuity to $p$ (a varifold as the one in Figure \ref{fig:Hchangingsign} below is ruled out by hypothesis (b)).

(ii) If $\iota_Y(p)$ is a touching singularity, denote by $D_1$ and $D_2$ the ordered $C^2$ disks that describe $\Greg{V}$ in a neighbourhood of $\iota_Y(p)$ in a reference frame with respect to $T_{\iota_{Y}(p)} V$. By Remark~\ref{minimal-touching}, for each $j=1, 2$, there is a non-empty open set $A_j \subset D_j$ on which $g\neq 0$ and $\iota_Y(p) \in \overline{A_j}$, where the closure is taken in $D_j$. By the stationarity assumptions (b) and (b$^T$), on each $D_j$ the mean curvature is given by $\vec{H}=g \tilde{\nu}_{j}$, where $\tilde{\nu}_{j}$ is a choice of orientation of $D_j$. The orientation $\nu$ agrees with $\tilde{\nu}_{j}$ on $\iota_Y^{-1}(A_j)$ for $j=1, 2$. Moreover, by construction of $Y$ and $\iota_{Y}$,  $p$ has a neighborhood in $Y$ that is mapped by $\iota_{Y}$ into $D_{j}$ for precisely one $j \in \{1, 2\}$. Therefore 
$\tilde{\nu}_{j}$ provides the continuous extension of $\nu$ at $p$. 
\end{proof}

\begin{figure}[h]
\centering
 \includegraphics[width=5cm]{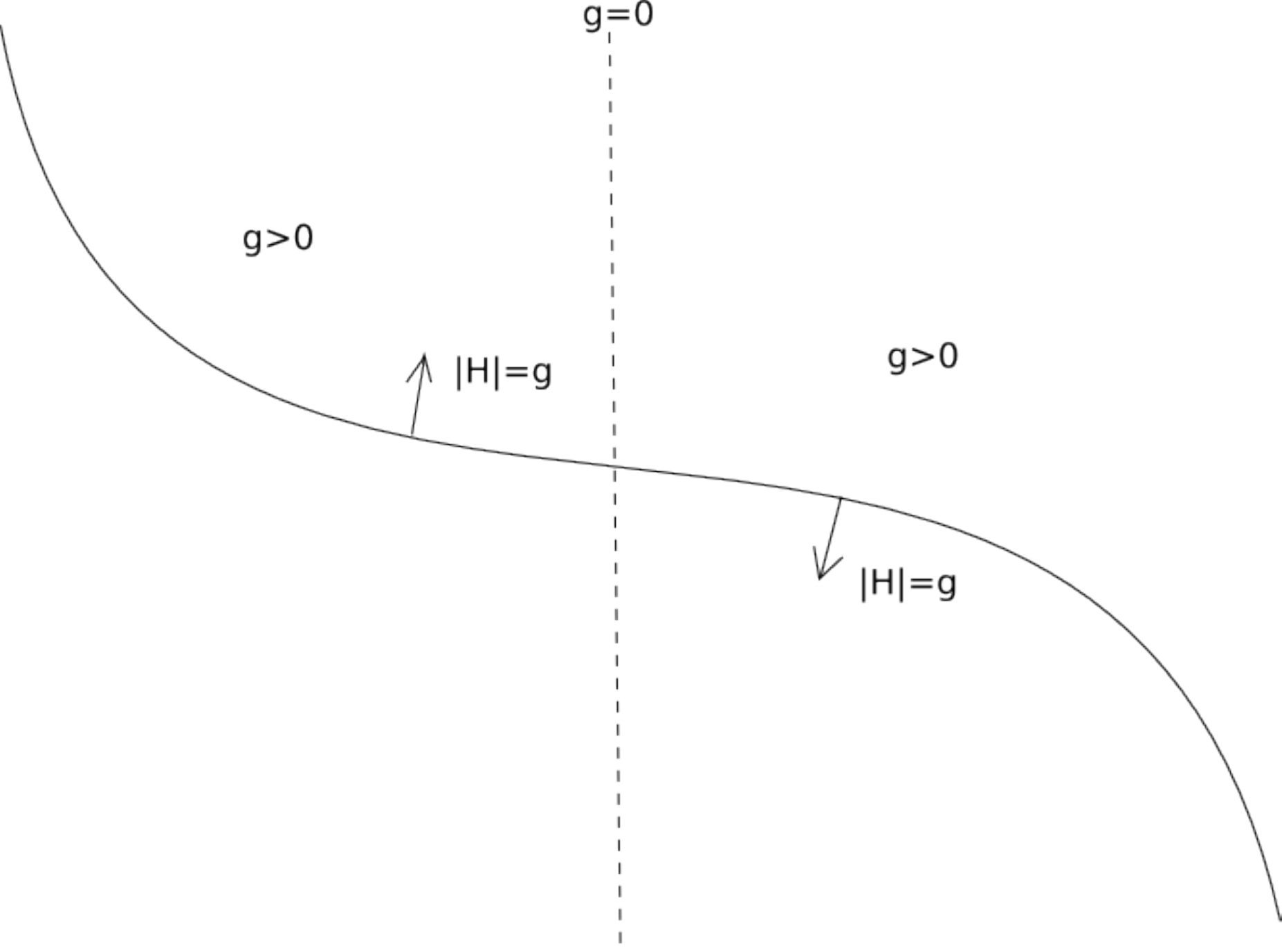} 
\caption{\small In the picture we take $g\geq 0$. The orientable embedded hypersurface depicted fulfils the condition that $|\vec{H}|=g$ but fails to satisfy the stationarity condition with respect to $J_g$. As shown in this section, $\Reg_1\,V$ (actually all of $\Greg{V}$) can be oriented away from $\stackrel{\circ}{(\Reg_1\,V \cap \{g=0\})}$ by extending $\frac{\vec{H}}{g}$ by continuity.}
\label{fig:Hchangingsign}
\end{figure}

\begin{oss}
Note that, in case (ii), the fact that $D_1$ and $D_2$ both contain a non-empty open set $A_j \subset D_j$ ($j=1,2$) on which $g\neq 0$ (this is true in any small neighbourhood of $\iota_Y(p)$) implies that there are two distinct points in $Y$ that are mapped to the same point $\iota_Y(p)$.
\end{oss}

\begin{oss}
From Lemma \ref{lem:orientability} we conclude, in particular, that when $\{g=0\} \cap \Greg{V}$ has empty interior then $\Greg{V}$ is the image of an orientable immersion and we can choose an orientation that agrees with $\frac{\vec{H}}{g}$ on $\{g\neq 0\}$. This property allows us to state Theorem \ref{thm:mainreg_emptyinterior} below, where we drop (a4) (comparing with Theorem \ref{thm:mainreg_g0_version1}).
\end{oss}
     
\begin{Prop}[Orientability of $\Greg{V}$ globally in a ball, assuming smallness of the ``pure'' singular set]
\label{Prop:orientabilityimmersion}
Let $V$ be an integral $n$-varifold in a simply connected open set $B$ such that assumptions (a1)-(a2)-(a3)-(a4)-(b)-(b$^T$) of Theorem \ref{thm:mainreg_g0_version1} are satisfied; assume moreover that $\text{dim}\left((\spt{V} \setminus \Greg{V}) \cap B \right) \leq n-7$. 
Then there exists a manifold $S_B$ and an immersion $\iota \, : \, S_{B} \to B$ such that $\iota_{\#} \, (|S_{B}|) = V \res B$ and $\iota(S_{B}) = \Greg{V} \cap B$; moreover $S_B$ is orientable and the orientation can be chosen so that it agrees with $\frac{\vec{H}}{g}$ on $\{g\neq 0\}$.
\end{Prop}

\begin{proof}

\textit{Step 0: reminders and notation}. Recall that $M_1 = \stackrel{\circ}{(\{g=0\}\cap \Reg_1\,V)} \cap \{\Theta=1\}$ and $M_2 = \stackrel{\circ}{(\{g=0\} \cap \Reg_1\,V)} \cap \{\Theta=2\}$.  Let 
$E_{2} = \iota_{{\Oc}_{2}}^{-1} (B \setminus \overline{M}_{1} \setminus \overline{M}_{2} \setminus ({\rm spt} \, \|V\| \setminus {\Greg} V)).$ Then $E_{2}$ is an open subset of the manifold $S_{{\Oc}_{2}}$ with $\iota_{{\Oc}_{2}}(E_{2}) = \Greg V \setminus \overline{M_{1}} \setminus \overline{M_{2}}.$ Clearly $S_{U} \subset E_{2}$, and it follows from Remark~\ref{minimal-touching} that 
$E_{2} \subset \overline{S}_{U}$ (closure taken in $Y$). Hence by Lemma~\ref{lem:orientability} $E_{2}$ is oriented (by the orientation on $\overline{S}_{U}$) and the orientation extends continuously to $\overline{E_{2}}.$  

\textit{Step 1: existence of an oriented immersion with image $\Greg V \setminus \overline{M_{1}}$}. Next, we wish to show that $G=\Greg V \setminus \overline{M_{1}}$ is the image of an immersion of an oriented $n$-dimensional oriented submanifold $S_{{\Oc}_{1}}$ of the unit sphere bundle over $B$. The working assumption implies $G= G_1 \cup G_2$, where $G_1=\{x \in G: \Theta(\|V\|, x)=1\}$ and $G_2=\{x \in G: \Theta(\|V\|, x)=2\}$. The set $G_1$ is open in $G$ by upper-semi-continuity of the density and it is an embedded manifold oriented by $\nu$, because $G_1 \subset \iota_{{\Oc}_{2}}(E)$. The set $G_2$ contains $M_2$ and $G_2 \setminus M_2 \subset \Greg{V} \setminus \Reg_1\,V$; for any $x\in G_2$ we denote by $\pm n(x)$ the two unit normals at $x$ to $T_x V$. We define $S_{{\Oc}_{1}} \subset {\Oc}_{1} \times S^n$ by
$$S_{{\Oc}_{1}}=\{(x,\nu(x)): x \in G_1\} \cup \{(x,n(x)), (x, -n(x)):x \in G_2\}.$$
In order to show the manifold structure of $S_{{\Oc}_{1}}$ we need to prove that any point in $S_{{\Oc}_{1}}$ admits a neighbourhood in ${\Oc}_{1} \times S^n$ whose intersection with $S_{{\Oc}_{1}}$ is an $n$-dimensional disk (the validity of the separation axiom and the countable basis axiom are immediate because they hold in the ambient sphere bundle, that induces the topology on $S_{{\Oc}_{1}}$). This property will follow upon proving that for every $y \in G$ there exists a ball $B^{n+1}_r(y)$ such that $S_{{\Oc}_{1}} \cap \left(B^{n+1}_r(y) \times S^{n}\right)$ is either a single embedded disk or the union of two \textit{disjoint} embedded disks. If $y \in G_1$ then we choose $B^{n+1}_r(y)$ such that $G \cap B^{n+1}_r(y) \subset G_1$ and is a disk; then $S_{{\Oc}_{1}} \cap \left(B^{n+1}_r(y) \times S^{n}\right)$ is a single embedded disk. If $y \in M_2$ then we choose $B^{n+1}_r(y)$ such that $G \cap B^{n+1}_r(y) \subset M_2$ and is a disk; then $S_{{\Oc}_{1}} \cap \left(B^{n+1}_r(y) \times S^{n}\right)$ is the union of two disjoint embedded disk. In the remaning case $y\in G_2 \setminus M_2 \subset \Greg{V} \setminus \Reg_1\,V$ the point $y$ is a touching singularity ad we choose $B^{n+1}_r(y)$ so that $\spt{V} \cap B^{n+1}_r(y)$ is the union of two embedded disks $D_1, D_2$ that intersect only tangentially, with $y\in D_1 \cap D_2$, with $D_1 \setminus D_2, D_2 \setminus D_1 \subset G_1$ non-empty and orientations $\nu_j$ on $D_j$ such that $\vec{H} = g \nu_j$ on $D_j$ (by Proposition~\ref{Prop:loc_station_greg}); on $G_{1} \cap D_{j}$, the orientation $\nu_j$ agrees with the orientation $\nu$ defined on $G_1$ (by Remark \ref{minimal-touching} there exist open sets $A_j \subset D_j$ with $y \in \overline{A}_j$, for $j=1,2$, in which $g\neq 0$ and $\nu$ was defined by continuously extending $\frac{\vec{H}}{g}$). Moreover, $\nu_1 = -\nu_2$ on $D_1\cap D_2$ (again by Proposition~\ref{Prop:loc_station_greg}). This means that $S_{{\Oc}_{1}} \cap \left(B^{n+1}_r(y) \times S^{n}\right)$ agrees with the union of the two sets 
$$\{(x,v):x \in D_1, v= \nu_1(x)\} \cup \{(x,v):x \in D_2, v= -\nu_1(x)=\nu_2(x)\},$$
which is the union two disjoint embedded disks (respectively the oriented lifts of $D_1$ and $D_2$). This establishes that $S_{{\Oc}_{1}}$ is an $n$-dimensional submanifold of the unit sphere bundle over $B$. The orientability of $S_{{\Oc}_{1}}$ follows by considering the projection onto the second factor $S^n.$ The immersion $\iota_{{\Oc}_{1}}:S_{{\Oc}_{1}} \to {\Oc}_{1}$ with image $\Greg V \setminus \overline{M_{1}}$ is defined by restricting to $S_{{\Oc}_{1}}$ the projection onto ${\Oc}_{1}$.

\textit{Step 2: existence of an immersion with image $\Greg V \setminus (\p M_1 \cap \p M_2)$}. We have shown, in particular, that $\Greg{V} \setminus \overline{M_1}$ is an immersed hypersurface; we proved in Lemma \ref{lem:immersion1} that $\Greg{V} \setminus \overline{M_2}$ is an immersed hypersurface. The open sets $E_{1} = \iota_{{\Oc}_{1}}^{-1}\left({\Oc}_{1} \setminus \overline{M}_2\right)$ and $E_{2} \subset S_{{\Oc}_{2}}$ are mapped respectively by $\iota_{{\Oc}_{1}}$ and $\iota_{{\Oc}_{2}}$ onto the same set. Moreover, for each $j \in \{1, 2\}$, the immersion $\iota_{{\Oc}_{j}}$ is injective on $\Omega_{j} = E_{j} \setminus F_{j},$ where $F_{j}  =\iota_{\Oc}^{-1}(K_{j})$ (a closed subset of $E_{j}$ with empty interior) with $K_{j} \subset \Greg{V} \setminus \Reg_1\,V$. Then $\iota_{{\Oc}_{2}}^{-1}\circ \iota_{{\Oc}_{1}} \, : \, \Omega_{1} \to \Omega_{2}$ is a bijection, which extends to a diffeomorphism $\psi \, : \, E_{1} \to E_{2}$. (For each $x \in E_{1}$, there is a small ball $B_{r}(x) \subset E_{1}$ and an open set $U_{x} \subset E_{2}$ such that $\left.\iota_{{\Oc}_{1}}\right|_{B_{r}(x)}$ and $\left.\iota_{{\Oc}_{2}}\right|_{U_{x}}$ are embeddings with the same image. Define $\Psi \, : \, E_{1} \to E_{2}$ by setting $\left.\Psi\right|_{B_{r}(x)}  = \left.\iota_{{\Oc}_{2}}\right|_{U_{x}}^{-1} \circ \left.\iota_{{\Oc}_{1}}\right|_{B_{r}(x)}.$ On $B_{r}(x) \setminus F_{1}$, $\Psi$ agrees with the bijection $\iota_{{\Oc}_{2}}^{-1} \circ \iota_{{\Oc}_{1}}$, so $\Psi$ is well defined on $E_{1}$ and is the unique continuous extension of $\iota_{{\Oc}_{2}}^{-1} \circ \iota_{{\Oc}_{1}}.$ It is clear directly from the definition that $\Psi$ is a local diffeomorphism and it is is injective on $E_{1}.$) We get a new manifold $S_{{\Oc}_{1}} \cup_\psi S_{{\Oc}_{2}}$ defined by the gluing of $S_{{\Oc}_{1}}$ and $S_{{\Oc}_{2}}$ along the open sets $E_1$ and $E_{2}$ using the identification $\psi$. In order to prove that $S_{{\Oc}_{1}} \cup_\psi S_{{\Oc}_{2}}$ is a manifold we need to check the topological $T_2$-axiom (separation). 

To check the $T_2$ axiom, we only need to consider points on the boundary of $E_{1}$ and $E_{2}$ (the glued open sets), so let $p_1 \in S_{{\Oc}_{1}}$ be in $\p E_1$ and $p_2 \in S_{{\Oc}_{2}}$ be in $\p E_{2}$. Note that $\p E_1$ and $\p E_{2}$ (respectively in $ S_{{\Oc}_{1}}$ and $ S_{{\Oc}_{2}}$) are mapped (respectively by $\iota_{{\Oc}_{1}}$ and $\iota_{{\Oc}_{2}}$) to $\overline{M}_2 \setminus \overline{M}_{1}$ and $\overline{M}_1 \setminus \overline{M}_{2}$ (by definition of $E_1$ and $E_{2}$), and therefore the images (respectively under $\iota_{{\Oc}_{1}}$ and $\iota_{{\Oc}_{2}}$) of $p_1$ and $p_2$ are distinct points in $B \setminus (\spt{V} \setminus \Greg{V})$. We can then select disjoint open balls $B_1, B_2 \subset B \setminus (\spt{V} \setminus \Greg{V})$ centred at the two images; the open sets $\iota_{{\Oc}_{1}}^{-1}(B_1)$ and $\iota_{{\Oc}_{2}}^{-1}(B_2)$ have the property that $\psi(\iota_{{\Oc}_{1}}^{-1}(B_1) \cap E_{1})$ is disjoint from $\iota_{{\Oc}_{2}}^{-1}(B_2) \cap E_{2}$. This guarantees the validity of the $T_2$-axiom. 

We now have a $C^2$ immersion $\iota_{{\Oc}_{1} \cup {\Oc}_{2}}$ on $S_{{\Oc}_{1}} \cup_\psi S_{{\Oc}_{2}}$ (by gluing $\iota_{{\Oc}_{1}}$ and $\iota_{{\Oc}_{2}}$ along $E_{2}$ and $E_1$ via $\psi$) whose image is $\Greg V \setminus (\p M_1 \cap \p M_2)$. We need to extend this immersion so that it also covers $\p M_1 \cap \p M_2$ (step 3 below) and then, finally, we need to check (step 4 below) the orientability of this immersion; this will complete the proof of the proposition.

\textit{Step 3: existence of an immersion with image $\Greg V$}. We will ``complete'' the manifold $S_{{\Oc}_{1}} \cup_\psi S_{{\Oc}_{2}}$ to include points that map onto $\p M_1 \cap \p M_2 \setminus (\spt{V} \setminus \Greg{V})$. Let $y \in \p M_1 \cap \p M_2 \setminus (\spt{V} \setminus \Greg{V}).$ Then $y$ is a touching singularity and there exists an ambient ball $B(y)$, centred at $y$, in which $\spt{V}$ agrees with the union of two distinct embedded $C^2$ disks $D_1$ and $D_2$ that intersect only tangentially, with each lying  on one side of the other and with $y \in D_1 \cap D_2$.  By Remark \ref{minimal-touching}, $y$ is in the closure of $\Reg_1\,V \cap \{g \neq 0\} \cap D_{j}$ for each $j$, and therefore we can choose orientations $\nu_j$ of $D_j$, for $j=1,2$, so that $\nu_j$ agrees with $\nu=\frac{\vec{H}}{g}$ on $\Reg_1\,V \cap \{g \neq 0\} \cap D_j.$ Moreover, $D_j$ with the orientation $\nu_j$ is stationary with respect to $J_g$. Let $A = \cup_y B(y) \subset B \setminus (\spt{V} \setminus \Greg{V})$; we claim that $A \cap \Greg{V}$ is an oriented immersed manifold. 

To achieve this, we will show that the orientations chosen in each ball for the two disks are coherent, i.e.~whenever disks from two distinct balls overlap, the orientations match on the overlap. Upon choosing a small enough ball around each point, we can ensure that each disk is geodesically convex, and also that each disk is a graph over a subset of the tangent plane at any of its points. 

Let $z, w \in \p M_1 \cap \p M_2 \setminus (\spt{V} \setminus \Greg{V})$ be distinct points, and let $D_z$ (resp.\ $D_{w}$) be one of the two disks corresponding to $z$ (resp.\ $w$) with $D_z \cap D_w \neq \emptyset$. Note that each point in  
$D_{z} \cap D_{w} \cap \{\Theta =2\}$  lifts to two distinct points in the sphere bundle. So we only have to show that at every point in $T \equiv D_{z} \cap D_{w} \cap \{\Theta =1\}$, the two orientations on $D_{z}$ and $D_{w}$ agree (so that that point lifts unambiguously to a point in the sphere bundle). 

So assume $T \neq \emptyset$, and note that $T$ is an open subset of both $D_{z}$ and $D_{w}$.  Let $\widetilde{T}$ be a connected component of $T$. Since $\widetilde{T} \neq D_{z}$ (else 
$z \in T$ contrary to $\Theta \, (\|V\|, z) = 2$), we have that $\partial \, \widetilde{T} \cap D_{z} \neq \emptyset$. Similarly, $\partial \, \widetilde{T} \cap D_{w} \neq \emptyset$. We consider the two cases: 
\begin{itemize}
\item[(i)] $\Theta \, (\|V\|, y) = 1$ for some point $y \in \partial \, \widetilde{T} \cap D_{z}$ and $\Theta \, (\|V\|, y') = 1$ for some point $y' \in \partial \, \widetilde{T} \cap D_{w};$ 
\item[(ii)] $\Theta \, (\|V\|, y) = 2$ for every point $y \in \partial \, \widetilde{T} \cap D_{z}$ or $\Theta \, (\|V\|, y) = 2$ for every point $y \in \partial \, \widetilde{T} \cap D_{w}$.
\end{itemize}  
If case (i) holds, let $\gamma_{yz}(t)$ be the geodesic segment connecting $y$ to $z$ parameterized by $t \in [0, 1]$ with $\gamma_{yz}(0) = y$ and $\gamma_{yz}(1) = z,$ and let $t_{1}$ be the smallest value of $t \in [0, 1]$ with $\Theta (\|V\|, \gamma_{yz}(t_{1})) = 2$.   Let $\delta = \min_{t \in [0, 1]} {\rm dist} \, (\gamma_{yz}(t), \partial \, D_{z})$. Choose $t_{2} \in [0, t_{1})$ such that $d(\gamma_{yz}(t_{1}), \gamma_{yz}(t_{2}))< \delta/2$. Then there is a connected $C^{2}$ open subset $L_{z} \subset\subset D_{z}$ containing $\gamma_{yz}([0, t_{2}])$ such that $\Theta(\|V\|, x) = 1$ for each $x \in L_{z}$ and $\Theta(\|V\|, x_{z}) = 2$ for some point $x_{z} \in \partial \, L_{z}$. If $g = 0$ on $L_{z}$, this contradicts the Hopf boundary point lemma (by the argument of Proposition~\ref{Prop:loc_station_greg} and Remark~\ref{minimal-touching}), so there must exist a point $q_{z} \in L_{z}$ with $g(q_{z}) \neq 0.$   
By the same argument, there is a connected $C^{2}$ open subset $L_{w} \subset\subset D_{w}$ such that $g(q_{w}) \neq 0$ for some $q_{w} \in L_{w}.$  Let $\widetilde{L} = \widetilde{T} \cup L_{z} \cup L_{w}$. Since $\widetilde{L}$ is an open subset of $D_{z} \cup D_{w}$, it is a graph over the tangent plane to $V$ at any point in $T$ (by the way we chose the disks, $D_{z} \cup D_{w}$ is a graph over the tangent plane at any point in $D_{z} \cap D_{w}$), and thus it is orientable. So since $\widetilde{L} \subset {\rm reg}_{1} \, V$, we have by hypothesis that there exists a continuous unit normal $\vec{n}$ on $\widetilde{L}$ such that $\vec{H} = g \vec{n}$ on $\widetilde{L}$. Since we have an orientation $\nu_{z}$ on $D_{z}$ such that $\vec{H} = g \nu_{z}$ on $D_{z}$, it follows that $\vec{H}(q_{z}) = g(q_{z})\nu_{z}(q_{z}) = g(z) \vec{n}(q_{z})$ so $\vec{n}(q_{z}) = \nu_{z}(q_{z})$. Similarly, $\vec{n}(q_{w}) = \nu_{w}(q_{w}).$ 
Consequently, since both $\vec{n}$ and $\nu_{z}$ are continuous on the connected open set $L_{z} \cup \widetilde{T}$, we see that $\nu_{z} = \vec{n}$ on $L_{z} \cup \widetilde{T}$. 
Similarly, $\nu_{w} = \vec{n}$ on $L_{w} \cup \widetilde{T}$. So in particular $\nu_{z} = \nu_{w}$ on $\widetilde{T}.$ 

If case (ii) holds, pick a point $x_{0} \in \widetilde{T}.$ Then the small geodesic ball $B^{\prime} = B_{d(x_{0}, \partial \, \widetilde{T})}(x_{0})$ is $C^{2},$ 
$\Theta \, (\|V\|, x) = 1$ for every $x \in B^{\prime}$ and there is a touching singularity $x_{1} \in \partial \, B^{\prime}.$ Arguing again as in the preceding paragraph (using the Hopf boundary point lemma), we find a point $q \in B^{\prime}$ such that $g(q) \neq 0.$ Since $\vec{H} = g \nu_{z}$ on $D_{z}$ and  $\vec{H} = g \nu_{w}$ on $D_{w},$ evaluating at $q$ we conclude that 
$\nu_{z}(q) = \nu_{w}(q)$, and hence, since $\widetilde{T}$ is connected, $\nu_{z} = \nu_{w}$ on $\widetilde{T}.$ 

Since $\widetilde{T}$ is an arbitrary connected component of $T$, we have shown that $\nu_{z} = \nu_{w}$ on $T$. We thus have a well-defined orientation of $\Reg_1\,V$ in $A$ and we can argue as done in step 1, lifting $\Greg{V} \cap A$ to the sphere bundle over $A$ to get an embedded oriented $n$-submanifold $S_A$, whose projection $\iota_{A} \, : \, S_{A} \to A$  is a $C^2$ immersion with image $\Greg{V} \cap A$. We can then glue $S_A$ to the manifold $S_{{\Oc}_{1}} \cup_\psi S_{{\Oc}_{2}}$ along respectively the open sets $S_A \setminus \iota_A^{-1}(\p M_1 \cap \p M_2)$ and $\iota_{{\Oc}_{1} \cup {\Oc}_{2}}^{-1}(A)$ using an argument as the one in step 2. This produces an $n$-manifold, that we denote by $S_B$, and a $C^2$-immersion $\iota_B:S_B \to B \setminus (\spt{V} \setminus \Greg{V})$, whose image is $\Greg{V} \cap B$.

\textit{Step 4: orientability of the immersion}. 
To begin with we show, by a slight modification of an argument developed in \cite{Samelson}, the following:\\
\noindent
{\it Claim:} Whenever $A \subset B \setminus (\spt{V} \setminus \Reg_1\,V)$ is open and $\Reg_1\,V$ has multiplicity $1$ in $A$, then $\Reg_1\,V \cap A$ is orientable. 

If orientability fails, then we can find a closed curve $\gamma$ in $A$ that intersects the hypersurface $\Reg_1\,V \cap A$ at exactly one point, as explained in \cite{Samelson}. (Take an embedded closed curve $\sigma \, : \, [0, 1] \to \Reg{V} \cap A$ such that any continuous unit normal $\vec{n}(t),$ $t \in (0, 1]$, 
on $\sigma(0, 1]$ has no continuous extension to $0$. Fix such $\vec{n}$, and let $\widetilde{\gamma}(t) = \sigma(t) + \epsilon \vec{n}(t)$ for $t \in (0, 1]$ and suitably small constant $\epsilon >0$. 
Then $\widetilde{\gamma}(0,1]$ is disjoint from $\Greg{V}$. Complete $\widetilde{\gamma}$ to a continuous closed curve $\gamma$ by connecting the end points 
$\lim_{t \to 0^{+}} \widetilde{\gamma}(t)$ and  $\widetilde{\gamma}(1)$ with a geodesic segment.) Although $\gamma$ is not necessarily contractible in $A$, it is so in $B$, so there exists a homotopy of $\gamma$ to a point: we denote by $f:S^1 \times [0,1] \to B$ this homotopy;
without loss of generality we may make $f$ transverse to $\iota_B$ and with image disjoint from the codimension-$7$ singular set (by slightly perturbing it). By intersection theory mod 2 (or by analogous transversality arguments)
we find a contradiction, since we saw before that $\gamma \cap \Greg{V}$ ($= \gamma \cap {\rm reg}_{1} \, V$) consists of exactly one point  (in the present context, we are counting multiplicities, i.e.~we count intersections with the immersion $\iota_B$) while $f_1$, that maps to a point, can be made disjoint from $\Greg{V}$. So the claim is proved. 

What is left to prove is that $S_B$ is orientable and the orientation can be chosen so that it matches $\frac{\vec{H}}{g}$ on $\iota_B^{-1}(\{g\neq 0\})$. We know that $\iota_B^{-1}(\Greg{V}\setminus M_1)$ satisfies this orientability condition (by step 1 and by  Lemma~\ref{lem:orientability}). We will show that we can choose an orientation for $\iota_B^{-1}(M_1)$ (which is orientable by the claim above) that matches continuously the one on $\iota_B^{-1}(\Greg{V}\setminus M_1)$. 

Let $M$ be a connected component of $M_1$. We will use the notation $\p M = \overline{M} \setminus (\spt{V} \setminus \Greg{V}) \setminus M$. We only have to address the case when $\p M \cap \Greg{V} \neq \emptyset$. Let $\widetilde{M}$ be the largest connected component of ${\rm reg}_{1} \, V$ containing $M$. Since each $x \in \partial \, M \cap {\rm reg}_{1} \, V$ has a neighborhood in which ${\rm spt} \, \|V\|$ is $C^{1}$ embedded, it follows that $\partial \, M \cap \Reg_{1} \, {V} \subset \widetilde{M}.$ By the constancy lemma (\cite[Lemma A1]{BW}), $\Theta(\|V\|, \cdot) = 1$ on $\widetilde{M}$, and moreover $\widetilde{M}$ is orientable by the above claim. 

Next we show that $ \p M \cap \Reg_1\,{V}\neq \emptyset$ and that $\partial \, M$ is contained in the closure of  $\widetilde{M} \cap \{g \neq 0\}.$  
Let $y \in \partial \, M$. If $y \in \p M \cap \Reg_{1} \, V$, then it is clear that any neighborhood of $y$ in $\Reg_{1} \, V$ contains a point where $g \neq 0$, so $y$ belongs to the closure of 
$\widetilde{M} \cap \{g \neq 0\}$. If on the other hand $y \in \p M \cap (\Greg{V} \setminus \Reg_1\,V)$, i.e.~if $y$ is a touching singularity,\footnote{This would be prevented by Hopf lemma if the open set $\{g\neq 0\} \subset B$ satisfied an exterior ball condition at $y$, but this is not necessarily true.}  
choose a reference frame with respect to $T_y V$ and denote by $\pi$ the projection onto $T_y V$. In a cylindrical neighbourhood of $y$ of the form $U_y = B^n_y \times (-\sigma, \sigma)$ (where $B_y^n$ is a ball in $T_y V$ centred at $y$), $\spt{V}$ is the union of two $C^2$ disks $\underline{D}_y$ and $\overline{D}_y$ that intersect tangentially. Consider the projection $K$ of the touching set onto $B_y^n$. For any point $a\in B_y^n \cap (\pi(M) \setminus K)$ sufficiently close to $\pi(y)$ there is an open ball $B(a) \subset \subset B_{y}^{n}$ centred at $a$ that is disjoint from $K$ and such that there is a point $z \in \partial \, B(a) \cap K$. Then there is a touching singularity at $(z, z') \in U_{y}$ and we must have that both $\underline{D}_y \cap (B(a) \times (-\sigma,\sigma))$ and $\overline{D}_y  \cap (B(a) \times (-\sigma,\sigma))$ contain non-empty open sets where $g \neq 0$, otherwise we would contradict Hopf boundary point lemma (see the argument of Proposition~\ref{Prop:one-sided-max} and Remark~\ref{minimal-touching}). 
At least one of the two disks $\underline{D}_y \cap (B(a) \times (-\sigma,\sigma))$ and $\overline{D}_y  \cap (B(a) \times (-\sigma,\sigma))$ must intersects $M$, by the choice of $a$; denote this disk by $Z_{a}$.

At this stage we have established that $Z_a \cap \p M \neq \emptyset$, because $Z_a \cap M \neq \emptyset$ and $Z_a$ contains points where $g\neq 0$ ($g=0$ on $M$); moreover, $Z_a \subset \Reg_1\,V$ by construction. So the claim that $\p M \cap \Reg_1\,{V}\neq \emptyset$ is proved (and $Z_a \subset \widetilde{M}$). Note that $Z_a$ is as close to $y$ as we wish, which implies that $y$ is in the closure of $\widetilde{M} \cap \{g \neq 0\}$. 

By hypothesis (b), there is a choice of orientation that makes $\widetilde{M}$ (and hence $M$) stationary with respect to $J_g$, and this choice is unique since it has to agree with $\frac{\vec{H}}{g}$ on the non-empty set $\widetilde{M} \cap \{g \neq 0\}.$ 

It remains to check that the uniquely determined orientations on $\iota^{-1}_{B} (\widetilde{M})$ and $S_{B} \setminus \iota_{B}^{-1}(M_{1})$ agree on their intersection $\iota^{-1}_{B}(\widetilde{M} \setminus M_{1}).$ But this is clear now, since $\widetilde{M} \setminus M_{1}$ is in the closure of the set ${\rm reg}_{1} \, V \cap \{g \neq 0\},$ and by Lemma~\ref{lem:orientability}, the  orientation on 
$S_{B} \setminus \iota_{B}^{-1}(M_{1})$ is determined by the orientation on ${\rm reg}_{1} \, V \cap \{g \neq 0\}.$ Since $M$ was an arbitrary connected component of $M_{1}$, this completes the proof.
\end{proof}

\begin{proof}[Proof of Theorem~\ref{thm:orientabilityGreg}]
This is now immediate from Proposition~\ref{Prop:orientabilityimmersion} and Remark~\ref{discard}.  
\end{proof}

\begin{oss}(\textit{what if $B$ is not simply connected in Theorem~\ref{thm:orientabilityGreg}?})
\label{oss:O_not_simply_conn}
We have shown that, if $B$ is a simply connected open set and $\text{dim}_{\Hc}\left(\left(\overline{\Greg{V}} \setminus \Greg{V}\right) \cap B\right) \leq n-7$, then $\Greg{V} \cap B$ is an oriented immersion (and the orientation agrees with $\frac{\vec{H}}{g}$ wherever $g\neq 0$). We note that these two assumptions were used only in two places: in the claim in Step 4 of Proposition~\ref{Prop:orientabilityimmersion}, and in Remark \ref{discard}. In the first case it was used to deduce the orientability of  $M_{1},$ and in the second to deduce orientability of   any minimal hypersurface separated from $\Greg{V} \setminus \Reg_1\,V.$  

Consider now $\Oc=N \setminus (\Sing{V} \setminus \Greg{V})$ (where $N$ is the ambient manifold) and denote by $S_j$ any (embedded) minimal hypersurface that is separated from $\Greg{V} \setminus \Reg_1\,V$ and has constant multiplicity $j$, and by $M_1$ the multiplicity-$1$ minimal hypersurface not separated from $\Greg{V} \setminus \Reg_1\,V$ and by $M_2$ the multiplicity-$2$ minimal hypersurface not separated from $\Greg{V} \setminus \Reg_1\,V$. Then the arguments used in this section prove that $\Greg{V} \setminus (M_1 \cup_{j \text{odd}} S_j)$ is the image of a two-sided immersion. (As before, $M_2$ and any minimal hypersurfaces $S_j$ with $j$ even can be lifted to the sphere bundle as its double cover.) If we know a priori the two-sidedness of each $S_j$ for $j$ odd, and of $M_1$  (for example, this is the case in the situation addressed in \cite{BW2} by the results in \cite{RogTon}) then we can conclude, using the same arguments exploited in this section, that $\Greg{V}$ is globally the image of a two-sided immersion in $N$ (conclusion (ii) of Theorem \ref{thm:mainreg_g0_version1}) and the proofs given in this section can be shortened.
\end{oss}

\section{Local consequences of stability}
\label{stability-section}

\begin{oss}
\label{oss:assume_strongstab_withambient}
Let $\widetilde{V}$ be an integral $n$-varifold in $N$  such that the hypotheses of Theorem~\ref{thm:mainregstationarityimmersedparts} hold with $\widetilde{V}$ in place of $V$. Let 
$X_{0} \in {\rm spt} \, \|\widetilde{V}\|$, $\rho_{0} \in (0, R_{X_{0}})$  where $R_{X_{0}}>0$ is the injectivity radius at $X_{0}$, and let  $V={\text{exp}_{X_0}^{-1}}_\# \left(\tilde{V} \res \mathcal{N}_{\rho_{0}}(X_0)\right).$  Let $\Oc  = {\mathcal N}_{\rho_{0}}(X_{0}) \setminus \left(\spt{\widetilde{V}} \setminus \Greg{\widetilde{V}}\right)$. Then, in accordance with hypothesis (b*), 
$\widetilde{V} \res \Oc = \widetilde{\iota}_{\Oc \, \#} \, (|S_{\Oc}|)$  where $S_{\Oc}$ is an oriented $n$-manifold and $\widetilde{\iota}_{\Oc} \, : \, S_{\Oc} \to \Oc$ is a $J_{g}$-stationary $C^{2}$ immersion. In view of the fact that the regularity conclusion (i) of our theorems is of local nature and the fact that finite Morse index implies local stability, by taking $\rho_{0}$ sufficiently small we may, and we will, assume that $\widetilde{V} \res {\Oc}$  is stable with respect to $J_g$ for all normal deformations induced by an ambient test function $\phi \in C^{1}_{c}(\Oc)$ (i.e.\ deformations as in Defintion~\ref{df:index}), or equivalently (see Section~\ref{stabilitydiscussion}) that:  
$$\int \left(|A|^2 +\text{Ric}(\nu,\nu)-D_\nu g\right)\phi^2 d\|\widetilde{V}\| \leq \int |\nabla \phi|^2 d\|\widetilde{V}\|$$ 
for all $\phi \in C^{1}_{c}(\Oc)$,
where $A$ is the second fundamental form of the immersion $\iota_{\Oc}$ in the Riemannian ambient manifold, norms are computed with respect to the Riemannian metric and the gradient is the intrinsic one (on the hypersurface). 
\end{oss}

\subsection{Localization to a chart; Schoen's inequality}
\label{Schoenslemma}

Let $\widetilde{V}$, $V$, $X_{0} \in {\rm spt} \, \|\widetilde{V}\|$, $\Oc = {\mathcal N}_{\rho_{0}}(X_{0}) \setminus \left(\spt{\widetilde{V}} \setminus \Greg{\widetilde{V}}\right)$ and $\widetilde{\iota}_{\Oc}$ be as above. Let $\iota_{\Oc} \, : \, S_{\Oc} \to \exp_{X_{0}}^{-1}(\Oc) \subset B_{\rho_{0}}^{n+1}(0)$ be the immersion 
$\iota_{\Oc} = \exp_{X_{0}}^{-1} \circ \widetilde{\iota}_{\Oc}$.  Relabeling $\sqrt{|\mathscr{h}|}g \circ \exp_{X_{0}}$ as $g$ (as explained in Section~\ref{exponentialchart}), we then have that 
the $J_{g}$-stationarity of $\widetilde{\iota}_{\Oc}$ (as above) is equivalent to stationarity of $\iota_{\Oc}$ with respect to  
$$\mathcal{F}(t) =
\int_{\exp_{X_{0}}^{-1}(\Oc)} F(x,\nu(x)) \, d\|\psi_{t \, \#} \, |S_{\Oc}|\|+\text{Vol}_g(t),$$ 
(i.e.\ to the condition  ${\mathcal F}^{\prime}(0) = 0$) for all $C^{2}$ deformations $\psi \, : \, (-\epsilon, \epsilon) \times S_{\Oc} \to \exp^{-1}_{X_{0}}(\Oc)$ with $\psi(0, x) = \iota_{\Oc}(x)$ for $x \in S_{\Oc}$ and $\psi(t, x) = \iota_{\Oc}(x)$ for $(t, x) \in (-\epsilon, \epsilon) \times S_{\Oc} \setminus K$ for some compact $K \subset S_{\Oc}$.  Here $\psi_{t}(\cdot)  = \psi(t, \cdot)$ and $F$ is as in Section~\ref{exponentialchart}. We  note that 
$$\text{Vol}_g(t) = \int \vec{f}(x) \cdot \nu(x) \, d\|\psi_{t \, \#} \, |S_{\Oc}|\| - \int \vec{f}(x) \cdot \nu(x) \, d\|\iota_{\Oc \,  \#} \, |S_{\Oc}|\|$$  
for any $C^{2}$ vector field $\vec{f}$ satisfying $\text{div} \vec{f} = g$ in $B_{\rho_{0}}^{n+1}(0)$, and we can find such $\vec{f}$ by setting $\vec{f} = \nabla u$ where $u$ is the unique solution to    
$\Delta u= g$ in $B_{\rho_{0}}^{n+1}(0)$, $u= 0$ on $\partial \, B_{\rho_{0}}^{n+1}(0)$. Since $g \in C^{1, \alpha}(\overline{B_{\rho_{0}}^{n+1}(0))}$, it follows by standard elliptic theory that $\vec{f} \in C^{2, \alpha}(\overline{B_{\rho_{0}}^{n+1}(0))}$, with $\sup_{B_{\rho_{0}}^{n+1}(0)} \, (|\vec{f}| + \rho_{0}|D\vec{f}|+ \rho_{0}^{2}|D^{2} \, \vec{f}|) + \rho_{0}^{2 + \alpha}\sup_{x_{1}, x_{2} \in B_{\rho_{0}}^{n+1}(0), x_{1} \neq x_{2}} \, \frac{|D^{2} \, \vec{f}(x_{1}) - D^{2} \, \vec{f}(x_{2})|}{|x_{1} - x_{2}|^{\alpha}} \leq C \left(\rho_{0}\sup_{B_{\rho_{0}}^{n+1}(0)} |g| +  
\rho_{0}^{1+ \alpha}\sup_{x_{1}, x_{2} \in B_{\rho_{0}}^{n+1}(0), x_{1} \neq x_{2}} \, \frac{|g(x_{1}) - g(x_{2})|}{|x_{1} - x_{2}|^{\alpha}}\right)$ where $C = C(n)$. Fix this $\vec{f}$ and for any constant vector $\vec{c}$, set 
\begin{equation}\label{new-integrand}
\widetilde{F}_{\vec{c}}(x, p) = F(x, p) + (\vec{f}(x) + \vec{c})\cdot p.
\end{equation}
It follows then that 
\begin{eqnarray}\label{euclidean-stationarity}
\mbox{$\widetilde{\iota}_{\Oc}$ is $J_{g}$-stationary} \iff \mbox{for any constant vector $\vec{c}$, the immersion $\iota_{\Oc}$ is stationary}&& \nonumber\\ 
&&\hspace{-4in}\mbox{with respect to} \nonumber\\
&&\hspace{-5in}{\mathcal F}_{\vec{c}}(\iota_{\Oc}) = \int_{\exp_{X_{0}}^{-1}(\Oc)} \widetilde{F}_{\vec{c}}(x,\nu(x)) \, d\|\iota_{\Oc \, \#} \, |S_{\Oc}|\| \;\; \mbox{(i.e.\ $\left.\frac{d}{dt}\right|_{t=0} \, {\mathcal F}_{\vec{c}}(\psi_{t}) = 0.$)}  
\end{eqnarray}

Now by the same reasoning as in \cite{SS}, the integrand $F$ satisfies conditions \cite[(1.2)-(1.6)]{SS} for some constants $\overline{\mu}$, $\overline{\mu}_{1}$ (say) in place of $\mu$, $\mu_{1}$. In particular (by \cite[(1.6)]{SS}), for each point $y \in B_{\rho_{0}/2}^{n+1}(0)$, there is 
 a diffeomorphism $\eta_{y} \, : \, B_{\rho_{0}}^{n+1}(0) \to B_{\rho_{0}}^{n+1}(0)$ which takes $0$ to $y$ and transforms $F$ to $\eta_{y}^{\#}F$ (notation as in \cite{SS}) such that $\eta_{y}^{\#}F$ satisfies 
 conditions \cite[(1.2)-(1.5)]{SS} with $\overline{\mu}$, $\overline{\mu}_{1}$ (independent of $y$) in place of $\mu,$ $\mu_{1}$.\footnote{This map $\eta_{y}$ is called $\psi_{y}$ in \cite{SS}.} We may, and we shall, take $\eta_{0} = identity.$ From this it follows that for each $y \in B_{\rho_{0}/2}^{n+1}(0)$, there is a vector $\vec{c}(y)$ (namely, $\vec{c}(y) = -\vec{f}(y)$) such that 
 $\eta_{y}^{\#} \widetilde{F}_{\vec{c}(y)}$ satisfies \cite[(1.2)-(1.5)]{SS} taken with $\eta_{y}^{\#} \widetilde{F}_{\vec{c}(y)}$ in place of $F$, again with a fixed choice of $\mu$, $\mu_{1}$ independent of $y$ (but depending on $\overline{\mu}$, $\overline{\mu}_{1}$, $g$, $\rho_{0}$ and $\alpha$).\footnote{Note that $\widetilde{F}_{\vec{c}}$ is in $C^{2, \alpha}$ in the $x$ variables and is $C^{3}$ in the $p$ variables; although it is assumed in \cite{SS} that the integrand $\widetilde{F}_{\vec{c}}$  is of class $C^{3}$, the proofs therein require less regularity. What is needed in fact are \cite[(1.3)]{SS} and \cite[(1.4)]{SS}, each of which requires at least one derivative of $\widetilde{F}_{\vec{c}}$ in the $\nu$ variable, and hence at most 2 derivatives of $f$; thus our assumption that $f \in C^{2, \alpha}$ suffices to satisfy the regularity conditions on $\widetilde{F}_{\vec{c}}$ needed in the proofs of \cite{SS}.}

Similar reasoning applies with regard to the stability hypothesis. Thus, the $J_{g}$-stability of $\widetilde{\iota}_{\Oc}$ is equivalent to stability of $\iota_{\Oc}$ with respect to ${\mathcal F}_{\vec{c}}(\iota_{\Oc})$ for any $\vec{c}$, i.e.\ to the condition 
\begin{equation}\label{euclidean-stability}
\left.\frac{d^{2}}{dt^{2}}\right|_{t=0}{\mathcal F}_{\vec{c}}(\psi_{t}) \geq 0
\end{equation}
for any $\vec{c}$ and all deformations $$\psi(t, x) = \exp_{X_{0}}^{-1}\circ \exp_{\widetilde{\iota}(x)} \circ\left(\widetilde{\iota}_{\Oc}(x) +  t (\phi \circ \exp_{X_{0}}^{-1} \circ \widetilde{\iota}_{\Oc}(x))\nu_{\widetilde{\iota}_{\Oc}(x)}\right), \;\; (x, t) \in S_{\Oc} \times (-\epsilon, \epsilon)$$ 
with $\phi \in C^1_c(\text{exp}_{X_{0}}^{-1}(\Oc))$. Computing this second derivative  and following the arguments in \cite{SS} leads to the validity of the inequality \cite[(1.17)]{SS}:

\begin{equation}
 \label{eq:SS_Strongstab_witherrors}
\int  |A|^2  \phi^2 d\|V\| \leq \int |\nabla \phi|^2 d\|V\| + 
\end{equation}
$$+c_5 \mu_1 \int (\mu_1 \phi^2 + \phi |\nabla \phi| + \phi |A| + |x| |\nabla \phi|^2 + |x|\phi^2 |A|^2 + \mu_1 \phi^2 |x|^2 |A|^2) d\|V\|$$ 
 for all $\phi \in C^1_c(\text{exp}^{-1}(\Oc)),$ 
where $A$ is the second fundamental form of the immersion $\iota_{\Oc}$, norms are Euclidean, $\nabla$  is the intrinsic gradient of $\iota_{\Oc}$, $c_5$ is a dimensional constant and $\mu_1$ is the constant fixed as in the preceding discussion. 

\medskip
 
\textit{Adaptation of Schoen's lemma.} The validity of (\ref{eq:SS_Strongstab_witherrors})  implies the following: there exists $c_1, \epsilon_{0}>0$ (depending only on $n$, $\mu$, $\mu_{1}\rho_{0}$) such that for any $\rho$ with $\mu_{1}\rho \leq \epsilon_{0}$ and 
$\varphi \in C^1_c(B_{\rho}^{n+1}(0) \setminus (\spt{V} \setminus \Greg{V})),$ 
\begin{equation}
 \label{eqSchoenlemma}
 \int_M |A|^2 \varphi^2 d \|V\| \leq c_1 \int_M (1-(\nu \cdot e^{n+1})^2) |\nabla^M \varphi|^2+ c_1 \mu_1^2 \int_{M}  \varphi^2 d  \|V\|,
\end{equation}
where the quantities are relative to an immersion in $\Oc$ and $\nabla^M$ is the intrinsic gradient. 
The proof of (\ref{eqSchoenlemma}) is by the argument of \cite[Lemma 1]{SS} incorporating the modifications described in \cite[proof of Theorem 3.4]{BW}. 

Within the proof of our regularity theorem we will need to use (\ref{eqSchoenlemma}) on the hypersurface $\Greg{V}$ in open balls in which the singular set $\spt{V} \setminus \Greg{V}$ is known to have codimension $7$ or higher.

\subsection{Strong stability inequality for graphs and a preliminary sheeting theorem}
The stability hypothesis enters the proof of our regularity results primarily via applications of the preliminary sheeting theorem (Theorem \ref{thm:SSsheeting_second}) below. This result is very much in the spirit of \cite[Theorem 1]{SS}, in that it assumes a priori smallness of the (genuine)  singular set (hypothesis (iii)).

We point out that Theorem \ref{thm:SSsheeting_second} has non-variational hypotheses, i.e.~the varifold is not assumed to be a critical point of a geometric functional. For its application in the present work (namely, within the proof of Theorem \ref{thm:mainregstationarityimmersedparts}, or more specifically within the proof of Theorem \ref{thm:sheeting} below), it would suffice to prove a ``variational counterpart'' of Theorem \ref{thm:SSsheeting_second} (in the style of \cite{SS}, or of the analogous result \cite[Theorem 3.4]{BW}). 
More precisely, for the purposes of the present paper it would suffice to prove Theorem \ref{thm:SSsheeting_second} under the additional variational assumptions on $\Greg{V}$ that are available in Theorem \ref{thm:sheeting} (namely, conditions (I), (II) in Definition \ref{class_of_varifolds}).
However, we observe that the arguments in \cite{SS} (that need to be followed closely to prove Theorem \ref{thm:SSsheeting_second}, and to which we will refer below) only make use of the variational hypotheses insofar as they imply hypotheses (v) and (vi) of Theorem \ref{thm:SSsheeting_second}. We therefore drop, in Theorem \ref{thm:SSsheeting_second}, any stationarity and stability requirements on $\Greg{V}$ in favour of (the weaker) hypotheses (v) and (vi).
While not needed in this generality for the present work, Theorem \ref{thm:SSsheeting_second} does not require any additional effort to prove compared to its variational counterpart. More importantly, Theorem \ref{thm:SSsheeting_second} is applicable to non-variational settings as well, 
such as the one arising in our work \cite{BW2} addressing the question of existence of prescribed mean curvature hypersurfaces in compact Riemannian manifolds. Eventhough such a hypersurface has the variational characterisation as a critical point of $J_{g}$ where $g$ is the prescribing function for the mean curvature, the 
Allen--Cahn approximation scheme employed in that work produces, in the first instance, a limit varifold $V$ which may develop stationary (i.e.\ zero generalised mean curvature) portions in addition to mean curvature $g$ parts; this may happen, as far as anyone knows, even when $g$ is a non-zero constant. If this happens then clearly $V$ will not be variationally characterrised. 
The approach taken in \cite{BW2} is to prove regularity of $V$ nonetheless, and then argue that the mean curvature $g$ portion can be extracted from $V$ by virtue of regularity of $V$.

We shall comment further on the hypotheses of  Theorem \ref{thm:SSsheeting_second}, and the applicability of Theorem~\ref{thm:SSsheeting_second} within the proof of Theorem \ref{thm:sheeting}, in Remarks \ref{oss:SS_1}, \ref{oss:SS_1'} and \ref{oss:SS_2} below.

\begin{thm}[Sheeting Theorem \`{a} la Schoen--Simon]
\label{thm:SSsheeting_second}
 Given $q$ a positive integer, $p > n$, and numbers $\mu$, $\mu_{1}$, $c_3$, $\rho_{0}$, there exists $\eps = \eps(n, q, \mu, \mu_{1}\rho_{0}, p, c_3) \in (0, 1/2)$ such that the following holds:  
 Let $V$ be an integral $n$-varifold in $B_{\rho_{0}}^{n+1}(0)$ such that:

\noindent (i) there is an oriented $n$-manifold $S_{V}$ and a $C^2$ immersion $\iota_{V} \, : \, S_{V} \to B_{1}^{n+1}(0)$ such that $\iota_{V}(S_{V}) = \Greg{V}$ and $\iota_{V \, \#} \, (|S_{V}|) = V \res \left(B_{\rho_{0}}^{n+1}(0)\setminus \left(\spt{V}\setminus \Greg{V}\right)\right);$ 

\noindent (ii) for each $y \in \Greg{V} \cap \text{sing}_T\, V$ and any neighborhood $\Oc$ of $y$ with $\spt{V}\cap \Oc = {\rm graph} \, u_{1} \cup {\rm graph} \, u_{2}$ for $C^{2}$ functions $u_1 \leq u_2$ and with $\iota_{V}^{-1}(\Oc)$ equal to a finite number of (topological) disks\footnote{The existence of such a neighbourhood $\Oc$ follows from the definitions of $\Greg{V}$ and $\text{sing}_T\, V$}, we have that the image under $\iota_V$ of each of these disks is equal to either ${\rm graph} \, u_{1}$ or ${\rm graph} \, u_{2}$;

\noindent (iii) $\text{sing} \, V \setminus \text{sing}_T\,  V  = \emptyset$ in case $n < 6$, $\text{sing} V \setminus \text{sing}_T \, V$ is discrete in case $n=7$ or 
${\rm dim}_{\mathcal H} \, (\text{sing} V \setminus \text{sing}_T V) \leq n-7$ in case $n\geq 8$; moreover, $\text{sing}_T\,  V \subset \Greg{V}$;

\noindent (iv) condition (a1$^{\prime}$) of Section \ref{exponentialchart} holds;

\noindent (v) the scalar mean curvature $h$ of the immersion $\iota_{V}$ satisfies 
$$|h(x)|\leq c_3(\mu_1 |\iota_V(x)| |A(x)| + \mu_1)$$
for $x\in S_V$, where $A(x)$ is the second fundamental form of $\iota_{V}$ at $x$ and $|\iota_V(x)|$ denotes the distance to the origin;

\noindent (vi) for $\mu_{1}\rho < \varepsilon$, inequality (\ref{eqSchoenlemma}) holds for all $\phi \in C^1_c(B^{n+1}_{\rho}(0) \setminus (\text{sing} V \setminus \text{sing}_T \, V))$, $\phi \geq 0$, with the ambient (i.e.\ ${\mathbb R}^{n+1}$) gradient $\nabla$ in place of $\nabla^{M}$;  

\noindent (vii)  

$$ E_{\rho} \equiv \rho^{-n-2}\int_{B_{\rho}^{n}(0) \times {\R} } |x^{n+1}|^{2} \, d\|V\|(X) + \left(\rho^{p-n}\int_{B_{\rho}^{n}(0) \times {\R} } |H_V|^{p} \, d\|V\|(X) \right)^{\frac{1}{p}}+\mu_1\rho < \eps$$
(where $H_V$ is the function as in condition (a1$^{\prime}$) for $Y=0$).

Then 
$$V \res \left(B_{\rho/2}^{n}(0) \times {\R} \right) = \sum_{j=1}^{q} |{\rm graph} \, u_{j}|,$$
where $u_{j} \in C^{2, \alpha} \, (B_{\rho/2}^{n}(0); {\mathbb R})$, $u_{1} \leq u_{2} \leq \ldots \leq u_{q}$ and for each $j$, 
\begin{eqnarray*}
&&\sup_{B_{\rho/2}^{n}(0)} \left(\rho^{-2}|u_{j}|^{2} + |Du_{j}|^2 + \rho^2|D^2 u_{j}|^2\right)\nonumber\\ 
&&\hspace{1.75in}+ \; \rho^{2+2\alpha}\sup_{x, y \in B^{n}_{\rho/2}(0), \, x \neq y} \frac{|D^2u_{j}(x) - D^2u_{j}(y)|^{2}}{|x - y|^{2\alpha}} \leq  C E_{\rho}
\end{eqnarray*} 
for some fixed constants $\alpha  = \alpha(n,p, q, \mu, \mu_{1}\rho_{0}) \in (0, 1/2),$ $C = C(n, p, q, \mu, \mu_{1}\rho_{0}) \in (0, \infty).$   
\end{thm}

\begin{oss}
\label{oss:SS_1}
The relevant case of Theorem \ref{thm:SSsheeting_second} in the present work is when $V$ is the pull back, via the exponential map, of an integral $n$-varifold $\tilde{V}$ in a Riemannian manifold. If the first variation of $\tilde{V}$ (with respect to the area fuctional induced by the Riemannian metric) is given by $-H_{\tilde{V}}\|\tilde{V}\|$ for a function $H_{\tilde{V}} \in L^p(\|\tilde{V}\|)$ with $p>n$ (the function $H_V$ is the generalised mean curvature of $\tilde{V}$), then hypothesis (iv) in Theorem \ref{thm:SSsheeting_second} holds for $V$, as explained in Section \ref{exponentialchart}. Moreover, if $\tilde{V}$ additionally satisfies that its $C^2$ immersed oriented portions have (classical) mean curvature prescribed by an ambient function $g$ of class $C^{1,\alpha}$ on the Riemannian manifold, then hypothesis (v) in Theorem \ref{thm:SSsheeting_second} holds true for $V$ (this follows by computations as in \cite[(1.16)]{SS} and only uses the fact that $g$ is bounded in $L^\infty$). Similarly, if $V$ satisfies (I) of Definition \ref{class_of_varifolds}, then $V$ also satisfies (v) in Theorem \ref{thm:SSsheeting_second}.
\end{oss}

\begin{oss}
\label{oss:SS_1'}
With notation as in Remark \ref{oss:SS_1}, if $\tilde{V}$ satisfies assumption (b*) of Theorem \ref{thm:mainregstationarityimmersedparts}, then $V$ satisfies (i) of Theorem \ref{thm:SSsheeting_second}. Moreover, if $\tilde{V}$ satisfies assumption (b$^T$) of Theorem \ref{thm:mainregstationarityimmersedparts}, then $V$ satisfies (ii) of Theorem \ref{thm:SSsheeting_second}. Similarly, if $V$ satisfies (I) of Definition \ref{class_of_varifolds}, then $V$ also satisfies (i) and (ii) of Theorem \ref{thm:SSsheeting_second}.
\end{oss}

\begin{oss}
\label{oss:SS_2}
We recall that (with notation as in Remark \ref{oss:SS_1}) the validity of hypothesis (vi) in Theorem \ref{thm:SSsheeting_second} is implied by the assumption that $\tilde{V}$ satisfies a stability hypothesis (Morse index equal to $0$ in the sense of Definition \ref{df:index}) on its $C^2$ immersed part. We note that if $V$ satisfies (II) of Definition \ref{class_of_varifolds} then (vi) in Theorem \ref{thm:SSsheeting_second} is satisfied by virtue of the discussion in Section \ref{Schoenslemma}.
\end{oss}

\begin{oss}
As already mentioned, Theorem \ref{thm:SSsheeting_second} will be used within the proof of the (more general) sheeting result Theorem \ref{thm:sheeting}. 
By virtue of Remarks \ref{oss:SS_1}, \ref{oss:SS_1'}, \ref{oss:SS_2}, any varifold $V$ as in Theorem \ref{thm:sheeting} below, that additionally satisfies an a priori bound on the (genuine) singular set (hypotheses (iii) of Theorem \ref{thm:SSsheeting_second}),
satisfies all the assumptions (and thus the conclusion) of Theorem \ref{thm:SSsheeting_second}.
The aim of Theorem \ref{thm:sheeting} is to drop assumption (iii) of Theorem \ref{thm:SSsheeting_second} in favour of the no-classical-singularities assumption and hypothesis ({\bf T}).
\end{oss}

\begin{oss}\label{SS-proof}
The proof of the preliminary sheeting theorem \ref{thm:SSsheeting_second} is achieved by suitably adapting the arguments in \cite{SS}. In order carry out the proof it is essential to ensure the validity of \cite[Lemma 1]{SS} for any graph $M$  that appears in the partial graph decomposition produced by following \cite{SS}. We achieve this by showing (in Lemma~\ref{graph-stability} below) that for any graph $M$ of class $C^2$ that satisfies the inequality in (v) of Theorem \ref{thm:SSsheeting_second}, it is automatically true that inequality \cite[(1.17)]{SS} holds (without an explicit strong stability assumption on $M$, from which the same inequality would immediately follow). This observation is the one that allows to state condition (II) in Theorem \ref{thm:regularity_stronglystable} (see the remark that follows it) with only ambient test functions and with the ambient gradient on the right-hand-side, rather than the intrinsic one. In fact, \cite{SS} uses \cite[Lemma 1]{SS} either with an ambient test function $\varphi$ (in which case the ambient gradient suffices in their arguments) or with a non-ambient test function that is supported on a single graph of the partial graph decomposition.
\end{oss}

\begin{lem}[\textbf{stability inequality for graphs}]
\label{graph-stability}
Let $\Omega \subset  B_{1/2}^{n}(0)$ be open, $f \, : \, \Omega \to [-1/2, 1/2]$ be of class $C^{2}$ and suppose that the mean curvature $H$ of $M=\text{graph}\,f$ satisfies
$$|H| \leq c\left(\,|(x, f(x))|\,|A|\, +\,1 \,\right),$$
where $c>0$, $A$ is the second fundamental form of $M$ and both $H$ and $A$ are evaluated at $(x, f(x))$. 
Then inequality (\ref{eq:SS_Strongstab_witherrors}) (i.e.~inequality \cite[(1.17)]{SS}) holds for all $\varphi\in C^1_c(M)$ with $\frac{c^2}{2}$ in place of $c_5 \mu$.
\end{lem}

\begin{proof}
Denote by $\nu$ the unit normal vector to $M$ such that $\nu \cdot e^{n+1} >0$ and let $H(x,f(x))$ denote the scalar mean curvature of $M$ with respect to $\nu$ at $(x,f(x))$.  Let $N_{t}$ be the graph of $f+t \phi$ for $\phi \in C^1_c(\Omega)$ and $t\in (-\eps, \eps)$. Let $L_{t}$ be the $(n+1)$-dimensional manifold with boundary $M-N_{t}$ defined, for $t > 0$,  by taking 
$\{(x,y):x\in \Om, \phi(x)>0, f(x)<y<f(x)+t\phi(x)\}$ with positive orientation and $\{(x,y):x\in \Om, \phi(x)<0, f(x)+t\phi(x)<y<f(x)\}$ with negative orientation (and similarly for $t<0$). Extend $\nu$ to a vector field on $\Om \times \R$ by setting $\nu(x,y) = \nu(x,f(x))$, for $x\in \Om$ and $y \in \R$ (vertically invariant extension). Define the $n$-form $\om=\iota_{\nu} d\text{vol}^{n+1}$, where $d\text{vol}^{n+1} = dx^1 \wedge \ldots \wedge dx^{n+1}$ is the standard volume form on $\R^{n+1}$ and $\iota$ denotes the inner product (contraction). Then, computing at $(x,y)$ we get $d\om = \text{div} \nu \, d\text{vol}^{n+1}= \widetilde{H}(x,y)\, d\text{vol}^{n+1}$ where $\widetilde{H}(x, y) = H(x, f(x))$ for $(x, y) \in L_{t}$. Note also that $\int_M \om = \mathcal{H}^n(M)$ and that $\om$ has comass $1$. 
Using Stokes' theorem we compute:

$$\mathcal{H}^n(M)=\int_M \om = \int_{N_{t}} \om + \int_{L_{t}} d\om = \int_{N_{t}} \om + \int_{L_{t}} \widetilde{H}(x,y) d\text{vol}^{n+1}. $$ 
Using $\int_{N_{t}} \om \leq \mathcal{H}^n(N_{t})$ we get 
\begin{eqnarray*}
 \label{eq:compareJ_g}
 \mathcal{H}^n(M)\leq \mathcal{H}^n(N_{t})+\int_{L_{t}} \widetilde{H}(x,y) d\text{vol}^{n+1}&=&\mathcal{H}^n(N_{t})+ \int_\Omega \int_{f(x)}^{f(x) + t\phi(x)} \widetilde{H}(x,y)  dy dx\nonumber\\
 &=& \mathcal{H}^n(N_{t}) + t\int_{\Omega} \phi(x)H(x, f(x)) \, dx.
 \end{eqnarray*}
This shows that the right hand side of the above inequality, as a function of $t$ near $0$,  has a minimum at $t=0$, and hence has non-negative second derivative at $t=0$. Therefore we conclude that
$$\left.\frac{d^{2}}{dt^{2}}\right|_{t=0} {\mathcal H}^{n}(N_{t})  \geq 0.$$ 
Since $\left.\frac{d^{2}}{dt^{2}}\right|_{t=0} {\mathcal H}^{n}(N_{t}) = \int_{M} |\nabla^{M} \, \phi|^{2}  - |A|^{2} \phi^{2} + H^{2}\phi^{2} \, d{\mathcal H}^{n}$ (\cite{SimonNotes}),  we conclude, in view of the assumption on $H,$ that  (\ref{eq:SS_Strongstab_witherrors}) holds with $\frac{c^2}{2}$ in place of $c_5 \mu$. 

\end{proof}

To apply Lemma \ref{graph-stability} (with $c=c_3\mu_1$) within the proof of Theorem \ref{thm:SSsheeting_second}, for any graph $M$ appearing in the partial graph decomposition, we choose $\phi=\varphi (1-(\nu \cdot \nu_0)^2)^{1/2}$ in (\ref{eq:SS_Strongstab_witherrors}) and arguing as in the proof of \cite[Lemma 1]{SS}, we then conclude the validity of
\begin{equation}
 \label{eq:Schoenslemmagraph}
\int_M |A|^2 \varphi^2 d{\Hc}^n \leq  c_1 \int_M (1-(\nu \cdot e^{n+1})^2)|\nabla_M \varphi|^2 d{\Hc}^n+ c_1 \mu_1^2 \int_M \varphi^2 d{\Hc}^n
\end{equation}
for any $\varphi \in C^1_c(M)$.

\begin{proof}[Proof of Theorem~\ref{thm:SSsheeting_second}]
The proof now proceeds as in \cite{SS}, incorporating the changes as described in the proof of \cite[Theorem 3.4]{BW} (in particular using inequality (\ref{eq:Schoenslemmagraph}) where 
the proof in \cite{SS} uses \cite[Lemma 1]{SS} on separate ``sheets'' of the partial graph decomposition). Note also that in \cite[Theorem 1]{SS} (see \cite[Remark 2]{SS}), the small excess requirement is that the 
tilt-excess $$\hat{E}_{\rho} \equiv \rho^{-n}\int_{B_{\rho}^{n}(0) \times {\mathbb R}} (1 - (\nu \cdot e^{n+1})^{2}) \, d\|V\|$$ 
is small, but this is implied by our assumption (vii) since it follows from the first variation inequality (\ref{approx-firstvar}) (analogously to \cite[Remark 2]{SS} or \cite[22.2]{SimonNotes}) that 
$\hat{E}_{\rho/2} \leq CE_{\rho}$ for a fixed constant $C = C(n).$
With these changes in place, we proceed as in \cite{SS}
to obtain the analogues of \cite[(4.33)]{SS}
and \cite[(4.36)]{SS};
this implies, by standard arguments, the claimed conclusion of the present theorem. 
\end{proof}

\section{Proof of the regularity theorems}
\label{proof_reg}

We discuss in this section the proofs\footnote{Theorem \ref{thm:mainreg_gpos} itself, i.e.~without assumption (\textbf{m}), will be discussed, in a more general formulation, in Appendix \ref{emptyinterior}.} of Theorem \ref{thm:mainreg_gpos} with (\textbf{m}) (as in Remark \ref{oss:add(m)}), Theorem \ref{thm:mainregstationarityimmersedparts} and Theorem \ref{thm:mainreg_g0_version1}. Conclusion (i) of all three theorems is local in nature; thus it suffices to prove it with $N$ replaced by a small geodesic ball of $N$. Therefore, recalling Remark \ref{oss:assume_strongstab_withambient}, we may and do assume that the Morse index is zero, i.e.\ that strong stability holds. So in view of the discussion in Sections~\ref{exponentialchart} and ~\ref{Schoenslemma}, the proof of Conclusion (i) reduces to Theorem~\ref{thm:regularity_stronglystable} (below) which concerns a class of varifolds ${\mathcal S}_{\G, \mu, \mu_{1}}$ in  a Euclidean ball that satisfies, among other things, stationarity and stability hypotheses with respect to a functional whose integrand belongs to the class of functions ${\mathcal I}(\mu, \mu_{1})$ defined below:

\begin{df}\label{integrands}
For constants $\mu$, $\mu_{1}$, let ${\mathcal I}(\mu, \mu_{1})$ denote the class of functions $\widetilde{F} \, : \, B_{\rho_{0}}^{n+1}(0) \times {\mathbb R}^{n+1} \setminus \{0\} \to {\mathbb R}$ satisfying the following requirements: 
\begin{itemize}
\item $\widetilde{F}( \cdot, p) \in C^{2, \alpha}(B_{2}^{n+1}(0))$,  $\widetilde{F}(x, \cdot) \in C^{3}({\mathbb R}^{n+1} \setminus \{0\})$.
\item for each $y \in B_{1}^{n+1}(0)$, there is a vector $\vec{c}(y) \in {\mathbb R}^{n+1}$ and $C^{2}$ diffeomorphism $\eta_{y} \, : \, B^{n+1}_{2}(0) \to B^{n+1}_{2}(0)$ with $\eta_{y}(0) = y$ and $\eta_{0} = identity$ such that 
 $\eta_{y}^{\#} \widetilde{F}_{\vec{c}(y)}$ satisfies \cite[(1.2)-(1.5)]{SS} (with the fixed $\mu$, $\mu_{1}$) taken with $\eta_{y}^{\#} \widetilde{F}_{\vec{c}(y)}$ in place of $F$ and $\rho_{0} = 2$, where $\widetilde{F}_{\vec{c}}(x, p) = \widetilde{F}(x, p) + \vec{c} \cdot p$.
 \end{itemize}
 \end{df}
\medskip

Definition \ref{integrands} is motivated, as we saw in Section \ref{exponentialchart}, by pulling back the area functional, or more generally the functional $J_g$, to an Euclidean ball by means of an exponential map. Moreover, we saw in Section \ref{exponentialchart} that (a1) of Theorem \ref{thm:mainregstationarityimmersedparts} becomes (a1$^{\prime}$) when we pull back via the exponential map. These facts justify the next definition.

\begin{df}
\label{class_of_varifolds}
Let $p >n$. Let $\Gamma, \mu, \mu_{1}, c  >0$. Denote by $\mathcal{S}_{\Gamma, \mu, \mu_{1}, c}$ the class of integral $n$-varifold $V$ in $B^{n+1}_{2}(0)$ such that:

\begin{enumerate}
  
 \item (a2), (a3) of Theorem \ref{thm:mainregstationarityimmersedparts} (taken with $N = B_{2}^{n+1}(0)$) hold, and  (a1$^{\prime}$) of Section \ref{exponentialchart} holds with $\rho_0=2$ and for a non-negative function ${\hat H}_V \in L^p(\|V\|;\R^{n+1})$ with $\int_{B_{2}^{n+1}(0)}|{\hat H}_V|^p \leq \Gamma$;

 \item For some $\widetilde{F} \in {\mathcal I}(\mu, \mu_{1})$, 
 letting  $\mathcal{F}=\int \widetilde{F}(x,\nu(x)),$ the varifold $V$ satisfies the following:
 
\noindent (I) the stationarity assumptions (b), (b$^T$) of Theorem~\ref{thm:mainregstationarityimmersedparts} (taken with $N = B_{2}^{n+1}(0)$) hold with ${\mathcal F}$ in place of $J_{g}$; 
the stationarity assumption (b*) of Theorem~\ref{thm:mainregstationarityimmersedparts} (taken with $N = B_{2}^{n+1}(0)$) holds  with ${\mathcal F}$ in place of $J_{g}$ and 
whenever $\Oc = \Omega \setminus Z$, with $\Om$ a simply connected open set and $Z$ a closed set with $\text{dim}_{\Hc}(Z)\leq n-7$;

\noindent (II) whenever $\Oc = \Om \setminus Z$ is as in (I), we have that 
\begin{equation*}
\int  |A|^2  \phi^2 d\|V\| \leq \int |\nabla \phi|^2 d\|V\| + 
\end{equation*}
$$+c \mu_1 \int (\mu_1 \phi^2 + \phi |\nabla \phi| + \phi |A| + |x| |\nabla \phi|^2 + |x|\phi^2 |A|^2 + \mu_1 \phi^2 |x|^2 |A|^2) d\|V\|$$ 
 for all $\phi \in C^1_c(\Om \setminus Z),$ 
where $A$ is the second fundamental form of the immersion $\iota_{\Oc}$.
\end{enumerate} 

\begin{oss} In the case of Theorem~\ref{thm:mainregstationarityimmersedparts}, the inequality in (II) with $c = c_{5},$ where $c_{5} = c_{5}(n)$ is as in  \cite[(1.17)]{SS}, follows from hypotheses (b*) and (c) (of Theorem~\ref{thm:mainregstationarityimmersedparts}); 
in the case of Theorem~\ref{thm:mainreg_g0_version1}, it follows (again with $c = c_{5}$) from hypothesis (c) (of Theorem~\ref{thm:mainreg_g0_version1}) and Theorem~\ref{thm:orientabilityGreg}.  
\end{oss}
\end{df}

\begin{thm}[Local regularity for stable $\mathcal{F}$-stationary hypersurfaces]
\label{thm:regularity_stronglystable}
Let $V \in \mathcal{S}_{\Gamma, \mu, \mu_{1}, c}$. Then $\spt{V}$ satisfies conclusion (i) of Theorem \ref{thm:mainregstationarityimmersedparts}. 
\end{thm}

\begin{oss}
The stability assumption (II) can be replaced by the following: inequality (\ref{eqSchoenlemma}) holds, with the \emph{ambient} gradient in place of $\nabla^M$, for all $\phi \in C^1_c(\Oc)$, $\phi \geq 0$, and for any $\Oc$ as in (b*).
\end{oss}

Theorem \ref{thm:regularity_stronglystable} is deduced by proving the following three theorems by simultaneous induction on $q$.

\begin{thm}[Sheeting Theorem]
\label{thm:sheeting}
Let $q$ be a positive integer. There exists $\eps = \eps(n, p, q, \mu, \mu_{1},c) \in (0, 1)$ such that if $V \in \mathcal{S}_{\Gamma, \mu, \mu_{1}, c}$ satisfies, for $\sigma \in (0,1)$,
$$(\omega_{n}2^{n})^{-1}\|\eta_{0,\s\, \#}V\|(B_{2}^{n+1}(0)) < q + 1/2,\,\,\,\, q - 1/2 \leq \omega_{n}^{-1}\|\eta_{0,\s\,\#}V\|\left(B_{1}^{n}(0) \times {\mathbb R} \right) < q + 1/2$$
and the following flatness condition 
$$\int_{B_{1}^{n}(0) \times {\R} } |x^{n+1}|^{2} \, d\|\eta_{0,\s\,\#}V\|(X) + \frac{1}{\eps} \left(\int_{B_{1}^{n}(0) \times {\R} } |H_{\eta_{0,\s\,\#}V}|^{p} \, d\|\eta_{0,\s\,\#}V\|(X) \right)^{\frac{1}{p}} +\sigma< \eps,$$
then $$\eta_{0,\s\,\#}V \res \left(B_{1/2}^{n}(0) \times {\R} \right) = \sum_{j=1}^{q} |{\rm graph} \, u_{j}|,$$
where $u_{j} \in C^{1, \alpha} \, (B_{1/2}^{n}(0); {\mathbb R})$ and $u_{1} \leq u_{2} \leq \ldots \leq u_{q}$, with 
\begin{eqnarray*}
&&\|u_{j}\|^2_{C^{1, \alpha}(B_{1/2}^{n}(0))} \leq C \int_{B_{1}^{n}(0) \times {\mathbb R}} |x^{n+1}|^{2} \, d\|\eta_{0,\s\,\#}V\|(X)\nonumber\\ 
&&\hspace{2in}+ \frac{C}{\eps} \left(\int_{B_{1}^{n}(0) \times {\R} } |H_{\eta_{0,\s\,\#}V}|^{p} \, d\|\eta_{0,\s\,\#}V\|(X) \right)^{\frac{1}{p}}+C\s
\end{eqnarray*}
for some fixed constants $\alpha  = \alpha(n,p, q, \mu, \mu_{1}, c) \in (0, 1/2),$ $C = C(n, p, q, \mu, \mu_{1}, c) \in (0, \infty)$ and each $j=1, 2, \ldots, q$.  
\end{thm}

\begin{thm}[Minimum Distance Theorem]
\label{thm:minimum_dist}
Let $\delta \in (0, 1/2)$ and ${\mathbf C}$ be a stationary integral varifold that is a cone in ${\mathbb R}^{n+1}$  such that ${\rm spt} \, \|{\mathbf C}\|$ consists of three or more $n$-dimensional half-hyperplanes meeting along a common $(n-1)$-dimensional subspace.  
There exists $\eps = \eps({\mathbf C}, \delta, n, p, \Gamma, \mu, \mu_{1}, c) \in (0, 1)$ such that if $V \in \mathcal{S}_{\Gamma, \mu, \mu_{1}, c}$ satisfies, for $\s\in(0,1)$,  
$(\omega_{n}{\s}^{n})^{-1}\|V\|(B_{\s}^{n+1}(0)) \leq \Theta \, (\|{\mathbf C}\|, 0) +\delta$ then 
$$\s+\frac{1}{\s}{\rm dist}_{\mathcal H} \, ({\rm spt} \, \|V\| \cap B_{\s}^{n+1}(0), {\rm spt} \, \|{\mathbf C}\| \cap B_{\s}^{n+1}(0)) > \eps.$$
\end{thm}

\begin{thm}[Higher Regularity Theorem]
\label{thm:higher_regularity}
Let $q$ be a positive integer and let $V \in \mathcal{S}_{\Gamma, \mu, \mu_{1}, c}$ be  such that 
$$V \res \left(B_{1/2}^{n}(0) \times {\R} \right) = \sum_{j=1}^{q} |{\rm graph} \, u_{j}|$$
where $u_{j} \in C^{1, \alpha} \, (B_{1/2}^{n}(0); {\mathbb R})$ for some $\alpha \in (0, 1/2),$ and $u_{1} \leq u_{2} \leq \ldots \leq u_{q}$. Then 
\begin{itemize}
\item[(i)] $u_j \in C^2(B_{1/2}^{n}(0); {\mathbb R})$ and their graphs are stationary with respect to $\mathcal{F}$ (and hence by elliptic regularity $u_{j} \in C^\infty(B_{1/2}^{n}(0); {\mathbb R})$) for each $j$ if $g\in C^\infty$);
\item[(ii)] if $q \geq 2$, the graphs of $u_{j}$ touch at most in pairs, i.e.~if there exist $x  \in B_{1/2}^{n}(0)$ and $i \in \{1, 2, ... q-1\}$ such that $u_i(x) = u_{i+1}(x)$ then $Du_i(x) = Du_{i+1}(x)$ and $u_j (x) \neq u_i(x)$ for all $j \in \{1, 2, ... q\} \setminus \{i, i+1\}$.
\end{itemize}
\end{thm}

The induction scheme for the proofs of the above theorems is as follows. Let $q \geq 2$ be an integer, and assume the following:

\noindent
{\sc induction hypotheses:}\\
\noindent
(H1) Theorem~\ref{thm:sheeting} holds with any $q^{\prime} \in \{1, \ldots, q-1\}$ in place of $q$.\\
\noindent
(H2) Theorem~\ref{thm:minimum_dist} holds whenever $\Theta \, (\|{\mathbf C}\|, 0) \in \{3/2, \ldots, q-1/2, q\}.$  \\
\noindent
(H3) Theorem~\ref{thm:higher_regularity} holds with any $q^{\prime} \in \{1, \ldots, q-1\}$ in place of $q$.\\

Completion of induction is achieved by carrying out, assuming (H1), (H2), (H3), the following four steps in the order they are listed:
\begin{itemize}
\item[(i)] prove Theorem~\ref{thm:sheeting};
\item[(ii)] prove Theorem~\ref{thm:minimum_dist} in case $\Theta \, (\|{\mathbf C}\|, 0) = q+1/2$;
\item[(iii)]  prove Theorem~\ref{thm:minimum_dist} in case $\Theta \, (\|{\mathbf C}\|, 0)  = q+1$;
\item[(iv)]  prove Theorem~\ref{thm:higher_regularity}.
 \end{itemize}

The three theorems above are combined within the induction with the following proposition (used at several places in the induction argument, both in the proof of the sheeting theorem and in the proof of the minimum distance theorem), whose proof relies on a standard tangent cone analysis. 
(see \cite[Proposition 3.1]{BW} for details).

\begin{Prop}
\label{Prop:elementaryconsequence}
Let $V$ be an integral $n$-varifold in $\Omega$, an open subset of $B_{2}^{n+1}(0)$, with $V$ satisfying the assumptions of Theorem \ref{thm:regularity_stronglystable}. For $2\leq q \in \N$ let $S_{q} = \{Z \, : \, \Theta \, (\|V\|, Z) \geq q\}$. Assume that $(H1)$, $(H2)$, $(H3)$ are satisfied and assume further that $S_{q} \cap \Omega = \emptyset$. Then $(\text{sing} V \setminus \Greg{V}) \cap \Omega = \emptyset$ if $n \leq 6,$ $(\text{sing} V \setminus \Greg{V}) \cap \Omega$ is discrete if $n=7$ and $\text{dim}_{\mathcal H} \,\left( \left(\text{sing} V \setminus \Greg{V}\right)  \cap \Omega\right)\leq n-7$ for $n \geq 8$. 

\end{Prop}

The steps (i), (ii), (iii) above follow very closely the arguments developed in \cite[Sections 4, 5, 6]{BW}, to which we refer. In fact, a key advance provided by \cite{BW} lies in the fact that the constant-mean-curvature property is not used directly in the proof of these inductive steps, but it is only used, through the inductive assumption, insofar as it allows to exploit the regularity given by (H1), (H2), (H3). The estimates and the excess decay arguments in \cite[Sections 4, 5]{BW} only require, in terms of direct use of the assumptions, the validity of hypothesis (a1$^{\prime}$): this makes the arguments robust and possibly adaptable to other variational problems or even to non-variational settings  in which we have an integral $n$-varifold with first variation summable in $L^p$ for $p>n$. Theorem~\ref{thm:abstract_sheeting_thm} below makes this observation explicit. Similarly, the arguments in  \cite[Sections 6]{BW} only require (H1), (H2), (H3), Step (i) and the structural assumption (a2) (incidentally, this is the only step in which the no-classical-singularity assumption is used non-inductively). To complete the induction scheme we need step (iv), for which we need the following initial lemma (this is the analogue of \cite[Lemma 7.1]{BW} and it is obtained using the arguments developed in Section \ref{furtherconsequencesoftheassumptions}). 

\begin{lem}
 \label{lem:no_ell-fold}
Under the assumptions of Theorem \ref{thm:higher_regularity} and assuming the validity of (H1), (H2), (H3), let $X=(x, X^{n+1}) \in \spt{V}$ be a point of density $q$ where $\spt{V}$ is not embedded. Then 
$X \in \SingT{V}$ (a ``two-fold'' touching singularity). In other words, $\SingT^\ell(V) =\emptyset$ for $\ell \geq 3$.
\end{lem}

\begin{proof}
The only difference, when we compare with Section \ref{furtherconsequencesoftheassumptions}, lies in the fact that we may have more than two sheets. Denote by $\pi:B_{1/2}^n(0)\times \R \to B_{1/2}^n(0)$ the standard projection, let $C=\pi\{y\in \spt{V}: \Theta(\|V\|,y)=q\}$ (closed set). We argue by contradiction and assume $X\in \SingT^\ell(V)$ for $\ell \geq 3$. We know, by inductive assumption, that away from points of density $q$ the ordered graphs are $C^2$ and individually stationary for $\mathcal{F}$. We can therefore follow \cite[Lemma 7.1]{BW} and select two ordered graphs (when $g$ can take both positive and negative values, this uses (b$^T$)) over a connected component of $B_{1/2}^n(0) \setminus C$ such that they do not coincide identically (here we use the fact that $X$ is not an embedded point) both graphs are $C^2$ and satisfy the PDE (\ref{eq:PDEsfortwographs}) or (\ref{eq:E-L_Riem}) with the same sign on the right-hand-side (this is there we use $q\geq 3$ and the contradiction assumption), which contradicts the one-sided maximum principle (cf.\ Section \ref{furtherconsequencesoftheassumptions}).
\end{proof}

The remainder of the proof of (iv) follows the general scheme of \cite{BW}, with the due modifications of the PDE arguments developed in \cite[Section 7]{BW}.

\medskip

\textit{Conclusion (ii) of Theorems \ref{thm:mainregstationarityimmersedparts} and \ref{thm:mainreg_g0_version1}}. This follows locally in any ball by the arguments in Section \ref{orientability} once the estimate on $\Sigma$ is obtained in conclusion \textit{(i)}. The global conclusion is discussed at the end of Section \ref{orientability}.

\section{Proof of the compactness theorem}
\label{proof_comp}

We describe in this section the proof of Theorem \ref{thm:generalcompactness} (and of its special case Theorem \ref{thm:compactness_gpos}). Allard's theorem \cite{Allard} gives subsequential convergence of $V_j$ to $V_\infty$ in the varifold topology (since mass is bounded locally uniformly and $g_j \to g_\infty$ in $C_{\rm loc}^0$), with the integral $n$-varifold $V_\infty$ satisfying (a1). We will check that all other assumptions of Theorem \ref{thm:mainregstationarityimmersedparts} are preserved on $V_\infty$, from which the regularity of $V_\infty$ follows. 

\begin{oss}
\label{oss:index_and_necks}
In carrying out the proof we also obtain finer information. The (subsequential) convergence is graphical (and $C^2$, possibly with multiplicity) away from a collection $X$ of at most $s$ points (the ``neck regions'' in claim 3 below), where $s$ is the uniform bound on the Morse index: the index bound is preserved in the limit, and actually every point that gives rise to a ``neck region'' reduces the index by $1$, i.e.~the Morse index  of $V_\infty$ is $\leq s-\# X$.
\end{oss}

\medskip

The following claim allows to reduce, upon neglecting a set of isolated points and suitably localizing, to the case in which the Morse index $s$ is $0$; it is a generalization of the argument used in Section \ref{stabilitydiscussion}, however it requires more care since we have a sequence of varifolds, rather than only one. The starting assumption is that $V_j \to V_\infty$ in an open neighbourhood $U$ of the closure of an open ball $B$, with $V_j$ as in Theorem \ref{thm:generalcompactness}, in particular the Morse index  in $U$ is bounded, uniformly in $j$, by $s\in \N$. 

\noindent
\textit{Claim 1}: We can select a collection $X \subset B$ of at most $s$ points (i.e.~$0\leq \# X \leq s$) such that for any $x\in B\setminus X$ there exists a ball $B_r^{n+1}(x)$ with the property that for a subsequence ${j_k}$ the varifolds $V_{j_k} \res B_r^{n+1}(x)$ are stable, i.e.~they satisfy (c) with $s=0$ in $B_r^{n+1}(x)$.  (The subsequence might depend on the point and on the ball.)

\begin{proof}[\textit{proof of claim 1.}]
For every $r>0$ we consider the collection $\{B_x\}_x$, where $B_x:=B_r^{n+1}(x)$ and select the collection $X_r$ of centres $x\in B$ such that there is no subsequence $j_k$ for which $V_{j_k} \res B_x$ is stable (in the sense of (c) --- we know that $\spt{V_j}\setminus \Greg{V_j}$ has codimension $7$ at least, by regularity of $V_j$). If $X_r = \emptyset$ for a certain $r>0$, then we have the result with $X = \emptyset$. Otherwise, we repeat the selection for every $r$ and take $r\to 0$. The set $X_r$, non-empty for every $r>0$, will have a set of limit points $X \subset \overline{B}$, where $x\in X$ if there exists $x_q \to x$ (as $q\to \infty$) with $x_q \in X_{r_q}$ and $r_q \to 0$ as $q\to \infty$. We will next prove that $X$ is made of at most $s$ points. If we can find $s+1$ distinct points $z_1, ..., z_{s+1}$ in $X$, then consider $R>0$ small enough such that $B_R^{n+1}(z_a)$ are pairwise disjoint for all $a\in \{1, ..., s+1\}$. By definition we have, for every $a$, sequences $x^a_\ell \to z_a$ as $\ell \to \infty$, with $x^a_\ell \in X_{r_\ell}$ and $r_\ell \to 0$. In particular, for all $\ell$ large enough and for all $a\in \{1, ..., s+1\}$ we have $B_{r_\ell}^{n+1}(x^a_\ell) \subset B_R^{n+1}(z_a)$. This means that there exists $\ell$ such that the $s+1$ balls (indexed on $a$) $B_{r_\ell}^{n+1}(x^a_\ell)$ are disjoint and $V_j \res B_{r_\ell}^{n+1}(x^a_\ell)$ are unstable for all $j$ large enough (there are finitely many balls so we can take $j$ large enough independently of the ball). This implies that the Morse index  of $V_j$ in $U$ is $\geq s+1$, contradiction. So we have proved that $X$ is made of at most $s$ points. The claim now follows by showing that for every $y \in B\setminus X$ we can find a ball centred at $y$ and a subsequence $V_{j_k}$ such that $V_{j_k}$ is stable in the ball. If this were not the case, then we would have that for every $r>0$ the varifold $V_j \res B_r^{n+1}(y)$ is not stable for all $j$ large enough. This means, by definition, $y \in X_r$ for every $r>0$, so $y\in X$, contradiction.
\end{proof}

The localization provided by claim 1
will allow to check assumptions (a2), (a3), (b) for $V_\infty \res (B\setminus X)$, since these assumptions are of local nature. Before doing that, we point out the following. The existence of a classical singularity implies, by definition, the existence of an $(n-1)$-dimensional set of classical singularities: since $n\geq 2$, ensuring (a2) on $B\setminus X$ actually proves (a2) on $B$. Assumption (a3) involves proving that an $n$-dimensional measure is $0$, so here $n\geq 1$ suffices to conclude that the validity of (a3) in $B\setminus X$ implies the validity of (a3) in $B$. Using again $n\geq 2$ we also conclude that the stationarity condition (b) in $B\setminus X$ implies the validity of (b) in $B$ (point singularities are removable for the stationarity condition if $n\geq 2$).

\medskip

\textit{Claim 2: (a2), (a3), (b) are valid for for $\spt{V_\infty}$.} The proof now follows \cite[Section 8]{BW} closely, to which we refer for details; here we discribe the key ideas. As explained above, it suffices to prove (a2), (a3), (b) in $B_r^{n+1}(y)$ for $y\in B\setminus X$, where $B_r^{n+1}(y)$ is chosen so to ensure the existence of $V_{j_k}$ that are stable in it, $V_{j_k}\to V_\infty$. 

\begin{proof}[\textit{proof of claim 2}.]\footnote{We stress that Theorems \ref{thm:sheeting},\ref{thm:minimum_dist} and \ref{thm:higher_regularity} are results in their own sake, whose validity for every $q$ is guaranteed after completing the induction in Section \ref{proof_reg}. In doing so, the regularity conclusion of the sheeting theorem can be improved to $C^2$, with elliptic estimates, upon combining it with Theorem \ref{thm:higher_regularity}.}
The validity of (a2) on $V_\infty$ is proved by contradiction and is a consequence of the minimum distance theorem (Theorem \ref{thm:minimum_dist}) combined with a standard rescaling argument. 

Conditions (a3) and (b) must be checked locally around points at which $V_\infty$ has a planar tangent (possibly with multiplicity)\footnote{Since uniqueness of tangents in not known, by this we mean that there is at least a planar tangent at the chosen point.}: in this case, on which we focus next, the sheeting and higher regularity theorems (Theorems \ref{thm:sheeting} and \ref{thm:higher_regularity}) play a decisive role.

Consider any $x\in \spt{V_\infty}$ at which the tangent to $V_\infty$ is a plane with multiplicity $q \in \N$. We claim that, in a (possibly suitably dilated) ball centred at $x$ every $V_j$ in the subsequence extracted in the localization step satisfies, for all $j$ large enough, the conditions of the Sheeting Theorem. In view of the localization step, the only ones to check are the mass and flatness conditions of Theorem \ref{thm:sheeting}. Varifold convergence implies convergence of the masses; moreover convergence of $\spt{V_j}$ to $\spt{V_\infty}$ in Hausdorff distance is implied by the monotonicity formula applied to each $V_j$ (using the uniform smallness of the mean curvatures $g_j$). Analogous observations hold for a sequence of dilations of $V_\infty$ that converge to the tangent plane (counted $q$ times). We can ensure, upon suitably dilating a ball around $x$, that $V_\infty$ is as flat as we wish and that its mass is as close as we wish to $q$ in ball of radius $2$. These conditions imply the validity of the mass and flatness conditions of Theorem \ref{thm:sheeting} for $V_j$ for $j$ large enough. In particular, using the combined power of Theorems \ref{thm:sheeting} and \ref{thm:higher_regularity}, we obtain graphical convergence $V_j\to V_\infty$ with $C^2$-estimates around any such $x$ (for the subsequence $V_j$): this allows to check the validity of (b) and (a3) for $V_\infty$.
\end{proof}

The local $C^2$ convergence obtained also allows to check the validity of the remaining assumptions (b*), (b$^T$) and (c) in $B_r^{n+1}(y)$, since these conditions only concern regular parts of $V$, where the tangents are planar. In particular the local regularity of $\spt{V_\infty}$ can be established already at this stage. We skip this step, since in the next we prove directly the validity of (b*), (b$^T$) and (c) for $V_\infty$ in $B$.

\medskip

Note that, in verifying (b*) and (c) for $V_\infty$, we only need to consider open sets that are cylinders $C$ and in which the ``pure'' singular set is assumed to be of codimension $\geq 7$. Fix any such $C$. 

\textit{Claim 3: identification of at most $s$ ``neck regions''}. We can find a set $X=\{x_1, ..., x_m\}\subset \Greg{V}$ with $0\leq m\leq s$, such that for every $R>0$ there exist a subsequence $V_{j_k}$ such that $V_{j_k} \res \left(C\setminus \left(\cup_{\ell=1}^m \overline{B}_R^{n+1}(x_\ell)\right)\right)$ has Morse index  $\leq s-m$ for every $j_k$. (The subsequence may depend on $R$ but this will not affect the conclusion.) Moreover, for each $R$ (in the following we drop the dependence on $R$ for notational convenience) we can produce an $n$-manifold $S_\infty$ and an immersion $\iota_\infty:S_\infty \to C$ such that $(\iota_{\infty})_\sharp (S_{\infty}) = V_{\infty} \res \left(C \setminus \left(\cup_{\ell=1}^m \overline{B}_R^{n+1}(x_\ell)\right)\right)$ and a sequence of immersions $\iota_{j_k}:S_\infty \to C$ such that $(\iota_{j_k})_\sharp (S_{\infty}) = V_{j_k} \res \left(C\setminus \left(\cup_{\ell=1}^m \overline{B}_R^{n+1}(x_\ell)\right)\right)$ and such that $\iota_{j_k} \to \iota$ in $C^2$ and with locally graphical convergence in the image.

The validity of (b*) follows from the claim by sending $R\to 0$. For (c), it suffices to test on ambient test functions compactly supported away from $X$ and $C^2$ convergence gives that the index passes to the limit, so the index of $V_\infty$ in $C\setminus X$ is $\leq s-m$. A standard capacity argument gives that the index is the same in $C$.

\begin{proof}[\textit{proof of claim 3}.]
\textit{step 1}. This argument is similar to the one used in a previous claim but this time we are aiming for a different conclusion. Whenever $V_j$ in $U$ have Morse index  $\leq s \in \N$, we can prove that either (a) there exists $r>0$ and a cover of $\Greg{V_\infty}$ with balls $B_r(x)$ such that a subsequence $V_{j_k}$ is stable in each ball or (b) for every $r>0$ there exists a ball $B_r(x)$ in which stability fails for all $V_j$ with $j$ large enough. When we are in alternative (b), sending $r\to 0$ and taking a limit point of the associated points $x$ we obtain $y$ and $B_R(y)$ (with $R$ as small as we wish) such that stability fails in $B_R(y)$ on a subsequence $V_{j_k}$. For this $V_{j_k}$ we can therefore conclude that the Morse index  in $U\setminus \overline{B}_R(y)$ is at most $s-1$, using the argument in Section \ref{stabilitydiscussion}. Iterating (and each time extracting from the subsequence identified at the previous step) we select a collection $X$ as in the first part of the claim and such that there exists $r>0$ and a cover of $\Greg{V_\infty} \setminus \left(\cup_{\ell=1}^m \overline{B}_R^{n+1}(x_\ell)\right)$ with balls $B_r(x)$ such that a subsequence $V_{j_k}$ is stable in each ball (however the Morse index  of $V_{j_k}$ in $U \setminus \left(\cup_{\ell=1}^m \overline{B}_R^{n+1}(x_\ell)\right)$ may be $s-m>0$).

\textit{step 2}. We work in a slighly smaller cylinder whose compact closure is contained in the original $C$. By abuse of notation we still denote by $C$ the new cylinder and by $V_j$ the subsequence. By step 1 we can cover $\Greg{V_\infty} \cap \left( C \setminus  \left(\cup_{\ell=1}^m  B_R^{n+1}(x_\ell)\right)\right)$ with open cylinders of the type $B^n_r(x) \times (-\sigma, \sigma)$ with the first factor in the tangent $T_x V_\infty$ for $x\in \Greg{V_\infty}$ and ensuring the validity of the sheeting theorem in each cylinder. The cover can be made finite by compactness so that the sheeting result holds simultaneously in all cylinders from all large enough $j$. This allows to indentify, as $C^2$ manifolds, $S_j$ (the abstract manifold whose immersion gives $\Greg{V_j}$--- this exists by assumption and by the regularity of $V_j$) with $S_{j+\ell}$ for all $\ell>0$ for $j$ large enough. Once the identification is made, denoting by $S_\infty$ the manifold, we can reparametrize the immersions $\iota_j:S_j \to N$ as immersions $\iota_j^\prime:S_\infty \to N$ and show, thanks to the graphical convergence, that these converge in $C^2$ to a $C^2$-immersion $\iota_\infty: S_\infty \to N$ whose push-forward is $V_\infty$ (again, this follows from the graphical convergence as it is a local property).

Having done the previous step on a smaller cylinder, the conclusion for the original $C$ follows by exhausing the original $C$. 
\end{proof}
 
\textit{Verification of  (b$^T$)}. The verification of (b$^T$) requires to consider a cylinder $C=B\times (-\sigma, \sigma)$ in which $\spt{V_\infty}$ is the union of two (ordered) $C^{1,\alpha}$ graphs, $\text{graph}(u_1)$ and $\text{graph}(u_2)$, with gradients bounded e.g. by $1/2$. By removing a set $X$ of at most $s$ points, as done above, we may assume that the conditions that ensure the validity of the sheeting theorem are valid locally around each point $y\in \spt{V_\infty} \res (C\setminus X)$ and the graphs obtained in each small cylinder $C_y$ describe $V_j \res C_y$ and have small gradients (e.g.~bounded by $1/4$). Pasting the graphs, these bounds imply that $V_j  \res (C\setminus X)$ can be exressed as a union of graphs on $B$: the ordered graphs satisfy separately the stationarity condition (by the regularity of $V_j$). Passing to the limit the PDE we get the validity of (b$^T$) for each of the two graphs $\text{graph}(u_j) \setminus X$, and this immediately extends across $X$ since $n\geq 2$ (by improving the convergence to $C^2$ via the higher regularity theorem for $V_j$, we may conclude that $u_j$ are $C^2$ for $j=1,2$, but this is not needed).

\section{Proofs of the corollaries for Caccioppoli sets}
\label{proof_Caccioppoli}

\begin{proof}[proof of Corollary \ref{cor:Caccioppolicompactness}]
Each $E_j$ satisfies the local regularity conclusions of Theorem \ref{thm:mainreg_g0_version1} and therefore, by the results of Section \ref{orientability}, $E_j$ satifies (b*) of Theorem \ref{thm:generalcompactness} whenever $\Oc=A\setminus Z$ where $A$ is simply connected (e.g.~a fixed ambient geodesic ball) and $Z=\spt{D\chi_{E_j}} \setminus \Greg{|\p^*E_j|}$. By Theorem \ref{thm:mainregstationarityimmersedparts}, any limit $V$ of $|E_j|$ satisfies the assumptions of Theorem \ref{thm:mainregstationarityimmersedparts}. Moreover, by the proof of Theorem \ref{thm:generalcompactness}, the convergence of $|E_j|$ to $V$ is, away from a locally finite collection of points, locally $C^2$ graphical. 

All that is left to check to ensure that $V$ satisfies the assumptions of Theorem \ref{thm:non-orientable} (and therefore conclusion (i) of Theorem \ref{thm:mainreg_g0_version1}) is that (a4) of Theorem \ref{thm:mainreg_g0_version1} is valid for the limit $V$. Let $p\in \Reg{V} \cap \{g\neq 0\}$ such that in a geodesic ball around $p$, in which $g\neq 0$, the convergence of $|E_j|$ to $V$ is $C^2$ and graphical. Then by the one-sided maximum principle (Section \ref{furtherconsequencesoftheassumptions}), by the topological constraint that the graphs bound $E_j$ and by the condition that the mean curvatures on each graph do not vanish and have to point in opposite directions for any two adjacent graphs (ordering by height), we can conclude that there could only be one graph and therefore the density at $p$ is $1$. 

\medskip

The global two-sidedness conclusion will follow from the fact that any connected component of $\stackrel{\circ}{\Reg{V} \cap \{g=0\}}$ with odd multiplicity is two-sided (as in Theorem \ref{thm:non-orientable}). This is again a consequence of the sheeting result proved within the compactness theorem in Section \ref{proof_comp}. Let $M_m$ be any such connected component, with $m\in \N$ odd that denotes its multiplicity. Working away from a locally finite collection of points in $M_m$ (this collection does not disconnect $M_m$, since $n\geq 2$) we obtain locally graphical convergence of $\p E_j$ to $M_m$, with an odd number of graphs collapsing to the common limit. Since, for each $j$, the $m$ graphs bound $E_j$, we have well-defined outer normals to $E_j$ that give well-defined normals on the graphs. For topological reasons, adjacent graphs (ordered by height) must have opposite pointing normals. Therefore there exists a preferred normal on $M_m$, i.e.~the one pointing in the direction that agrees with the normal on $(m+1)/2$ of the graphs (while the remaining $(m-1)/2$ graphs have a normal pointing in the opposite direction). Since the outer normals to $E_j$ are well-defined globally in $N$, we get a well-defined normal on $M_m$ in $N$.
\end{proof}

We prove now that with analytic data $N$ and $g$ the weaker assumption (c$^{\prime}$) in Remark \ref{oss:analytic} actually implies the validity of the stability condition of Corollary \ref{cor:Caccioppoli2}.

\begin{proof}[proof of the statement in Remark \ref{oss:analytic}]
Let $D_1, D_2 \subset \Greg{|\p^*E|}$ be the two distinct embedded $C^2$ disks whose union is $\spt{D\chi_E}$ in a neighbourhood of a point in $\Greg{|\p^*E|}\setminus \Reg{|\p^*E|}$; in particular, $D_1$ and $D_2$ intersect only tangentially and their mean curvatures are given by $g \nu_j$, where the unit normal $\nu_j$ on $D_j$ is chosen, for $j\in \{1,2\}$, so that $\nu_1=-\nu_2$ at points in $D_1\cap D_2$. Then the analyticity of $g$ and standard PDE theory imply that $D_1$ and $D_2$ are analytic, therefore $D_1 \cap D_2$ is the finite stratified union of analytic submanifold of dimensions $\in\{0, 1, \dots, n-1\}$. The stationarity assumption of Corollary~\ref{cor:Caccioppoli2} rules out the possibility that $D_1\cap D_2$ is of dimension $n-1$: this can be seen by arguing as in \cite[Section 9]{BW}. Then $D_1\cap D_2$ has locally finite $n-2$-dimensional Hausdorff measure. With this size estimate, the stability inequality on $\Reg_1\,|\p^*E|$ that follows from (c$^{\prime}$) can be extended to $\Greg{|\p^*E|}$ by a standard capacity argument.
\end{proof}

\section{A byproduct: an abstract regularity theorem}
The argument used in the inductive step of the proof of Theorem~\ref{thm:sheeting} (under the induction hypotheses (H1), (H2), (H3), see Section \ref{proof_reg}) gives an abstract regularity result for codimension 1 integral varifolds that we make explicit in Theorem~\ref{thm:abstract_sheeting_thm} below. The hypotheses of this theorem are ``non-variational'' in the sense that they do not require the varifold to be stationary point of a functional. 
For a given integral varifold $V$, Theorem \ref{thm:abstract_sheeting_thm} is applicable at a point $y$ where one tangent cone is a multiplicity $q$ plane, provided there are no classical singularities nearby and a suitable ``good behaviour in $C^2$ sense'' holds in the region where density $\leq q-1$ (see hypothesis (d)); the conclusion is that $C^{1,\alpha}$ regularity (with sheeting) holds in a neighbourhood of that point.

Theorem \ref{thm:abstract_sheeting_thm} is not required for the present work. However, it plays a fundamental role in \cite{BW2}, where it is applied to varifolds obtained by taking limits of (the varifolds associated to) energy-bounded, index-bounded solutions to the inhomogeneous Allen--Cahn equation, as the Allen--Cahn parameter tends to $0$. In that setting, the varifold obtained in the limit is not necessarily a critical point of a geometric functional.

\begin{thm}
\label{thm:abstract_sheeting_thm}
Let $\mu, \mu_{1} \in (0, \infty)$. Let $q$ be a positive integer, $\beta \in (0, 1)$ and let $p >n$. Let ${\mathcal V}$ be a class of of integral $n$-varifolds $V$ on $B_{1}^{n+1}(0)$ 
with $0 \in {\rm spt} \, \|V\|$ and satisfying the following properties (a)-(c): 
\begin{itemize}
\item[(a)] for each $V \in {\mathcal V}$ there is a non-negative function ${\hat H}_{V} \in L^{p}_{\rm loc}(\|V\|)$ such that hypothesis (a1$^{\prime}$) of Section~\ref{exponentialchart}  holds with $\rho_{0} = 1.$ 
\item[(b)] if $V \in {\mathcal V}$ then $V$ has no classical singularities $Y$ with $\Theta \, (\|V\|, Y) = q$;
\item[(c)]  if $V \in {\mathcal V}$, $\sigma \in (0, 1)$, $(\omega_{n}2^{n})^{-1}\|\eta_{0, \sigma \, \#} \, V\|(B_{1}^{n+1}(0)) \leq q+ 1/2$, $q-1/2 \leq \omega_{n}^{-1}\|\eta_{0, \sigma \, \#} V\|((B_{1/2}^{n}(0) \times {\mathbb R}) \cap 
B_{1}^{n+1}(0)) \leq q + 1/2$, $\Theta \, (\|\eta_{0, \sigma \, \#} V\|, X) < q$ for each $X \in B_{1}^{n+1}(0)$ and 
\begin{eqnarray*}
&&E \equiv \int_{(B_{1/2}^{n}(0) \times {\mathbb R}) \cap B_{1}^{n+1}(0)} |x^{n+1}|^{2} \, d\|\eta_{0, \sigma \#} V\|\nonumber\\ 
&&\hspace{1in}+ \left(\int_{(B_{1/2}^{n}(0) \times {\mathbb R}) \cap B_{1}^{n+1}(0)} |\sigma {\hat H}_{V}|^{p} \, d\|\eta_{0, \sigma \, \#} V\|\right)^{1/p} + \mu_{1}\sigma < \beta,
\end{eqnarray*}
 then 
$\eta_{0, \sigma \, \#} V \res ((B_{1/2}^{n}(0) \times {\mathbb R}) \cap B_{1}^{n+1}(0))  = \sum_{j=1}^{q} |{\rm graph} \, u_{j}|$ for some functions 
$u_{j} \in C^{2}(B_{1/2}^{n}(0))$ satisfying $u_{1} \leq u_{2} \leq \ldots \leq u_{q}$ and $\|u_{j}\|_{C^{1, \alpha}(B_{1/2}(0))} \leq C \sqrt{E}$ for any $\alpha \in (0, 1)$ and some constant 
$C = C(n, p, q, \mu, \mu_{1}, \alpha).$ 
\end{itemize}
Conclusion : there exists $\epsilon = \epsilon (n, p, q, \mu, \mu_{1}, \beta, {\mathcal V}) \in (0, 1)$ such that if $V \in {\mathcal V}$, $\sigma \in (0, 1)$, 
 $(\omega_{n}2^{n})^{-1}\|\eta_{0, \sigma \, \#} \, V\|(B_{1}^{n+1}(0)) \leq q+ 1/2$, $q-1/2 \leq \omega_{n}^{-1}\|\eta_{0, \sigma \, \#} V\|((B_{1/2}^{n}(0) \times {\mathbb R}) \cap 
B_{1}^{n+1}(0)) \leq q + 1/2$ and 
\begin{eqnarray*}
&&E \equiv \int_{(B_{1/2}^{n}(0) \times {\mathbb R}) \cap B_{1}^{n+1}(0)} |x^{n+1}|^{2} \, d\|\eta_{0, \sigma \#} V\|\nonumber\\ 
&&\hspace{1.5in}+ \left(\int_{(B_{1/2}^{n}(0) \times {\mathbb R}) \cap B_{1}^{n+1}(0)} |\sigma {\hat H}_{V}|^{p} \, d\|\eta_{0, \sigma \, \#} V\|\right)^{1/p} + \mu_{1}\sigma < \epsilon,
\end{eqnarray*}
 then $\eta_{0, \sigma \, \#} V \res ((B_{1/2}^{n}(0) \times {\mathbb R}) \cap B_{1}^{n+1}(0))  = \sum_{j=1}^{q} |{\rm graph} \, u_{j}|$ for some $u_{j} \in C^{1, \alpha}(B_{1/2}^{n}(0))$, $j=1, 2, \ldots, q,$ with $u_{1} \leq u_{2} \leq \ldots \leq u_{q},$ 
where $\alpha = \frac{1}{2} \left(1 - \frac{n}{p}\right).$ Furthermore, we have that $$\|u_{j}\|_{C^{1, \alpha}(B_{1/2}(0))} \leq C \sqrt{E}$$ 
for each $j \in \{1, 2, \ldots, q\},$ where $C = C(n, p, q, \mu, \mu_{1}).$ 
\end{thm} 

\begin{oss}
 It is important to recall that assumption (a) holds true when $V$ is obtained by pulling back via the exponential map an integral $n$-varifold on a Riemannian manifold whose first variation (with respect to the Riemannian area functional) is in $L^p$ for $p>n$, as described in Section \ref{exponentialchart}.
\end{oss}

\begin{proof} The argument is precisely as in the inductive step of Theorem~\ref{thm:sheeting}.
\end{proof}

\appendix

\section{Other instances of the regularity theorems}
\label{emptyinterior}
 
Consider the case in which $g \in C^{1, \alpha}(N)$ and $\{g=0\} \cap \Reg_1\,V$ has empty interior in $\Reg_1\,V$. In this situation, which holds for example when ${\Hc}^n\left(\{g=0\}\right)=0$ (or if $g$ is analytic and generic in the sense of Theorem \ref{thm:analytic} below), hypothesis (a4) in Theorem \ref{thm:mainreg_g0_version1} can be removed and we obtain a theorem that is the immediate generalization of Theorem \ref{thm:mainreg_gpos}. The fact that $\Greg{V}$ is the image of an an orientable immersion follows in this case directly from Lemma~\ref{lem:immersion1}, Lemma~\ref{lem:orientability} and Step 1 of Proposition~\ref{Prop:orientabilityimmersion} because 
both $M_{1}$ and $M_{2}$ (notation as in Section~\ref{orientability}) are empty sets. (Thus, under this extra condition on the nodal set,  the issue that appeared in Figure \ref{fig:immersion_high_mult} is removed.) $\Greg{V}$ is stationary with respect to $J_g$ by hypothesis (b) (neither (a4) of Theorem \ref{thm:mainreg_g0_version1} nor (b*) of Theorem \ref{thm:mainregstationarityimmersedparts} is needed). Hence it follows from Theorem~\ref{thm:mainregstationarityimmersedparts} that the same statement as Theorem~\ref{thm:mainreg_gpos} holds true with the additional assumption $(b^T)$ unless $g\geq 0$. So we have:

\begin{thm}[Regularity theorem when $\{g=0\} \cap \Reg_1\,V$ has empty interior]
\label{thm:mainreg_emptyinterior}
For $n \geq 2$ let $N$ be a Riemannian manifold of dimension $n+1$ and $g:N \to \R$ be a $C^{1,\alpha}$ function. Let $V$ be an integral $n$-varifold in $N$ such that $\{g=0\} \cap \Reg_1\,V$ has empty interior in $\Reg_1\,V$. Assume that:

\begin{enumerate}
\item[(a1)] the first variation of $V$ in $L^{p}_{\rm loc} \, (\|V\|)$ for some $p >n$;
\item[(a2)] no point of ${\rm spt} \, \|V\|$ is a  classical singularity of ${\rm spt} \, \|V\|$; 
\item[(a3)] $V$ satisfies (\textbf{T}).
\end{enumerate}
Suppose moreover that: 
\begin{enumerate}
\item[(b)]  the (embedded) $C^1$ hypersurface $S={\rm reg}_{1} \, V$ is stationary with respect to $J_{g}$ in the sense of Definition~\ref{stationary-def} 
taken with $\iota \, : \, S \to N$ equal to the inclusion map;

\item[(b$^T$)] \emph{(redundant if $g\geq 0$)} each touching singularity $p \in {\rm sing}_{T} \, V$ has a neighborhood $\Oc$ such that writing 
$\exp_{p}^{-1} (\spt{V}\cap \Oc) = {\rm graph} \, u_{1} \cup {\rm graph} \, u_{2}$ for $C^{1,\alpha}$ functions $u_1 \leq u_2,$\footnote{Such $u_{1}, u_{2}$ always exist by the definition of touching singularity; see remark~\ref{oss:touchingsinggraphs}.} we have that for $j = 1,2$, the stationarity of 
$S_{j} = \Reg_1 \,V \cap \exp_{p} \, \text{graph} (u_j)$ (which follows from (b)) holds for the orientation that agrees with one of the two possible orientations of $\exp_{p} \, \text{graph} (u_j)$;
\item[(b*)] \emph{(redundant if $g> 0$)} for each orientable open set ${\Oc} \subset N \setminus ({\rm spt} \, \|V\| \setminus {\Greg} \, V)$ there exist an oriented $n$-manifold $S_{\Oc}$ and a proper $C^2$ immersion $\iota_{\Oc}:S_{\Oc} \to \Oc$ with $V\res \Oc = (\iota_{\Oc})_\#(|S_{\Oc}|)$ such that $\iota_{\Oc}$ is stationary with respect to $J_{g}$; 

\item[(c)] for each orientable open set ${\Oc} \subset N \setminus ({\rm spt} \, \|V\| \setminus {\Greg} \, V)$ such that $\Greg{V} \cap \Oc$ is the image of a proper $C^2$ orientable immersion $\iota_{\Oc} \, : \, S_{\Oc} \to \Oc$ that is stationary with respect to $J_{g},$ 
$\iota_{\Oc}$ has finite Morse index relative to ambient functions (Definition~\ref{df:index}). 
\end{enumerate}
Then there is a closed set 
 $\Sigma \subset {\rm spt} \, \|V\|$ with $\Sigma = \emptyset$ if $n \leq 6$, $\Sigma$ discrete if $n=7$ and ${\rm dim}_{\mathcal H} \, (\Sigma) \leq n-7$ if $n \geq 8$ such that: 
 \begin{enumerate}
\item [(i)] locally near each point  $p \in {\rm spt} \, \|V\| \setminus \Sigma$, either ${\rm spt} \, \|V\|$ is a single $C^2$ embedded disk or ${\rm spt} \, \|V\|$ is precisely two $C^2$ embedded disks with only tangential intersection; if we are in the second alternative and if $g(p)\neq 0$, then locally around $p$ the intersection of the two disks is contained in an $(n-1)$-dimensional submanifold;
\item[(ii)] ${\rm spt} \, \|V\| \setminus \Sigma$ is the image of a proper $C^2$ oriented immersion $\iota:S_V \to N$ and there is a continuous choice of unit normal $\nu$ on $S_V$ so that the mean curvature $H_{V}(x)$ of $S_V$ at any $x \in S_{V}$ is given by $H_{V}(x) =  g(\iota(x))  \nu(x)$;
\item[(iii)] if additionally hypothesis ({\bf m}) (see Remark~\ref{oss:add(m)}) holds, then $S_{V}$ can be chosen such that  $V = \iota_\#(|S_V|).$ 
 \end{enumerate}
\end{thm}
 
Moreover, the following compactness theorem is now a direct consequence of Theorem~\ref{thm:generalcompactness}. 

\begin{thm}[\textbf{compactness}]
\label{thm:compactness-smallnodalset}
Let $n \geq 2$ and $N$ be a Riemannian manifold of dimension $n+1$, let $g_j:N\to \R$ and $g_\infty:N\to \R$ be $C^{1,\alpha}$ functions such that ${\mathcal H}^{n} \, (\{g_{j} = 0\}) = 0$ for each $j$ and $g_j \to g_\infty$ in $C_{\rm loc}^{1,\alpha}$. Let $V_j$ be a sequence of integral $n$-varifolds such that hypotheses $(a1), (a2), (a3), (b), (b^{T})$ and $(c)$ of Theorem~\ref{thm:mainreg_emptyinterior}  hold with $V_{j}$, $g_{j}$ in place of $V$, $g.$ Suppose further that $\limsup_{j \to \infty} \, \|V_{j}\| (N \cap K) < \infty$ for each compact $K \subset N$, and that $\limsup_{j \to \infty} \, \text{Morse Index}\, (\iota_{\Oc_{j}}) < \infty$ for each $\widetilde{\Oc} \subset\subset N$,  where $\Oc_{j} = \widetilde{\Oc} \setminus ({\rm spt} \, \|V_{j}\| \setminus \Greg{V_{j}})$ and $\iota_{\Oc_{j}}$ is as in hypothesis (c) of Theorem~\ref{thm:mainreg_emptyinterior}. Then there is a subsequence of $(V_{j})$ that converges in the varifold topology to a varifold $V$ that satisfies the conclusions (i) and (ii) of Theorem \ref{thm:mainreg_emptyinterior} with $g_\infty$ in place of $g$. Furthermore, if $V_{j}$ satisfies hypothesis $({\bf m})$ (see Remark~\ref{oss:add(m)}) for each $j$, then so does $V$, and $V$ satisfies conclusion (iii) of Theorem~\ref{thm:mainreg_emptyinterior}.
\end{thm}

\medskip
Here of course it need not be the case that $\{g = 0\} \cap \Reg_{1} \, V$ has empty interior. 
\medskip

Let us consider the case in which we have analytic data $N$, $g:N\to \R$, $g$ not identically $0$. Then the nodal set $\{g=0\}$ is the stratified union of analytic submanifolds of dimension $\in \{0, 1, \ldots,n\}$. Recall that if $\Reg_1\,V$ has an open set in common with an $n$-dimensional submanifold $\mathcal{N}_n$ in the top stratum of $\{g=0\}$, then $V$ is minimal in that open set, therefore $\mathcal{N}_n$ contains a minimal portion with non-empty interior. However, $\mathcal{N}_n$ itself is analytic and its mean curvature is an analytic function on $\mathcal{N}_n$, so its vanishing in an open set forces its vanishing everywhere on $\mathcal{N}_n$. We therefore conclude that

\begin{thm}[Regularity theorem for analytic functions]
\label{thm:analytic}
Let $N$ and $g:N\to \R$ be analytic, with $g$ not identically $0$, and assume that the no irreducible component of the top stratum of the nodal set $\{g=0\}$ is a minimal hypersurface in $N$. Let $V$ be an integral $n$-varifold in $N$ such that hypotheses (a1), (a2), (a3), (b), (b$^T$), (c) of Theorem \ref{thm:mainreg_emptyinterior} hold (if $g\geq 0$ then (b$^T$) is not needed). Then conclusions (i), (ii), (iii) of Theorem \ref{thm:mainreg_emptyinterior} hold.
\end{thm}

\begin{oss}
For an arbitrary analytic function $g$ on $N$, if the hypotheses of Theorem \ref{thm:mainreg_g0_version1} or of Theorem \ref{thm:mainregstationarityimmersedparts} are satisfied, then the only way in which $\{g=0\} \cap \Reg_1\,V$ may have non-empty interior in $\Reg_1\,V$ is when a whole $n$-manifold of the top stratum of the nodal set is contained in $\spt{V}$. Indeed, the regularity results Theorem \ref{thm:mainreg_g0_version1} or of Theorem \ref{thm:mainregstationarityimmersedparts} guarantee that $V$ is, away from a codimension-$7$ ``pure'' singular set, a smooth immersed oriented $g$-hypersurface, and therefore analytic by elliptic PDE theory. Then the fact that the immersion has an open set on which it coincides with the analytic $n$-manifold $\mathcal{N}_n$ forces a whole connected component of the immersion to agree with $\mathcal{N}_n$.
\end{oss}

\section{On the hypothesis (b$^T$)}
\label{gchangessign} 

Hypothesis  (b$^T$) in Theorem \ref{thm:mainregstationarityimmersedparts} is fulfilled in two instances where the theory can be applied: to characterze the class of integral varifolds obtained by taking limits of embedded stable $J_{g}$-stationary hypersurfaces, as discussed in Remark \ref{oss:compactnessforembedded}, and in the Allen--Cahn approximation schemes as in \cite{BW2}, \cite{RogTon}, as discussed in Section \ref{generalizationII}. By enforcing  (b$^T$) we rule out a varifold as the one depicted in Figure \ref{fig:g_changing_sign}, which satisfies the remaining assumptions of Theorem \ref{thm:mainregstationarityimmersedparts}: although it is a $C^2$-immersion, stationary and stable with respect to $J_g$, this varifold is not a limit of embedded 
$J_{g}$-stationary hypersurfaces.

\begin{figure}[h]
\centering
 \includegraphics[width=5.5cm]{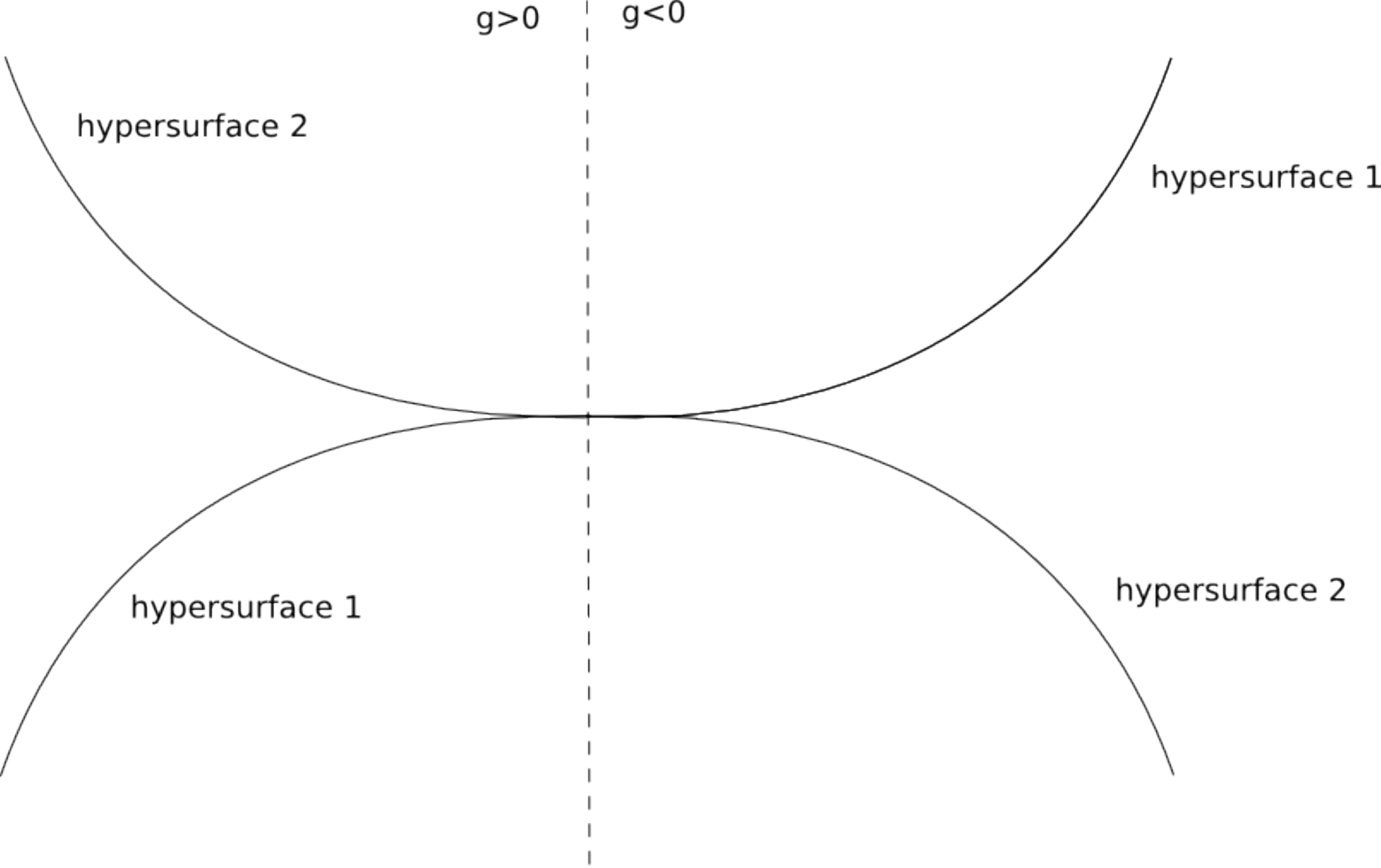} 
\caption{\small Hypersurface $1$ is oriented with normal pointing downwards and hypersurface $2$ is oriented with normal pointing upwards. The two hypersurfaces are separately stationary and stable with respect to $J_g$ (being graphs), therefore so is the immersion.}
 \label{fig:g_changing_sign}
\end{figure}

One could ask whether Theorem \ref{thm:mainregstationarityimmersedparts} holds without hypothesis (b$^T$). We wish to point out that in the absence of (b$^{T}$), the natural geometric class that includes examples as the one in Figure \ref{fig:g_changing_sign} is
a class of more general immersions than those allowed in the conclusions of Theorem~\ref{thm:mainregstationarityimmersedparts}, which are in fact ``quasi-embedded,'' i.e.\ locally near each non-embedded point, $\Greg{V}$ is the union of two \emph{separately $J_{g}$-stationary} ordered $C^{2}$ graphs. The immersions for which assumption  (b$^T)$ fails but all other hypotheses hold (such as that depicted in the example above) are not of this type, nor are they weak limits of embedded $J_g$-stationary hypersurfaces.
If we give up hypothesis (b$^{T}$), then it seems more natural to allow general smooth immersions of $J_{g}$-stationary hypersurfaces: i.e.\ (i) to also allow classical singularities but only those where the support of the varifold is immersed, and (ii) define $\Greg{V}$ to be the $C^{2}$ immersed part. In that generality however, the non-immersed points of the varifold may contain branch point singularities (even in the case $g=0$), and their analysis is more subtle. We do not pursue this direction in the present work.


\begin{thebibliography}{99}

\bibitem[All72]{Allard} W. K. Allard {\it On the first variation of a varifold},  Ann. of Math. (2) 95 (1972), 417-491
\bibitem[BarDoC84]{BarbDoCarmo} J. L. Barbosa, M. do Carmo {\it Stability of Hypersurfaces of Constant Mean Curvature} Math. Zeit. 185 (1984) 3 339-353. 
\bibitem[BDE88]{BarbDoCEsch} J. L. Barbosa, M. do Carmo, J. Eschenburg {\it Stability of Hypersurfaces of Constant Mean Curvature
 in Riemannian Manifolds} Math. Z. 197, 123-138 (1988).
\bibitem[BelWic-1]{BW} C. Bellettini, N. Wickramasekera {\it Stable CMC integral varifolds of codimension $1$: regularity and compactness}, arXiv 2018.
\bibitem[BelWic-2]{BW2} C. Bellettini, N. Wickramasekera {\it The inhomogeneous Allen--Cahn equation and the existence of prescribed-mean-curvature hypersurfaces}, arXiv 2020.
\bibitem[BCW19]{BCW} C. Bellettini, O. Chodosh, N. Wickramasekera {\it Curvature estimates and sheeting theorems for weakly stable CMC hypersurfaces}, Adv. Math. 352 (2019) 133-157.
\bibitem[DeG61]{DG2} E.~De~Giorgi {\it Frontiere orientate di misura minima} Sem. Mat. Scuola Norm. Sup. Pisa (1961), 1-56.
\bibitem[Fed70]{Fed70} H. Federer {\it The singular sets of area minimizing rectifiable currents with codimension one and area minimizing flat chains modulo two with arbitrary codimension} Bull. Amer. Math. Soc. 76 (1970), 767-771.
\bibitem[GilTru]{GT} D. Gilbarg, N. Trudinger \textit{Elliptic Partial Differential Equations of Second Order} 3rd Edition, Springer-Verlag.
\bibitem[GMT80]{GonzMassTaman80}  E. Gonzalez, U. Massari, I. Tamanini {\it Existence and regularity for the problem of a pendent liquid drop}, Pacific J. Math. 88 (1980) {n. 2}, {399-420}.
\bibitem[GMT83]{GonzMassTaman}  E. Gonzalez, U. Massari, I. Tamanini {\it On the regularity of boundaries of sets minimizing perimeter with a volume constraint} Indiana Univ. Math. J. 32 (1983), no. 1, 25-37. 
\bibitem[HutTon00]{HutchTon} J. E. Hutchinson, Y. Tonegawa {\it Convergence of phase interfaces in the van der Waals--Cahn--Hilliard theory} Calc. Var. 10 (2000), no. 1, 49-84.
\bibitem[Mag12]{Maggi} F. Maggi {\it Sets of finite perimeter and geometric variational problems}, Cambridge Studies in Advanced Mathematics, {135}, {An introduction to geometric measure theory}, {Cambridge University Press, Cambridge}, {2012},  {xx+454}.
\bibitem[Mor03]{Morgan} F. Morgan, {\it Regularity of isoperimetric hypersurfaces in Riemannian manifolds} Trans. Amer. Math. Soc. 355 (2003), 5041-5052
\bibitem[RogTon08]{RogTon} M. R\"oger, Y. Tonegawa {\it Convergence of phase-field approximations to the Gibbs-Thomson law} Calc. Var. (2008) 32:111-136.
\bibitem[Sam69]{Samelson} H. Samelson {\it Orientability of hypersurfaces in {$R^{n}$}},  {Proc. Amer. Math. Soc.}, {22}, {1969}, {301-302}.
\bibitem[SSY75]{SSY} R. Schoen, L. Simon, S. T. Yau {\it Curvature estimates for minimal hypersurfaces} Acta Mathematica 1975, Volume 134, Issue 1, pp 275-288   
\bibitem[SchSim81]{SS} R. Schoen, L. Simon  {\it Regularity of stable minimal hypersurfaces} Comm. Pure Appl. Math. 34 (1981), 741-797
\bibitem[Sim83]{SimonNotes} L. Simon {\it Lectures on Geometric Measure Theory} Proceedings of the Centre for Mathematical Analysis 3, Canberra, (1984), VII+272.
\bibitem[JSim68]{Simons} J. Simons {\it Minimal varieties in Riemannian manifolds} Ann. of Math. 88 (1968), 62-105.
\bibitem[Ton05]{Ton} Y.\ Tonegawa, {\it A diffused  interface whose chemical potential lies in a Sobolev space}, Ann\ Sc.\ Norm.\ Super.\ Pisa Cl.\ Sci.\ (5) 4(3), 487-510 (2005). 
\bibitem[Wic14]{WicAnnals} N. Wickramasekera {\it A general regularity theory for stable codimension 1 integral varifolds} Ann. of Math. 179 (2014), 843-1007.
  
\end{thebibliography}
\end{document}